\documentclass[11pt]{article}

\input{copula-geometry-second-order-preamble}

\title{\paperdisplaytitle}
\author{Tianle Liu\\[0.35em]
  \small\textit{\authoraffiliation}\\[-0.1em]
  \small\href{mailto:\authoremail}{\texttt{\authoremail}}}
\date{\paperdate}

\begin{document}
\maketitle

\begin{abstract}
Heavy-tailed $p$-value combination tests are attractive under unknown
dependence because their null tails can be first-order robust even when the
exact dependent null distribution is unavailable.  That robustness does
not resolve the calibration problem: a statement of the form
$\Pr\{T>q(\alpha)\}=\alpha+o(\alpha)$ neither quantifies the remaining size
error nor determines whether the test is conservative or anticonservative.
We develop a second-order calibration theory for positive Half-Cauchy and
reciprocal, or harmonic-mean, aggregation.

After exact marginal standardization, extremal dependence is represented by
an index-one exponent measure on the coordinate-face lattice.  Its support
separates axial, full-interior, proper-face, and mixed geometries.  An exact
M\"obius decomposition shows how the noncompact weighted half-space probes
every face, while quantitative face limits determine whether the resulting
correction is integrable, critically amplified, or nonintegrable.  Together
with common-factor inversion and a tail-to-calibration map, this gives a
reusable machinery that applies to dependence structures more broadly than
individual parametric copula families.

The general results include an integrable hidden-face transfer theorem, a
sharp fixed-dimensional local-corner theorem with explicit
higher-face control, a common-heavy-factor theorem, and size and
critical-value expansions relative to independence.  Gaussian,
standard multivariate-$t$, positive Clayton, and max-linear pair-shock models
exhibit distinct power, logarithmic, radial--angular, and proper-face
mechanisms; the Gaussian weighted-half-space transfer is conditional on
explicit face-boundary hypotheses.  The resulting theory shows how
dependence geometry controls the rate, coefficient, and direction of the
calibration error left unresolved by first-order validity.  Asymptotics use
$t\to\infty$, equivalently vanishing significance levels.

\end{abstract}

\tableofcontents

\section{Statistical problem and organizing principle}
\label{sec:introduction}

\subsection{The calibration problem}

Let exact-uniform $p$-values be transformed into nonnegative scores with
right tail proportional to $t^{-1}$ and aggregated with simplex weights.
The two leading examples are the reciprocal score $p^{-1}$, whose weighted
sum is the inverse harmonic mean, and the positive Half-Cauchy score
$\cot(\pi p/2)$.  The former is connected to the harmonic-mean $p$-value of
Wilson~\cite{Wilson2019}; the latter is the positive Half-Cauchy
combination test (HCCT) in the ``heavily right'' program of Liu, Meng and
Pillai~\cite{LiuMengPillai2025}; and the
Cauchy-combination starting point is due to Liu and Xie~\cite{LiuXie2020}.

These methods address a persistent difficulty in combining dependent
evidence.  The joint null law of the aggregate is generally unavailable
without specifying or estimating the full dependence structure.  Heavy right
tails partly bypass that obstacle: under index-one multivariate regular
variation (MRV), linearity and exact marginal normalization make the leading
tail constant independent of the spectral measure.  This projection identity
is classical in MRV and now appears explicitly in the heavy-tailed-combination
literature
\cite{GuiMaoWangWang2026,ChakrabortyGuoSheddenStoev2026}.  It explains
first-order robustness, but it does not provide exact calibration.

Indeed, let $S_{\boldsymbol w}$ denote either the reciprocal aggregate or
the positive Half-Cauchy aggregate, and let
$q_{\Pi,S,\boldsymbol w}(\alpha)$ be its exact upper $\alpha$-quantile under
independent $p$-values.  First-order validity only asserts
\[
 \Pr_C\{S_{\boldsymbol w}>q_{\Pi,S,\boldsymbol w}(\alpha)\}
 =\alpha+o(\alpha),
 \qquad \alpha\downarrow0.
\]
It does not determine how large the error is at a small but finite level,
whether the reference is conservative or anticonservative, or how its
critical value should be corrected.  Those questions are precisely the
second-order calibration problem.  They are statistically consequential
because unavailable exact dependent calibration is the central price paid
for the robustness and simplicity of heavy-tailed aggregation.

\subsection{Relation to existing work}
\label{sec:intro-literature}

The harmonic-mean $p$-value of Wilson~\cite{Wilson2019} aggregates
reciprocal scores, while the heavily right framework of Liu, Meng and
Pillai~\cite{LiuMengPillai2025} studies positive Half-Cauchy scores and
builds on the Cauchy combination literature
\cite{LiuXie2020,LongEtAl2023}.  The two nonnegative rules share their
index-one leading geometry but differ at second order because reciprocal
scores are exactly Pareto whereas the Half-Cauchy map has an intrinsic
analytic remainder.

Recent work establishes first-order validity and angular-measure
calibration for heavy-tailed combination tests under several dependence
regimes \cite{GuiJiangWang2025,GuiMaoWangWang2026,
ChakrabortyGuoSheddenStoev2026}; finite-level harmonic-mean inequalities
provide a complementary perspective \cite{ChenWangWangZhu2026}.  The
classical MRV, hidden-regular-variation, and coordinate-face foundations
are not new \cite{BasrakDavisMikosch2002,HultLindskog2002,
LedfordTawn1996,Resnick2002,Resnick2008,BalkemaEmbrechts2007}.

Second-order aggregation of heavy-tailed risks is also established theory
\cite{GelukEtAl1997,BarbeMcCormick2005,Kortschak2012,DasKratz2020}.
In particular, the unit-index logarithm and related integrability boundary
appear in existing bivariate copula-sum expansions
\cite{YangZhang2023}.  Gaussian hidden tails, Archimedean frailty,
standard-$t$ tail dependence, and max-linear spectral atoms likewise have
substantial prior literatures.  Accordingly, the contribution claimed here
is not those ingredients separately.  It is the common
independence-relative calibration problem for harmonic-mean and positive
Half-Cauchy aggregation, together with a face-lattice calculus that tracks
the rate, coefficient, sign, and critical-value consequence across support
regimes.  An extended literature review and a claim-by-claim novelty audit
are in \LiteratureAuditLocation.

\subsection{Three contributions}

The paper develops one chain of ideas:
\begin{equation}\label{eq:intro-conceptual-chain}
 \begin{aligned}
  \text{dependence geometry}
  &\ \Longrightarrow\ 
  \text{face-specific tail correction}\\
  &\ \Longrightarrow\ 
  \text{aggregate-tail expansion}
  \ \Longrightarrow\ 
  \text{size and critical-value distortion}.
 \end{aligned}
\end{equation}
Its contributions are best understood in three parts.

\begin{enumerate}[leftmargin=2.2em]
 \item \emph{Second-order calibration as statistical theory.}
 We take the exact independence law as the operational reference and derive
 the leading dependence-induced error in its tail probability, rejection
 probability, and critical value.  The output is a rate, coefficient, and
 sign, rather than an unquantified $o(\alpha)$ guarantee.

 \item \emph{A reusable second-order machinery.}
 Exact Pareto standardization, quantitative exponent measures, an exact
 M\"obius decomposition of the weighted half-space, hidden-face transfer,
 local-corner analysis, common-factor inversion, and calibration inversion
 form a common calculus.  This machinery separates perturbations of strata
 already charged at first order from hidden mass on new faces, critical
 accumulation near subfaces, nonintegrable boundary layers, and
 score-transformation effects.

 \item \emph{A geometric organization of dependence.}
 The support of the exponent measure on the coordinate-face lattice gives
 four primary geometries: axes, full interior, nontrivial proper faces, and
 mixtures of these strata.  More importantly, the way probability approaches
 charged and uncharged faces predicts the possible second-order mechanism.
 The resulting organization groups copulas and more general dependence
 structures by their extremal geometry rather than merely by family name.
\end{enumerate}

The general theory yields a sharp local-corner theorem in arbitrary fixed
dimension, with unequal weights and explicit control of higher faces; an
integrable hidden-face transfer theorem; a common-heavy-factor theorem; and a
calculus converting tail coefficients into size and critical-value errors.
The model results are designed to expose distinct parts of the machinery.
Independence is the critical axial benchmark; Gaussian copulas exhibit hidden
pair and boundary phenomena; standard multivariate-$t$ and positive Clayton
copulas exhibit radial--angular interior or mixed behavior; and a max-linear
pair-shock construction gives an exact proper-face example.

The geometry is summarized at three resolutions.
\cref{tab:second-order-mechanisms} gives the four principal measure-transfer
mechanisms, \cref{tab:principal-model-rates} compares the Gaussian,
independence, standard-$t$, and positive Clayton benchmarks, and the full
ledger in \FullLedgerPlacement{} records model-specific support,
second-order rate, mechanism, proof status, and bounded-density diagnostics.
The detailed ledger and its supporting calculations are placed there so that the main development can
emphasize the general geometry.

\subsection{Paper organization}
\label{sec:paper-organization}

The argument links testing, copulas, multivariate extremes, and heavy-tail
aggregation.  We review the needed MRV ingredients, emphasizing exponent
measures, the index-one projection identity, and spectral support on the
coordinate-face lattice.

\cref{sec:mrv-geometry} develops the score transformations and the
exponent-measure geometry.  The lower-tail copula formulation is treated in
the following section.  We then introduce the calibration functional and
the face calculus before turning to the four geometric regimes.  The
quantitative assumption hierarchy and the generic common-factor inversion
theorem are stated in \TechnicalLocation, while their consequences needed
for the main argument are summarized at the relevant points.
Each model analysis follows the same route: identify the
first-order support, locate the finite-threshold mass that can move the
weighted sum across its boundary, determine the corresponding rate and
coefficient, and translate the result into calibration error for both
reciprocal and positive Half-Cauchy aggregation.  Statements in theorem,
proposition, lemma, and corollary environments are claimed under their
displayed hypotheses; conjectures, open problems, and classifications
labelled \emph{expected} are not claimed as proved.
All formal proofs are collected in \ProofLocation.

\section{Heavily right scores and exponent-measure geometry}
\label{sec:mrv-geometry}

\subsection{Harmonic-mean and positive Half-Cauchy scores}
\label{sec:score-equivalence}

Let $m\geq1$ be fixed and let $p_1,\ldots,p_m$ have exact
$\operatorname{Uniform}(0,1)$ null margins, with dependence unrestricted
until stated otherwise.  Put $U_i=p_i$, $Y_i=1/U_i$,
$H_i=\cot(\pi U_i/2)$, and $c=2/\pi$.  For simplex weights, define
\begin{equation}\label{eq:HCCT}
  R_{\boldsymbol w}
   =\sum_{i=1}^m w_iY_i
   =\sum_{i=1}^m\frac{w_i}{p_i},
  \qquad
  T_{\boldsymbol w}
   =\sum_{i=1}^m w_iH_i=\HCCT,
  \qquad
  w_i\geq0,
  \qquad
  \sum_{i=1}^m w_i=1.
\end{equation}
Here $R_{\boldsymbol w}$ is the inverse weighted harmonic mean and
$T_{\boldsymbol w}$ is the positive Half-Cauchy statistic.  Zero-weight
coordinates may be deleted.  Exactly,
$\Pp(Y_i>y)=y^{-1}$ for $y\geq1$, whereas
$\Pp(H_i>x)=2\arctan(x^{-1})/\pi=c/x+O(x^{-3})$.
We write $\FRw$ and $\FHCw$ for the corresponding CDFs under independent
exact-uniform $p_i$ and fixed weights.

\begin{proposition}[Pareto--Half-Cauchy equivalence]
\label{prop:equiv}\label{lem:bounded}
For $h(u)=\cot(\pi u/2)$ and $d(u)=c/u-h(u)$, the function $d$ is
increasing from $0$ to $c$ on $[0,1]$.  Hence
\begin{equation}\label{eq:bounded-difference}
 cY_i-c\leq H_i\leq cY_i,
 \qquad \|\boldsymbol H-c\boldsymbol Y\|_\infty\leq c.
\end{equation}
For fixed $m$, $\boldsymbol Y\in\MRV_1(\mu_Y)$ if and only if
$\boldsymbol H\in\MRV_1(\mu_H)$, with
$\mu_H(A)=\mu_Y(c^{-1}A)$ and
$\mu_H\{\boldsymbol x:x_i>1\}=c$ for every $i$.
\end{proposition}

Thus the right object to inspect in the original $p$-value copula is the
Pareto vector $(1/p_1,\ldots,1/p_m)$.  The two score vectors therefore have
the same first-order geometry up to scale, but their second-order terms can
differ because only the reciprocal margins are exactly Pareto.  Indeed,
$|T_{\bw}-cR_{\bw}|\leq c$.  Under the mild local tail regularity needed to
control this bounded threshold displacement,
\begin{equation}\label{eq:ehmp-hcct-coefficient-transfer}
 \begin{aligned}
 &\Pp(R_{\bw}>s)=s^{-1}+Bs^{-\beta}L(s)+o\{s^{-\beta}L(s)\},
 \qquad s^{-\beta}L(s)\gg s^{-2},\\[-2pt]
 &\hspace{2em}\Longrightarrow\quad
 \Pp(T_{\bw}>t)=\frac ct+Bc^\beta t^{-\beta}L(t/c)
 +o\{t^{-\beta}L(t)\}.
 \end{aligned}
\end{equation}
This covers the positive-correlation Gaussian term, the independent
$(\log t)/t^2$ term, and the multivariate-$t$ radial term for $\nu>2$.
At order $t^{-2}$ or smaller the bounded displacement can change the
coefficient, so the two aggregates must be expanded separately.

\subsection{Multivariate regular variation and its polar representation}

Put
\[
  \mathbb E_m=[0,\infty]^m\setminus\{\boldsymbol0\}.
\]
A set is bounded away from the origin if it does not intersect some
neighborhood of $\boldsymbol0$.

\begin{definition}[Index-one MRV]\label{def:mrv}
A nonnegative random vector $\boldsymbol X=(X_1,\ldots,X_m)$ is
multivariate regularly varying with index one, written
$\boldsymbol X\in\MRV_1(\mu)$, if there exists a nonzero Radon measure
$\mu$ on $\mathbb E_m$ such that
\begin{equation}\label{eq:mrv-vague}
  t\,\Pp\!\left(\frac{\boldsymbol X}{t}\in A\right)
  \longrightarrow \mu(A)
\end{equation}
for every Borel set $A\subset\mathbb E_m$ that is bounded away from the
origin and satisfies $\mu(\partial A)=0$.
\end{definition}

The measure $\mu$ is called the exponent or tail measure.  It is
homogeneous of order $-1$:
\begin{equation}\label{eq:homogeneity}
  \mu(aA)=a^{-1}\mu(A),\qquad a>0.
\end{equation}
Using $L^1$ polar coordinates
\[
  r=\|\boldsymbol x\|_1,
  \qquad
  \boldsymbol s=\frac{\boldsymbol x}{\|\boldsymbol x\|_1},
  \qquad
  \mathbb S^{m-1}_+=\{\boldsymbol s\geq0:\|\boldsymbol s\|_1=1\},
\]
the tail measure has the representation
\begin{equation}\label{eq:spectral}
  \mu(d r,d\boldsymbol s)=r^{-2}\,d r\,S(d\boldsymbol s),
\end{equation}
where $S$ is a finite measure on the positive simplex.  The location of
$S$ describes extremal dependence:
\begin{itemize}[leftmargin=2em]
  \item mass only on the coordinate axes means asymptotic tail
  independence;
  \item mass away from the axes means that simultaneous extremes have
  first-order probability;
  \item mass on the diagonal is the limiting geometry of perfect positive
  tail dependence.
\end{itemize}

MRV is therefore not synonymous with tail independence.  Both Gaussian
copulas with correlations strictly below one and multivariate-$t$ copulas
are covered, but their spectral measures have different supports.

\paragraph*{Exponent-measure interpretation}
The exponent measure is the intensity measure of the limiting Poisson
cloud of normalized extremes.  Its traditional name comes from the
max-stable representation: exponentiating the negative of the relevant
exceedance mass gives the limiting distribution.  It is unrelated to a
Lie-theoretic exponential.  The full point-process and
max-stability derivation is given in \TechnicalLocation.  What matters for
aggregation is that the measure records both the radial frequency of large
observations and the angular locations at which they occur.

\subsection{Spectral measures and coordinate-face geometry}
\label{sec:exponent-support-geometry}

For unit-Pareto margins, the spectral measure obeys
\begin{equation}\label{eq:spectral-moment-constraints}
 \int_{\mathbb S_+^{m-1}}s_i\,S(d\boldsymbol s)=1,
 \qquad i=1,\ldots,m.
\end{equation}
Consequently $S(\mathbb S_+^{m-1})=m$.  Conversely, any finite positive
measure on the simplex satisfying \cref{eq:spectral-moment-constraints}
defines an index-one exponent measure with unit-Pareto margins.  Radially
the measure is rigid, because \cref{eq:spectral} puts the same
$r^{-2}dr$ intensity along every charged direction.  Angularly it is
flexible: $S$ may be atomic, diffuse, singular, supported on lower faces,
or any mixture of these, subject only to the marginal moment constraints.

For nonempty $I\subseteq[m]$, define the relative interior of the
coordinate face
\begin{equation}\label{eq:coordinate-open-face}
 F_I^\circ
 =\{\boldsymbol s\in\mathbb S_+^{m-1}:
       s_i>0\ (i\in I),\ s_j=0\ (j\notin I)\}.
\end{equation}
It is useful to record the strata carrying positive spectral mass:
\begin{equation}\label{eq:charged-strata}
 \mathcal C(S)=\{I\subseteq[m]:I\ne\varnothing,
                         \ S(F_I^\circ)>0\}.
\end{equation}
This measure-theoretic definition avoids confusing mass on a relative-open
stratum with topological support in its closure.  It gives four mutually
exclusive first-order geometries:
\begin{enumerate}[leftmargin=2.2em]
 \item \emph{axial}: every $I\in\mathcal C(S)$ has $|I|=1$;
 \item \emph{interior}: $\mathcal C(S)=\{[m]\}$, equivalently $S$ assigns
 the simplex boundary zero mass;
 \item \emph{proper-face}: every $I\in\mathcal C(S)$ satisfies
 $2\leq |I|\leq m-1$;
 \item \emph{mixed}: more than one of the preceding three stratum types
 carries positive mass.
\end{enumerate}
The shorthand classifies the exponent measure associated with a copula, not
the support of the copula distribution itself.  Proper faces do not exist
when $m=2$.  A ray is an atom of $S$ at a direction and is not a fifth
category: it is axial, proper-face, or interior according to the relative-open
stratum containing that direction.  With equal unit-Pareto marginal
constants, a spectral measure supported on a single ray must use the full
diagonal; several asymmetric rays can balance one another.

\begin{figure}[t]
 \centering
 \includegraphics[width=\GeometryFigureWidth]{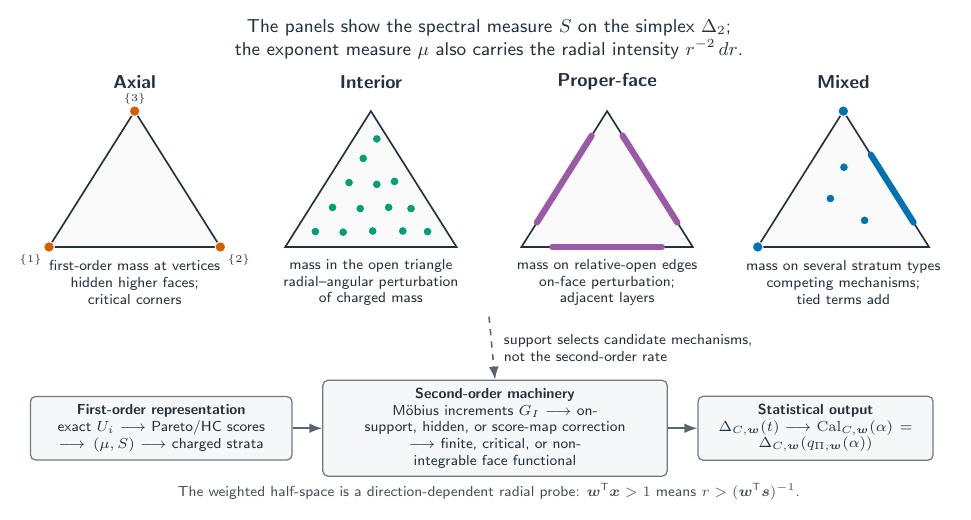}
 \caption{Exponent-measure geometry and the route to calibration in
 dimension three.  The upper panels depict positive mass of the angular
 spectral measure $S$ on relative-open simplex strata; they do not depict
 the radial exponent measure $\mu$ itself.  An interior measure may approach
 the boundary in topological support while assigning it zero mass, and atoms
 or rays may occur within any charged stratum.  The lower strip shows the
 paper's organizing chain.  First-order support selects candidate
 second-order mechanisms, whereas quantitative convergence and integration
 against the signed weighted-half-space increments determine the rate and
 coefficient.  The final target is error relative to exact independence
 calibration, rather than the raw tail's second term.}
 \label{fig:exponent-measure-geometry}
\end{figure}

The four panels of \cref{fig:exponent-measure-geometry} describe where
first-order extremes occur.  They do not determine how fast finite-threshold
probability approaches those strata.  In particular, independence, Gaussian
copulas with correlations below one, exclusion models, and rate-free
constructions may all be axial while having fundamentally different
second-order behavior.  That separation between support and approach rate is
the reason geometry and quantitative face analysis enter as distinct parts of
the theory.

\subsection{Tail index one and universal first-order aggregation}
\label{sec:mrv-index}

For a general tail index $\alpha>0$, MRV is formulated with a scaling
$b(t)\in\operatorname{RV}_{1/\alpha}$:
\begin{equation}\label{eq:general-alpha-mrv}
 t\Pp\{\boldsymbol X/b(t)\in\cdot\}
 \xrightarrow{v}\mu(\cdot),
 \qquad
 \frac{b(tx)}{b(t)}\longrightarrow x^{1/\alpha}.
\end{equation}
The exponent measure is homogeneous: $\mu(aA)=a^{-\alpha}\mu(A)$.
\par\noindent
Its polar form is
$\mu(dr,d\boldsymbol s)=\alpha r^{-\alpha-1}\,dr\,S(d\boldsymbol s)$.
A Pareto-$\alpha$ tail $\Pp(X>x)=x^{-\alpha}$ has maxima of order
$n^{1/\alpha}$.  Smaller $\alpha$ means a heavier tail.

The reciprocal margins have index one exactly, while
\cref{prop:equiv} transfers the same index to the positive Half-Cauchy
scores.  This statement concerns the transformed $p$-value scores: a raw
standard multivariate-$t_\nu$ vector has index $\nu$, and a raw Gaussian
vector is not regularly varying, even though their exact-$p$ reciprocal
scores have index one.

Marginal Pareto standardization can convert continuous margins to index one,
but it changes aggregation on the original scale.  Index one is especially
consequential here.  For a general index $\alpha$, let
$A_{\boldsymbol w}=\{\boldsymbol x\in\mathbb E_m:
\boldsymbol w^{\mathsf T}\boldsymbol x>1\}$.  Radial integration gives
\begin{equation}\label{eq:general-alpha-projection}
 \mu(A_{\boldsymbol w})
 =\int(\boldsymbol w^{\mathsf T}\boldsymbol s)^\alpha
      S(d\boldsymbol s),
\end{equation}
where the marginal tail constants are
$c_i=\int s_i^\alpha S(d\boldsymbol s)$.  Only at $\alpha=1$ does the
projection linearize:
\begin{equation}\label{eq:linear-tail}
 \mu(A_{\boldsymbol w})
 =\int\boldsymbol w^{\mathsf T}\boldsymbol s\,S(d\boldsymbol s)
 =\sum_{i=1}^m w_i c_i.
\end{equation}
This is why the leading constants below depend on the margins and weights,
but not on whether the spectral measure is axial, interior, proper-face, or
mixed.

\begin{theorem}[Universal first-order harmonic-mean and Half-Cauchy tails]
\label{thm:hcct-mrv}
Let $m$ be fixed.  Suppose that the $p_i$ have exact uniform margins, that
$\boldsymbol w$ is a simplex weight vector, and that
$\boldsymbol Y=(1/p_1,\ldots,1/p_m)$ is multivariate regularly varying with
index one.  Equivalently, by \cref{prop:equiv}, $\boldsymbol H$ is
multivariate regularly varying with index one.  Then
\begin{align}
 \Pp(R_{\boldsymbol w}>t)
 &=\frac{1}{t}+o(t^{-1}),
 \label{eq:hmp}\\
 \Pp(T_{\boldsymbol w}>t)
 &=\frac{2}{\pi t}+o(t^{-1}).
 \label{eq:hcct-mrv-tail}
\end{align}
The same leading terms hold under independent exact-uniform $p$-values;
therefore
\begin{equation}\label{eq:hcct-ratios}
 \frac{\Pp(R_{\boldsymbol w}>t)}{1-\FRw(t)}\longrightarrow1,
 \qquad
 \frac{\Pp(T_{\boldsymbol w}>t)}{1-\FHCw(t)}\longrightarrow1.
\end{equation}
\end{theorem}

The corresponding first-order thresholds are
$t_{R,\alpha}=\textstyle\frac1\alpha$ and
$t_{T,\alpha}=\textstyle\frac2{\pi\alpha}$.  Thus
\cref{thm:hcct-mrv} gives
\begin{equation}\label{eq:alpha-size}
 \Pp(R_{\boldsymbol w}>t_{R,\alpha})
 =\alpha+o(\alpha),
 \qquad
 \Pp(T_{\boldsymbol w}>t_{T,\alpha})
 =\alpha+o(\alpha).
\end{equation}
For the Half-Cauchy statistic one may instead use the marginal quantile
$\cot(\pi\alpha/2)$, which is asymptotic to $2/(\pi\alpha)$.  These are
$\alpha\downarrow0$ statements, not claims of exact calibration at a fixed
positive level.  The common first-order limit also does not say that the
scores are independent; it uses nonnegativity, index one, exact marginal
normalization, and simplex weights.

\section{Lower max-domains for p-value copulas}
\label{sec:lower-mda-copulas}

\subsection{Lower extremes and stable-tail dependence}
\label{sec:lower-mda-foundations}

\begin{definition}[Lower max-domain of attraction]
\label{def:lstd}
The copula of $\boldsymbol U$ is in the lower max-domain of attraction if,
for every $\boldsymbol x=(x_1,\ldots,x_m)\in[0,\infty)^m$, the following limit exists:
\begin{equation}\label{eq:ell}
  \ell_L(\boldsymbol x)
  =\lim_{s\downarrow0}\frac{1}{s}
    \Pp\!\left(\bigcup_{i=1}^m\{U_i\leq sx_i\}\right).
\end{equation}
\end{definition}

Exact uniform margins give, for every $s>0$ small enough,
\begin{equation}\label{eq:lower-union-frechet}
  s\max_i x_i
  \leq
  \Pp\!\left(\bigcup_i\{U_i\leq sx_i\}\right)
  \leq
  s\sum_i x_i.
\end{equation}
Thus the union probability is automatically of order $s$; lower MDA asks
whether its coefficient settles to a limit.  The resulting function
$\ell_L$ is convex and homogeneous of degree one and satisfies
\begin{equation}\label{eq:ell-bounds}
  \max_i x_i\leq\ell_L(\boldsymbol x)\leq\sum_{i=1}^m x_i.
\end{equation}

For $m=2$, define the lower-tail copula, when the limit exists, by
\[
  \Lambda_L(x,y)=\lim_{s\downarrow0}\frac{C(sx,sy)}{s}.
\]
Inclusion--exclusion gives
\begin{equation}\label{eq:ell-biv}
  \ell_L(x,y)=x+y-\Lambda_L(x,y).
\end{equation}

For the reciprocal vector $\boldsymbol Y=1/\boldsymbol U$,
\cref{thm:copula-mrv} turns lower MDA into index-one MRV.  If $S_Y$
is its spectral measure, the lower stable-tail function, exponent function,
and spectral representation are the same object in reciprocal coordinates:
\begin{equation}\label{eq:lower-stdf-exponent-identification}
 \ell_L(\boldsymbol x)
 =V_Y(1/x_1,\ldots,1/x_m)
 =\int_{\mathbb S_+^{m-1}}
    \max_i(x_is_i)\,S_Y(d\boldsymbol s),
\end{equation}
with the usual convention for zero coordinates.

In two dimensions, write $\boldsymbol s=(r,1-r)$.  Then $S_Y/2$ is a
probability distribution on $[0,1]$ with mean $1/2$.  Independence has
$S_Y=\delta_0+\delta_1$, perfect dependence has
$S_Y=2\delta_{1/2}$, and
\begin{equation}\label{eq:axis-diagonal-spectral-mixture}
 S_Y=(1-\lambda)(\delta_0+\delta_1)+2\lambda\delta_{1/2}
\end{equation}
interpolates between them.  The lower-tail copula has the spectral form
\begin{equation}\label{eq:joint-tail-spectral-mass}
 \Lambda_L(x,y)
 =\int_0^1\min\{xr,y(1-r)\}\,S_Y(dr).
\end{equation}
Consequently, positive interior spectral mass is equivalent to
$\Lambda_L(1,1)>0$.  In that case,
\begin{equation}\label{eq:interior-mass-joint-tail}
 \Pp(Y_1>t,Y_2>t)\sim\frac{\Lambda_L(1,1)}{t}.
\end{equation}
Axis support makes this constant zero and gives only $o(t^{-1})$; it does
not determine the smaller rate.

The ratio $C(sx,sy)/(sx)$ converges to $\Lambda_L(x,y)/x$.
Thus zero means simultaneous order-$s$ events are negligible relative to
a single small $p$-value, while a positive limit means that joint extremes
already occur at first order.
The union/intersection identities and their order-statistic interpretation
are given in \TechnicalLocation.

For $m>2$, pairwise lower-tail limits do not generally determine the
full multivariate lower-MDA limit, because higher-order intersections may
contain additional information.  A useful sufficient condition for
axis-supported convergence is stated in \cref{prop:pairwise}.

For independent replicates, normalized componentwise minima of
$\boldsymbol U$ converge to a min-stable law with exponent $\ell_L$;
equivalently, normalized componentwise maxima of
$\boldsymbol Y=1/\boldsymbol U$ converge to the reciprocal max-stable law.
This is the origin of the phrase lower max-domain of attraction.  The full
replicate calculation is in \TechnicalLocation.

Hidden regular variation, residual tail orders, and higher-order tail
densities describe the smaller probability of simultaneous extremes when
the first-order exponent measure is axial
\cite{LedfordTawn1996,LedfordTawn1997,Resnick2002,Resnick2008,
HuaJoeLi2014,LiHua2015}.  For Gaussian copulas, reciprocal and positive
Half-Cauchy scores have axial first-order MRV, while smaller-order joint
tails set the calibration rate.

Lower MDA need not imply tail independence.  Its universal first-order
aggregation consequence is stated in \cref{cor:lower-mda-validity}, while
the geometry remains decisive at second order.

\subsection{From lower MDA to MRV and universal aggregation}

\begin{theorem}[Copula criterion for index-one MRV]\label{thm:copula-mrv}
If the limit in \cref{eq:ell} exists for every $\boldsymbol x\geq0$, then
$\boldsymbol Y=(1/U_1,\ldots,1/U_m)$ is multivariate regularly varying with
index one.  Consequently, $\boldsymbol H$ is multivariate regularly varying
with index one.
\end{theorem}

\begin{remark}[How close this condition is to necessary]
For a fixed finite dimension and exact uniform margins, existence of the
limits in \cref{eq:ell} is one of the standard equivalent formulations of
index-one MRV of $(1/U_1,\ldots,1/U_m)$.  It is therefore not merely an
ad hoc sufficient condition; it is the natural copula formulation of the
desired tail property.
\end{remark}

\begin{proposition}[An easier condition implying axis-supported MRV]
\label{prop:pairwise}
Suppose that, for every $i\ne j$ and every fixed $x,y>0$,
\begin{equation}\label{eq:pair-comparable}
  \Pp(U_i\leq sx,U_j\leq sy)=o(s),
  \qquad s\downarrow0.
\end{equation}
Then
\begin{equation}\label{eq:ell-sum}
  \ell_L(\boldsymbol x)=\sum_{i=1}^m x_i,
\end{equation}
so $\boldsymbol Y$ and $\boldsymbol H$ are MRV with spectral mass only on
the coordinate axes.
\end{proposition}

Condition \cref{eq:pair-comparable} is convenient but substantially
stronger than general MRV.  It excludes any copula with nonzero lower tail
dependence.

The one-big-jump analysis of Liu, Meng and
Pillai~\cite{LiuMengPillai2025} uses a stronger asymmetric pairwise
condition.  In the lower--lower form relevant to the two nonnegative score
transformations studied here, there is a sequence $\delta_t\downarrow0$
with $t\delta_t\to\infty$ such that, for every ordered pair of active
coordinates $i\ne j$,
\begin{equation}\label{eq:earlier-pairwise-condition}
 \Pp\!\left(
   U_i<\frac{a_i}{t},
   U_j<\frac{a_j}{\delta_t t}
 \right)=o(t^{-1}),
 \qquad
 a_k=\frac{2mw_k}{\pi}>0.
\end{equation}
This condition implies \cref{eq:pair-comparable}.  Indeed, for fixed
$x,y>0$, put $T=a_i/(sx)$.  The first threshold in
\cref{eq:earlier-pairwise-condition} then equals $sx$, whereas the second
equals
\[
 s\frac{a_jx}{a_i\delta_T},
\]
which eventually exceeds $sy$ because $\delta_T\to0$.  Therefore
\[
 \Pp(U_i\leq sx,U_j\leq sy)
 \leq
 \Pp\!\left(
   U_i<\frac{a_i}{T},
   U_j<\frac{a_j}{\delta_TT}
 \right)
 =o(T^{-1})=o(s).
\]
Thus the earlier condition is a sufficient route to the axis-supported
case of \cref{prop:pairwise}; it is not necessary for universal
first-order aggregation under lower MDA.

For a decisive contrast, if $U_1=\cdots=U_m=U$, simplex weighting gives
\[
 R_{\boldsymbol w}=U^{-1},
 \qquad
 T_{\boldsymbol w}=\cot(\pi U/2).
\]
The pairwise condition fails maximally, but both aggregates reduce exactly
to their one-dimensional reference scores.  General lower MDA therefore
covers first-order extremal dependence that the earlier one-big-jump
condition deliberately excludes.

\begin{corollary}[Universal first-order tails under lower MDA]
\label{cor:lower-mda-validity}
Let $m$ be fixed, let the $p_i$ have exact uniform margins, and let
$\boldsymbol w$ be a simplex weight vector.  If the copula of
$\boldsymbol U$ is in the lower max-domain of attraction in the sense of
\cref{def:lstd}, then
\begin{align}
 \Pp(R_{\boldsymbol w}>t)
 &=\frac{1}{t}+o(t^{-1}),\label{eq:lower-mda-reciprocal-tail}\\
 \Pp(T_{\boldsymbol w}>t)
 &=\frac{2}{\pi t}+o(t^{-1}).\label{eq:lower-mda-hc-tail}
\end{align}
Each tail is asymptotically equivalent to its exact independence counterpart.
\end{corollary}

\paragraph*{The lower-MDA assumption is substantive}
Exact uniform margins alone do not force the limit in \cref{eq:ell} to
exist.  \TechnicalLocation{} gives a scale-oscillatory copula diagonal for
which the normalized lower-tail overlap has two subsequential limits.  The
example separates failure of a full tail measure from failure of one
particular aggregate-tail limit.

\section{Second-order calibration and quantitative tail control}
\label{sec:second-order}

\subsection{From tail expansions to second-order calibration}
\label{sec:calibration-functionals}

The word \emph{calibration} will refer to comparison with the exact null
law under independent $p$-values, not merely comparison with the leading
index-one tail $c_0/t$.  This distinction matters because the independent
law itself has a critical logarithmic correction for both aggregates.

Let $S_{C,\bw}$ be a fixed weighted aggregation statistic under copula
$C$, and write
\[
 \overline F_{C,\bw}(t)=\Pp_C(S_{C,\bw}>t).
\]
Let $\Pi$ denote the independence copula and define
\begin{align}
 D_{C,\bw}(t)
  &=\overline F_{C,\bw}(t)-\overline F_{\Pi,\bw}(t),
  \label{eq:tail-calibration-distortion}\\
 R_{C,\bw}(t)
  &=\frac{\overline F_{C,\bw}(t)}
          {\overline F_{\Pi,\bw}(t)}-1.
  \label{eq:relative-tail-calibration-distortion}
\end{align}
For a nominal level $\alpha$, let $q_{\Pi,\bw}(\alpha)$ be the upper
$\alpha$-quantile of the independence law.  Whenever
$\overline F_{\Pi,\bw}\{q_{\Pi,\bw}(\alpha)\}=\alpha$, the actual size
distortion is
\begin{equation}\label{eq:size-calibration-functional}
 \operatorname{Cal}_{C,\bw}(\alpha)
 =\Pp_C\{S_{C,\bw}>q_{\Pi,\bw}(\alpha)\}-\alpha
 =D_{C,\bw}\{q_{\Pi,\bw}(\alpha)\}.
\end{equation}
Positive and negative signs indicate anticonservative and conservative
independence calibration, respectively.

\begin{theorem}[Second-order calibration and quantile inversion]
\label{thm:calibration-inversion}
Let $j\in\{C,\Pi\}$, let $c_0>0$, and suppose that the two survival
functions are eventually continuous and strictly decreasing and satisfy
\begin{equation}\label{eq:common-second-order-tail-expansion}
 \overline F_{j,\bw}(t)
 =\frac{c_0}{t}
  +d_{j,\bw}\,t^{-1-\kappa}L(t)
  +o\{t^{-1-\kappa}L(t)\},
 \qquad \kappa>0,
\end{equation}
where $L$ is positive and slowly varying.  Put $q_0(\alpha)=c_0/\alpha$.
Then
\begin{align}
 q_{j,\bw}(\alpha)
 &=q_0(\alpha)
   +\frac{d_{j,\bw}}{c_0}
      q_0(\alpha)^{1-\kappa}L\{q_0(\alpha)\}
   +o\!\left(q_0(\alpha)^{1-\kappa}
             L\{q_0(\alpha)\}\right),
 \label{eq:second-order-quantile-expansion}\\
 \operatorname{Cal}_{C,\bw}(\alpha)
 &=(d_{C,\bw}-d_{\Pi,\bw})
   c_0^{-1-\kappa}\alpha^{1+\kappa}
   L(c_0/\alpha)
   +o\{\alpha^{1+\kappa}L(1/\alpha)\},
 \label{eq:second-order-size-distortion}\\
 \frac{\operatorname{Cal}_{C,\bw}(\alpha)}{\alpha}
 &=(d_{C,\bw}-d_{\Pi,\bw})
   c_0^{-1-\kappa}\alpha^{\kappa}
   L(c_0/\alpha)
   +o\{\alpha^{\kappa}L(1/\alpha)\}.
 \label{eq:second-order-relative-size-distortion}
\end{align}
Moreover, the dependence-specific correction to the independence critical
value is
\begin{equation}\label{eq:critical-value-calibration-shift}
 q_{C,\bw}(\alpha)-q_{\Pi,\bw}(\alpha)
 =\frac{d_{C,\bw}-d_{\Pi,\bw}}{c_0}
   q_0(\alpha)^{1-\kappa}L\{q_0(\alpha)\}
 +o\!\left(q_0(\alpha)^{1-\kappa}L\{q_0(\alpha)\}\right).
\end{equation}
\end{theorem}

\begin{remark}[Different competing rates]
If the dependent and independent corrections have different orders, first
subtract the two survival expansions and retain the slower term.  In
particular, if
\[
 D_{C,\bw}(t)=d_{\Delta,\bw}t^{-1-\kappa}L(t)
  +o\{t^{-1-\kappa}L(t)\},
\]
then \cref{eq:size-calibration-functional} and
$q_{\Pi,\bw}(\alpha)\sim c_0/\alpha$ directly give
\[
 \operatorname{Cal}_{C,\bw}(\alpha)
 =d_{\Delta,\bw}c_0^{-1-\kappa}\alpha^{1+\kappa}
  L(c_0/\alpha)
  +o\{\alpha^{1+\kappa}L(1/\alpha)\}.
\]
This is the form used when, for example, under the face-transfer hypotheses
a positive-correlation Gaussian power correction dominates the independent
logarithm.
\end{remark}

\paragraph*{A hierarchy of quantitative assumptions}
Ordinary lower MDA supplies no rate.  A second-order tail-measure bound on
the weighted half-space does, while shell-probability, score-density, and
bounded-copula-density assumptions are progressively stronger sufficient
conditions near independence.  The precise one-way implications and the
fact that none is necessary are recorded in \TechnicalLocation.

\subsection{Critical examples}
\label{sec:critical-examples}

Two boundary examples explain why a numerical second-order rate is extra
structure and why even the product copula sits at a nontrivial critical
boundary.

Ordinary MRV can converge more slowly than every positive power.
An exact-uniform axial example with a nonregular, arbitrarily slow
diagonal remainder is constructed in \TechnicalLocation; hence no
second-order rate follows from first-order support alone.

\begin{theorem}[Product copula and the critical logarithm]
\label{thm:independent-sharp}
Let $Y_1,\ldots,Y_m$ be independent standard Pareto variables,
$\Pp(Y_i>y)=y^{-1}$ for $y\geq1$, and let $\boldsymbol w$ be a fixed
simplex weight vector.  Then
\begin{equation}\label{eq:independent-pareto-m}
 \Pp\!\left(\sum_{i=1}^mw_iY_i>t\right)
 =\frac1t+2\sum_{i<j}w_iw_j\frac{\log t}{t^2}+O(t^{-2}).
\end{equation}
For $m=2$ and $w_1,w_2>0$, the exact identity, valid for
$t>w_1+w_2$, is
\begin{equation}\label{eq:independent-pareto-exact}
 \Pp(w_1Y_1+w_2Y_2>t)
 =\frac{w_1+w_2}{t}
  +\frac{w_1w_2}{t^2}
   \log\!\frac{(t-w_1)(t-w_2)}{w_1w_2}.
\end{equation}
At a zero-weight endpoint, this reduces to the corresponding
one-dimensional Pareto identity.
If $H_i$ are independent standard positive Half-Cauchy variables and
$c=2/\pi$, then
\begin{equation}\label{eq:independent-hc-m}
 \Pp\!\left(\sum_{i=1}^mw_iH_i>t\right)
 =\frac ct+2c^2\sum_{i<j}w_iw_j\frac{\log t}{t^2}+O(t^{-2}).
\end{equation}
Thus the logarithmic rate cannot generally be improved to $O(t^{-2})$.
\end{theorem}

The product example reveals the main geometric subtlety.  Rectangle
probabilities have the exact scale
$C_\Pi(sx,sy)=s^2xy$, yet the noncompact weighted half-space reaches both
axes.  The hidden product measure is only critically integrable there, and
the Pareto-proxy support cutoff at distance $1/t$ creates the additional
$\log t$.  Rectangle rates alone therefore do not determine a sum-tail
rate without boundary control.

\paragraph*{Concrete near-independence criteria}
Joint shell bounds imply the aggregate rate
$c_X/t+O\{(\log t)/t^2\}$ for both reciprocal and positive Half-Cauchy
scores.  Product-type pair densities imply those shell bounds, and bounded
lower-tail copula densities are a still stronger sufficient condition.
Formal statements and higher-dimensional qualifications appear in
\TechnicalLocation.

\paragraph*{Quantitative half-space projection}
If the exponent-measure approximation on
$A_{\boldsymbol w}=\{\boldsymbol x:\boldsymbol w^{\mathsf T}\boldsymbol x>1\}$
has relative error $O\{a(t)\}$, the aggregate has absolute tail error
$O\{a(t)/t\}$.  This elementary projection is useful only when the
noncompact half-space is an admissible continuity set; rectangle rates
alone do not provide the required boundary control.  Formal details are in
\TechnicalLocation.

\section{Face calculus for higher-order aggregation}
\label{sec:face-calculus}

The first-order exponent measure may live on axes, on higher-dimensional
faces, or in the full orthant.  A second-order sum expansion must therefore
identify which face first changes the weighted half-space probability.  The
following exact decomposition is the bookkeeping device used below.
M\"obius inversion for multivariate excess and face measures is part of the
existing geometric extreme-value toolkit; see Balkema and
Embrechts~\cite{BalkemaEmbrechts2007}.  The particular object needed here
is the M\"obius increment of a weighted half-space.  For a nonempty set
$I\subseteq[m]$, put
\[
 g_I(\boldsymbol x_I)
 =\mathbf 1\!\left\{\sum_{i\in I}w_ix_i>1\right\},
 \qquad g_\varnothing=0,
\]
and define its face increment by M\"obius inversion,
\begin{equation}\label{eq:face-increment}
 G_I(\boldsymbol x_I)
 =\sum_{J\subseteq I}(-1)^{|I|-|J|}g_J(\boldsymbol x_J).
\end{equation}

\begin{lemma}[Exact face decomposition]\label{lem:face-decomposition}
For every $\boldsymbol x\in[0,\infty)^m$,
\begin{equation}\label{eq:face-decomposition}
 \mathbf 1\{\boldsymbol w^{\mathsf T}\boldsymbol x>1\}
 =\sum_{\varnothing\ne I\subseteq[m]}G_I(\boldsymbol x_I).
\end{equation}
Consequently, the tail of a weighted sum is the sum of its one-coordinate
contributions and of the interaction increments on all higher faces.
\end{lemma}

Vague convergence on an open face does not by itself imply
\cref{eq:hidden-face-transfer-result}: $G_I$ is neither compactly supported
nor bounded away from proper subfaces.  The next theorem gives the needed
condition.

\begin{theorem}[Integrable hidden-face transfer]
\label{thm:hidden-face-transfer}
Let $I\subseteq[m]$, $|I|\ge2$, and let $r_I(t)\downarrow0$.  Suppose
\begin{equation}\label{eq:hidden-face-vague}
 r_I(t)^{-1}\Pp(\boldsymbol X_I/t\in\cdot)
 \xrightarrow{v}\nu_I(\cdot)
 \quad\hbox{on }(0,\infty]^{I}.
\end{equation}
Assume that the discontinuity set of $G_I$ is $\nu_I$-null,
$\int|G_I|\,d\nu_I<\infty$, and that for
\[
 K_{\varepsilon,M}
 =\{\boldsymbol x_I:\varepsilon\le x_i\le M, i\in I\}
\]
one has
\begin{equation}\label{eq:hidden-face-ui}
 \lim_{\varepsilon\downarrow0,\,M\uparrow\infty}
 \limsup_{t\to\infty}r_I(t)^{-1}
 \E\!\left[|G_I(\boldsymbol X_I/t)|
       \mathbf1\{\boldsymbol X_I/t\notin K_{\varepsilon,M}\}\right]=0.
\end{equation}
Then
\begin{equation}\label{eq:hidden-face-transfer-result}
 \E G_I(\boldsymbol X_I/t)
 =r_I(t)\int G_I\,d\nu_I+o\{r_I(t)\}.
\end{equation}
If the hypotheses hold on every retained face and every omitted face
satisfies
\begin{equation}\label{eq:hidden-face-omitted-bound}
 \E|G_I(\boldsymbol X_I/t)|=o\{r(t)\},
\end{equation}
then \cref{lem:face-decomposition} gives the expansion at target rate
$r(t)$ by adding the retained face integrals.  One sufficient route to
\cref{eq:hidden-face-omitted-bound} is the same transfer and uniform-
integrability hypotheses on that face together with $r_I(t)=o\{r(t)\}$;
the nominal vague scale alone is not sufficient.
\end{theorem}

Product-power hidden measures, including Gaussian ones, permit an explicit
integrability test.

\begin{proposition}[Product-power pair functional]
\label{prop:product-power-face}
Suppose $I=\{i,j\}$, $0<a_i,a_j<1$, and
\begin{equation}\label{eq:product-power-measure}
 \nu_I(d\boldsymbol x_I)
 =K_Ia_ia_jx_i^{-a_i-1}x_j^{-a_j-1}\,dx_i\,dx_j.
\end{equation}
Then $G_I\in L^1(\nu_I)$ and, with continuous interpretation when
$a_i+a_j=1$,
\begin{equation}\label{eq:product-power-functional}
 \int G_I\,d\nu_I
 =-K_Iw_i^{a_i}w_j^{a_j}
 \frac{\Gamma(1-a_i)\Gamma(1-a_j)}
      {\Gamma(1-a_i-a_j)}.
\end{equation}
In particular, a noncritical hidden face gives a pure power correction.
For faces of dimension at least three, individual conditions $a_i<1$ do
not suffice: partial sums on proper subfaces control integrability.
Equality at one can produce a logarithm; a larger sum can produce
polynomial divergence.
\end{proposition}

\begin{remark}[Critical-face aggregation]
This phrase means that the first non-negligible face correction lies at the
boundary of integrability.  It is the geometric reason for logarithms such
as the independent $(\log t)/t^2$ term.  If the hidden density is integrable
against $G_I$, one obtains a power correction.  At the critical boundary a
truncated integral grows logarithmically; beyond it, the proper subfaces
must be subtracted and analyzed first.  Closely related bivariate
integrable/nonintegrable edge classifications and the independent critical
logarithm are already developed for heavy-tailed copula sums by Yang and
Zhang~\cite{YangZhang2023}.  The purpose here is to embed that phenomenon
in the full coordinate-face lattice and translate it to harmonic-mean and
positive Half-Cauchy calibration.
\end{remark}

\subsection{Critical local-corner calibration near independence}
\label{sec:sharp-local-corner}

The general face calculus in the preceding section becomes operational
when the pairwise copula densities have finite limits at the lower corner.
This subsection isolates that near-independence regime and turns the local
corner limits into explicit logarithmic tail and calibration coefficients.
The independent logarithm is one instance of the broader critical-pair
phenomenon.  In dimension at least three, however, a theorem must also
control higher M\"obius faces: pairwise margins alone cannot exclude a
singular multivariate component.

\begin{lemma}[Product-corner envelope for higher face increments]
\label{lem:product-corner-face-envelope}
For $k\ge2$, put
\[
 \Gamma_k(\boldsymbol z)
 =\sum_{J\subseteq[k]}(-1)^{k-|J|}
  \mathbf1\!\left\{\sum_{j\in J}z_j>1\right\},
 \qquad \boldsymbol z\in(0,\infty)^k.
\]
For $0<\varepsilon<1/k$, there is a constant $K_k<\infty$ such that
\begin{equation}\label{eq:product-corner-face-envelope}
 \int_{[\varepsilon,\infty)^k}
 |\Gamma_k(\boldsymbol z)|
 \prod_{j=1}^kz_j^{-2}\,d\boldsymbol z
 \le
 \begin{cases}
  K_2\{1+\log(1/\varepsilon)\},&k=2,\\[3pt]
  K_k\varepsilon^{-(k-2)},&k\ge3.
 \end{cases}
\end{equation}
Consequently, a bounded product-type density can produce logarithmic
pair-face divergence, whereas each fixed face of dimension at least three
contributes at most $O(t^{-2})$ after nominal $t^{-k}$ scaling.
\end{lemma}

\begin{theorem}[Sharp critical-pair calibration]
\label{thm:sharp-local-corner}
Let $\bU=(U_1,\ldots,U_m)$ have exact uniform margins, let $\bw$ be a
fixed simplex weight vector with active weights strictly positive, and put
$Y_i=U_i^{-1}$.

For every $I\subseteq[m]$ with $|I|\ge2$, suppose the marginal copula of
$\bU_I$ has a density $c_I$.  Assume that, for some $u_0>0$ and constants
$M_I<\infty$,
\begin{equation}\label{eq:lower-strip-all-face-density}
 \mathop{\rm ess\,sup}_{\substack{\boldsymbol u_I\in(0,1)^I\\
                         \min_{i\in I}u_i\le u_0}}
 c_I(\boldsymbol u_I)\le M_I.
\end{equation}
For each pair $i<j$, assume that the finite lower-corner limit
\begin{equation}\label{eq:pair-corner-density-limit}
 \lambda_{ij}
 =\lim_{s\downarrow0}
   \mathop{\rm ess\,sup}_{0<u,v\le s}c_{ij}(u,v)
 =\lim_{s\downarrow0}
   \mathop{\rm ess\,inf}_{0<u,v\le s}c_{ij}(u,v)
\end{equation}
exists.  Define
\begin{equation}\label{eq:weighted-corner-coefficient}
 L_{\bw}^{(0)}=\sum_{i<j}\lambda_{ij}w_iw_j.
\end{equation}
Then
\begin{equation}\label{eq:pareto-local-corner-expansion}
 \Pp\!\left(\sum_{i=1}^mw_iY_i>t\right)
 =\frac1t+2L_{\bw}^{(0)}\frac{\log t}{t^2}
  +o\!\left(\frac{\log t}{t^2}\right).
\end{equation}

More precisely, let
\[
 \omega_{ij}(s)
 =\mathop{\rm ess\,sup}_{0<u,v\le s}
   |c_{ij}(u,v)-\lambda_{ij}|.
\]
If every corner modulus is Dini integrable,
\begin{equation}\label{eq:corner-dini}
 \int_0^{u_0}\frac{\omega_{ij}(s)}s\,ds<\infty,
 \qquad i<j,
\end{equation}
then
\begin{equation}\label{eq:pareto-local-corner-sharp}
 \Pp\!\left(\sum_{i=1}^mw_iY_i>t\right)
 =\frac1t+2L_{\bw}^{(0)}\frac{\log t}{t^2}+O(t^{-2}).
\end{equation}
The Dini condition holds, in particular, if
$c_{ij}(u,v)-\lambda_{ij}=O\{(u+v)^\eta\}$ for some $\eta>0$.

Let
\[
 H_i=\cot\!\left(\frac{\pi U_i}{2}\right),
 \qquad c=\frac2\pi.
\]
Under the same assumptions,
\begin{equation}\label{eq:hc-local-corner-expansion}
 \Pp\!\left(\sum_{i=1}^mw_iH_i>t\right)
 =\frac ct+2c^2L_{\bw}^{(0)}\frac{\log t}{t^2}
  +o\!\left(\frac{\log t}{t^2}\right),
\end{equation}
and the final remainder is $O(t^{-2})$ under
\cref{eq:corner-dini}.
\end{theorem}

\begin{remark}[What is really required of higher faces]
The density condition \cref{eq:lower-strip-all-face-density} for
$|I|\ge3$ is a convenient sufficient condition, not a necessary one.  The
proof only requires
\[
 \sum_{|I|\ge3}|\E G_I(\bY_I/t)|=O(t^{-2}).
\]
Thus it may be replaced by model-specific control of higher M\"obius
increments.  The condition is vacuous in the bivariate case, but pairwise
density bounds do not control higher-dimensional singular components.
\end{remark}

\begin{corollary}[Critical-corner calibration relative to independence]
\label{cor:critical-corner-calibration}
For $S\in\{R,T\}$, put $c_R=1$, $c_T=c=2/\pi$.  Let
$\overline F^S_{C,\bw}$ and $\overline F^S_{\Pi,\bw}$ denote the tails of
$S_{\bw}$ under copula $C$ and under independence, respectively, and let
$q^S_{C,\bw}$ and $q^S_{\Pi,\bw}$ be their upper quantiles.  Define
$\operatorname{Cal}^{S}_{C,\bw}$ using the independence quantile as in
\cref{eq:size-calibration-functional}.  Under the hypotheses of
\cref{thm:sharp-local-corner}, for both the reciprocal/harmonic-mean and
positive Half-Cauchy aggregates,
\begin{align}
 \overline F^S_{C,\bw}(t)-\overline F^S_{\Pi,\bw}(t)
 &=2c_S^2\sum_{i<j}(\lambda_{ij}-1)w_iw_j
   \frac{\log t}{t^2}
   +o\!\left(\frac{\log t}{t^2}\right),
 \label{eq:critical-corner-tail-calibration}\\
 \operatorname{Cal}^{S}_{C,\bw}(\alpha)
 &=2\sum_{i<j}(\lambda_{ij}-1)w_iw_j\,
   \alpha^2\log(1/\alpha)
   +o\{\alpha^2\log(1/\alpha)\}.
 \label{eq:critical-corner-size-calibration}
\end{align}
Under the Dini condition, the two final remainders are respectively
$O(t^{-2})$ and $O(\alpha^2)$, and
\begin{equation}\label{eq:critical-corner-quantile-calibration}
 q^S_{C,\bw}(\alpha)-q^S_{\Pi,\bw}(\alpha)
 =2c_S\sum_{i<j}(\lambda_{ij}-1)w_iw_j\log(1/\alpha)+O(1).
\end{equation}
\end{corollary}

For common bivariate copulas with an interior parameter, the theorem
gives the following lower-corner multipliers:
\begin{table}[!ht]
\makeatletter
\long\def\@makecaption#1#2{%
  \vskip\abovecaptionskip
  \centering\footnotesize #1: #2\par
  \vskip\belowcaptionskip}
\makeatother
\caption{Lower-corner density multipliers for common bivariate copulas.}
\label{tab:lower-corner-multipliers}
\centering
\small
\setlength{\tabcolsep}{4pt}
\begin{tabular}{@{}lc@{\hspace{3em}}lc@{}}
\toprule
Copula & $\lambda=c_C(0,0)$ & Copula & $\lambda=c_C(0,0)$\\
\midrule
product & $1$ & Frank with $\theta\ne0$ & $\theta/(1-e^{-\theta})$\\
FGM with parameter $\theta$ & $1+\theta$
  & Plackett with $\theta>0$ & $\theta$\\
Ali--Mikhail--Haq with $\theta<1$ & $(1-\theta)^{-1}$ & & \\
\bottomrule
\end{tabular}
\end{table}
The corresponding densities are analytic or polynomial at the lower corner for fixed
interior parameters, so the Dini condition holds.  Finite Bernstein and
checkerboard copulas are also covered whenever their lower-corner cell
density has a well-defined limit.  These specializations should be viewed
against the existing second-order risk-aggregation literature
\cite{Kortschak2012,Coqueret2014}; the theorem's intended contribution is
the face-calculus formulation and the harmonic-mean/positive-Half-Cauchy
calibration coefficient,
not a priority claim for every bivariate special case.

\section{The second-order landscape inside index-one MRV}
\label{sec:second-order-landscape}

It is useful to standardize terminology before treating the models.  If
$S_{\bw}$ has leading tail constant $c_0$, write
\begin{equation}\label{eq:relative-calibration-error}
 t\Pp(S_{\bw}>t)=c_0+E_{\bw}(t),
 \qquad
 E_{\bw}(t)\sim K_{\bw}t^{-\kappa}L(t).
\end{equation}
Here $L$ is \emph{slowly varying}:
\begin{equation}\label{eq:slow-variation}
 \frac{L(tx)}{L(t)}\longrightarrow1
 \quad\text{for every fixed }x>0.
\end{equation}
Constants, powers of $\log t$, and powers of $\log\log t$ are typical
examples.  The absolute second tail term in
\cref{eq:relative-calibration-error} is
$K_{\bw}t^{-1-\kappa}L(t)$.  A smaller positive $\kappa$ means slower
convergence.  At the same $\kappa$, a growing $L$ is slower and a decaying
$L$ is faster than a pure power.

Support and asymptotic order answer different questions.  The support of
$\mu$ records where first-order extremes occur.  We use the axial, interior,
proper-face, and mixed categories defined by the charged strata in
\cref{eq:charged-strata} and illustrated in
\cref{fig:exponent-measure-geometry}.  The word ``face-supported'' is
sometimes used broadly enough to include axes; here it always means a
nontrivial proper coordinate face unless stated otherwise.

A support category does not determine a rate.  Within each category below
we therefore distinguish perturbations of strata already charged by
$\mu$, hidden mass on uncharged strata, the intrinsic score-map scale, and
rate-free convergence.  Every hidden contribution is integrated against
its M\"obius half-space increment and classified as finite, critically
divergent, or polynomially divergent.  The slowest valid absolute rate
wins, tied terms add, and score-map terms are compared afterward.  This is
an audit framework rather than a completed recursive theorem for all
nested faces.

\subsection{Axis-supported copulas}
\label{sec:axis-supported-category}

Axis support means that ordinary MRV sees only one-coordinate extremes.
It says merely that every comparable multi-coordinate extreme is
$o(t^{-1})$; it does not determine its smaller rate.  In the present exact
$p$-value setting, the first possible ordinary axis mechanism can be
largely removed before studying the missing faces.

\paragraph*{What exact marginal normalization removes}

The reciprocal and Half-Cauchy scores must be distinguished at second
order.  Exact reciprocal scores have no marginal tail correction at all,
whereas the Half-Cauchy map has an intrinsic third-order tail term.

Exact Pareto standardization forces every marginal
projection of an ordinary signed second-order measure to vanish; if that
measure is supported only on the axes, it is zero.  In contrast, exact
Half-Cauchy margins have
\[
 t\Pp(H_i>tx)=\frac{c}{x}-\frac{c}{3x^3}t^{-2}+O(t^{-4}),
 \qquad c=\frac2\pi,
\]
so the score transformation itself contributes at absolute order $t^{-3}$.
The formal signed-measure statement is in \TechnicalLocation.

Thus a nonzero second term for the exact reciprocal sum cannot be a purely
marginal radial correction on the axes.  It must involve dependence,
angular redistribution, a hidden or boundary face, or a failure of a
regular expansion.  This does \emph{not} mean that every quantitative
main-cone correction vanishes: a signed measure can have zero marginal
projections and still be nonzero on a weighted half-space.  For
$T_{\bw}$ the $t^{-3}$ score-map correction is always a candidate, whereas
$R_{\bw}$ has no marginal score-map correction; both are often dominated,
for example, by the independent $(\log t)/t^2$ interaction term.

\paragraph*{The face integral and the
integrable--critical--nonintegrable trichotomy}

For two positive weights, the pair increment is
\begin{equation}\label{eq:bivariate-face-increment-explicit}
 G_{12}(x_1,x_2)
 =\mathbf1\{w_1x_1+w_2x_2>1\}
  -\mathbf1\{w_1x_1>1\}
  -\mathbf1\{w_2x_2>1\}.
\end{equation}
The exact identity
\begin{align}
 \Pp(w_1X_1+w_2X_2>t)
 ={}&\Pp(w_1X_1>t)+\Pp(w_2X_2>t)\notag\\
 &+\E G_{12}(X_1/t,X_2/t)
 \label{eq:bivariate-face-tail-identity}
\end{align}
shows that $G_{12}$ counts cooperative exceedances positively and removes
double-counted simultaneous marginal exceedances negatively.  If a hidden
measure $\nu_{12}$ appears at absolute scale $r_{12}(t)$, the number
\begin{equation}\label{eq:face-integral-definition}
 \int G_{12}\,d\nu_{12}
\end{equation}
is the \emph{face integral}.  When it is absolutely integrable and the
uniform-integrability condition of
\cref{thm:hidden-face-transfer} holds, it is the coefficient that transfers
the hidden probability scale to the sum tail.  If that absolute scale is
$r_I(t)=t^{-q_I}L_I(t)$, then its relative calibration exponent is
$q_I-1$.  Integrability transfers the scale as it stands; it does not
itself create a logarithm.

The difficulty is that the open face $(0,\infty]^2$ is not separated from
its axes.  Values such as $x_2\downarrow0$ represent an original score
$X_2=o(t)$ that may nevertheless range through all intermediate scales.

For a symmetric product-power face density with
exponent $a$, truncation at distance $\varepsilon$ from the axes gives
\[
 \int |G_{12}|\,d\nu_a=
 \begin{cases}
 O(1),&0<a<1,\\
 \Theta\{\log(1/\varepsilon)\},&a=1,\\
 \Theta(\varepsilon^{1-a}),&a>1.
 \end{cases}
\]
Thus the three regimes are integrable, critical, and polynomially
nonintegrable.  A formal proposition is given in \TechnicalLocation.

This proposition places independence and Gaussian copulas on one
continuous boundary diagram.  For independent unit-Pareto margins, the
hidden pair density is $x_1^{-2}x_2^{-2}$, so $a=1$, the nominal hidden
scale is $t^{-2}$, and the effective cutoff $\varepsilon\asymp t^{-1}$
produces $t^{-2}\log t$.  For a one-sided Gaussian copula with
$0<\rho<1$,
\begin{equation}\label{eq:gaussian-a-q-beta-taxonomy}
 a_\rho=\frac1{1+\rho}<1,
 \qquad q_\rho=2a_\rho=\frac2{1+\rho},
 \qquad \beta_\rho=1-a_\rho=\frac{\rho}{1+\rho}.
\end{equation}
Its hidden probability scale is
\begin{equation}\label{eq:gaussian-hidden-scale-taxonomy}
 r_\rho(t)
 =t^{-q_\rho}(\log t)^{-\beta_\rho}.
\end{equation}
Here $q_\rho$ is the power exponent, $\beta_\rho$ is the exponent of the
slowly varying Gaussian-quantile factor, and $a_\rho$ controls spatial
integrability near a proper face.  The logarithmic factor in
\cref{eq:gaussian-hidden-scale-taxonomy} is already part of the Gaussian
joint-tail scale; it is not created by the face integral.  Since
$a_\rho<1$, the face integral is finite; when the boundary-transfer
condition holds, the sum correction has the same order as $r_\rho(t)$.
Its relative exponent is
\begin{equation}\label{eq:gaussian-relative-exponent-preview}
 \kappa_G=q_\rho-1=\frac{1-\rho}{1+\rho}.
\end{equation}

At $\rho=0$, $a_\rho=1$ and the face integral becomes critical, producing
the growing independent logarithm.  Here ``critical'' means failure at the
boundary of integrability, not ``logarithmic by definition''; nested
boundaries can generate other slowly varying amplifications.  For
$-1<\rho<0$, $a_\rho>1$ and the
formal lower--lower hidden face is nonintegrable.  This is a concrete
nonintegrable-face example, but not a completed aggregation theorem:
proper edge regions must first be resolved, and the one-sided nonpositive
Gaussian expansion with at least one strictly negative correlation remains
open.  The all-zero boundary is independence and is already proved.  The
standard-$t$ angular boundary at
$\nu=1$ is another critical example, while $0<\nu<1$ is nonintegrable;
those cases also remain open.

\paragraph*{Singular exclusion geometry}
Hidden positive mass is not the only possible axis-supported mechanism.
For the countermonotone copula $(U_1,U_2)=(U,1-U)$, comparable lower
extremes are exactly excluded.  With weights $(w,1-w)$, direct solution of
the quadratic boundary gives
\begin{equation}\label{eq:axis-counter-reciprocal}
 \Pp(R_{\bw}>t)
 =1-\frac{\sqrt{t^2-2t+(1-2w)^2}}{t}
 =\frac1t+\frac{2w(1-w)}{t^2}+O(t^{-3}).
\end{equation}
For the positive Half-Cauchy scores, writing
$H=\cot(\pi U/2)$ gives the other score as $H^{-1}$ and
\[
 T_{\bw}=wH+(1-w)H^{-1}.
\]
The exact calculation in \CounterTailSource{} yields
\begin{equation}\label{eq:axis-singular-exclusion}
 \Pp(T_{\bw}>t)=\frac{2}{\pi t}
 +\frac2\pi\left\{2w(1-w)-\frac13\right\}t^{-3}
 +O(t^{-5}).
\end{equation}
Thus the independent critical logarithm disappears for both aggregates.
The reciprocal statistic retains an order-$t^{-2}$ dependence correction,
whereas for the positive Half-Cauchy statistic that coefficient cancels
and the $t^{-3}$ transformation scale becomes visible.  This is a
singular-support cancellation/exclusion mechanism, not an integrable
hidden-pair correction.

\paragraph*{Rate-free and other possibilities}
Ordinary MRV alone supplies no numerical rate.
\TechnicalLocation{} gives an exact-uniform copula with axial lower MDA
whose measure-level convergence is slower than every positive power and
has no nonzero regularly varying diagonal normalization.  Rate-free
convergence can occur under any support category, although it is especially
important here because axial first-order support otherwise tempts one to
infer a hidden power rate.  Signed cancellations, several hidden faces with
tied rates, or a copula singularity near an opposite corner can likewise
determine the first omitted term.  Axis support alone does not choose among
these possibilities.

\subsection{Interior-supported copulas}
\label{sec:interior-supported-category}

Interior support means that the ordinary exponent measure already sees
simultaneous extremes.  Consequently the standard hidden-regular-variation
question ``at what smaller scale do several coordinates become large
together?'' does not arise.

\begin{proposition}[No hidden subset joint-extreme cone under interior
support]
\label{prop:interior-no-hidden-subset}
Suppose $S$ is a nonzero spectral measure concentrated on
$F_{[m]}^\circ$.  For every nonempty $I\subseteq[m]$,
\begin{equation}\label{eq:interior-subset-mass}
 \mu\{\boldsymbol x:x_i>1\text{ for every }i\in I\}
 =\int_{F_{[m]}^\circ}\min_{i\in I}s_i\,S(d\boldsymbol s)>0.
\end{equation}
Hence every fixed subset joint exceedance has first-order probability:
\begin{equation}\label{eq:interior-subset-tail}
 \Pp(X_i>t\text{ for every }i\in I)
 \sim \frac{1}{t}
 \int\min_{i\in I}s_i\,S(d\boldsymbol s).
\end{equation}
\end{proposition}

A smaller-scale limit concentrated near an exact boundary face may still
describe an angular boundary layer, but it is not automatically an
additive second term.  Neighborhoods of that face have already been
integrated as part of the ordinary interior measure.  One must first
subtract the full first-order contribution and prove a quantitative
boundary expansion.  Thus the core second-order problem for an
interior-supported copula is ordinarily a perturbation of mass already
present---radial, angular, or both---rather than the first appearance of a
joint-extreme face.  Boundary integrability can still fail after this
subtraction, as in the standard-$t$ cases $\nu\le1$.

This is the canonical setting for ordinary second-order MRV: a signed
perturbation of the main-cone measure is evaluated on the weighted
half-space through the quantitative projection in \TechnicalLocation.
Exact Pareto margins force its one-coordinate projections to vanish, but
they do not force its joint half-space coefficient to vanish.

The common-factor models in this paper make the dominant mechanism
explicit.  The following elementary power transformation explains why
their score-scale radial variable has index one.

\begin{lemma}[Homogeneous power converts the common radial index to one]
\label{lem:common-radial-index-one}
Suppose
\begin{equation}\label{eq:raw-common-radial-alpha}
 \Pp(R>r)=a r^{-\alpha}
 \{1-br^{-\beta}+o(r^{-\beta})\},
 \qquad \alpha,\beta>0.
\end{equation}
Put $Q=R^\alpha$.  Then
\begin{equation}\label{eq:powered-common-radial-one}
 \Pp(Q>q)=a q^{-1}
 \{1-bq^{-\beta/\alpha}+o(q^{-\beta/\alpha})\}.
\end{equation}
If the leading score map is homogeneous of degree $\alpha$ in $R$, it can
therefore be written at leading order as $Q A(\Theta)$ with an index-one
common radial factor.
\end{lemma}

The support-agnostic common-factor inversion theorem in
\TechnicalLocation{} converts a second-order radial tail, an angular score
expansion, and the intrinsic score-map expansion into an aggregate-tail
coefficient.  Its central rule is that the smallest relative exponent
wins and tied radial, angular, and score-map terms add.  The theorem also
states the dominated-inversion conditions needed near angular subfaces.

For the standard multivariate-$t_\nu$ model, the transformed radial factor
is $Q=R_0^\nu$ and its relative second-order exponent is $2/\nu$.  Two-sided scores have a
strictly positive angular vector and hence interior support; one-sided
scores are mixed-supported because their positive-part angular vector can
lie on lower-dimensional faces.  The same inversion theorem applies to
both orientations.  For $\nu>1$ it yields the ordinary radial correction
at absolute order $t^{-1-2/\nu}$, with the additional one-sided reciprocal
angular term described in \cref{thm:t-second-order-formal}.

For positive Clayton, $Q=V^{-1/\theta}$ already has index one, its relative
second-order exponent is $\theta$, and the exponential angular coordinates are strictly
positive.  The model is therefore interior-supported.  The reciprocal
correction has absolute order $t^{-1-\theta}$, while the positive
Half-Cauchy correction meets its score-map scale at $\theta=2$ and is
governed by that scale for $\theta>2$; see
\cref{thm:clayton-second-order,thm:clayton-hc-phase}.  The formal
standard-$t$ collision occurs at $\nu=1$, where angular integrability also
becomes critical.

\subsection{Proper-face-supported copulas}
\label{sec:face-supported-copulas}

Genuine proper faces first appear when $m\ge3$.  Marshall--Olkin
and max-linear common-shock copulas can place spectral atoms on selected
proper-face rays, while Tawn asymmetric logistic models can place diffuse
mass on selected faces \cite{MarshallOlkin1967,Tawn1990}.  The required
lower-tail orientation must be stated explicitly.  A concrete
three-dimensional max-linear construction and its exact-uniform
transformation are recorded in \TechnicalLocation.

\paragraph*{A sharp proper-face pair-shock theorem}
The max-linear construction in
\TechnicalLocation{} admits an exact
shock representation on the $p$-value scale.  This is a spectrally
discrete max-linear model in the sense discussed, for example, by Wang and
Stoev~\cite{WangStoev2011}.  Let $E_{12},E_{13},E_{23}$ be independent
$\operatorname{Exp}(1)$ variables and put
\[
 T_1=\min(E_{12},E_{13}),\qquad
 T_2=\min(E_{12},E_{23}),\qquad
 T_3=\min(E_{13},E_{23}),
\]
together with
\begin{equation}\label{eq:pair-shock-uniforms}
 U_i=1-\exp(-2T_i),\qquad i=1,2,3.
\end{equation}
Since each $T_i$ is $\operatorname{Exp}(2)$, every $U_i$ is exactly
uniform.  Moreover,
\begin{equation}\label{eq:pair-shock-survival-copula}
 \Pp(U_1>u_1,U_2>u_2,U_3>u_3)
 =\prod_{i<j}\{1-\max(u_i,u_j)\}^{1/2}.
\end{equation}
Thus this is the lower-tail, or reflected, orientation of a
Marshall--Olkin pair-shock copula.  For a simplex weight vector define
\begin{equation}\label{eq:sigma-two-weights}
 \sigma_2(\bw)=\sum_{i<j}w_iw_j.
\end{equation}

The auxiliary three-edge critical-integral lemma in
\TechnicalLocation{} shows that two independent edge shocks create a
$t^{-2}\log t$ correction with an explicit coefficient.

\begin{theorem}[Proper-face pair-shock aggregation]
\label{thm:pair-shock-second-order}
Let $\bU$ be given by \cref{eq:pair-shock-uniforms} and suppose $w_i>0$
with $\sum_iw_i=1$.  Set
\[
 Y_i=U_i^{-1},
 \qquad
 H_i=\cot(\pi U_i/2),
 \qquad
 c=\frac2\pi.
\]
Then the spectral measure of $\bY$ is
\begin{equation}\label{eq:pair-shock-spectral-sharp}
 S=\delta_{(1/2,1/2,0)}
   +\delta_{(1/2,0,1/2)}
   +\delta_{(0,1/2,1/2)}.
\end{equation}
In particular, it is supported purely on nontrivial proper coordinate
faces.  Furthermore,
\begin{align}
 \Pp\!\left(\sum_{i=1}^3w_iY_i>t\right)
 &=\frac1t
   +\sigma_2(\bw)\frac{\log t}{t^2}
   +O(t^{-2}),
 \label{eq:pair-shock-pareto-second-order}\\
 \Pp\!\left(\sum_{i=1}^3w_iH_i>t\right)
 &=\frac ct
   +c^2\sigma_2(\bw)\frac{\log t}{t^2}
   +O(t^{-2}).
 \label{eq:pair-shock-hc-second-order}
\end{align}
\end{theorem}

\begin{corollary}[Proper-face calibration errors]
\label{cor:pair-shock-calibration}
For $S\in\{R,T\}$, set $c_R=1$, $c_T=c$, let
$F_{S,\boldsymbol w}^{\Pi}$ be the corresponding independent-reference
distribution with the same weights, and let
$q_{\Pi,S,\boldsymbol w}(\alpha)$ be its upper $\alpha$-quantile.  Under the
pair-shock model,
\begin{equation}\label{eq:pair-shock-independent-difference}
 \Pp(S_{\boldsymbol w}>t)
 -\{1-F_{S,\boldsymbol w}^{\Pi}(t)\}
 =-c_S^2\sigma_2(\bw)\frac{\log t}{t^2}+O(t^{-2}).
\end{equation}
Consequently, as $\alpha\downarrow0$,
\begin{equation}\label{eq:pair-shock-size-distortion}
 \Pp\{S_{\boldsymbol w}>q_{\Pi,S,\boldsymbol w}(\alpha)\}
 =\alpha-\sigma_2(\bw)\alpha^2\log(1/\alpha)+O(\alpha^2),
\end{equation}
whereas at the simpler first-order threshold $c_S/\alpha$,
\begin{equation}\label{eq:pair-shock-first-order-threshold}
 \Pp(S_{\boldsymbol w}>c_S/\alpha)
 =\alpha+\sigma_2(\bw)\alpha^2\log(1/\alpha)+O(\alpha^2).
\end{equation}
\end{corollary}

The interpretation is recursive.  One large shock supplies the
proper-face first-order exponent measure.  Two independent edge shocks
supply hidden full-interior mass of order $t^{-2}$, and its accumulation
near the already-charged edge rays is critical, producing the logarithm.
The coefficient is exactly half the independent-coordinate coefficient.
This is a sharp theorem for the symmetric three-dimensional pair-shock
model, not a general second-order theorem for all proper-face-supported
copulas.

The second-order problem in this category can combine three directions of
movement: quantitative perturbation on the already-charged proper faces,
hidden mass on larger faces not charged at first order, and nonuniform
approach to smaller subfaces.  The weighted half-space reaches all of them.
This is the natural setting for recursive M\"obius subtraction.  Nested
critical boundaries can create higher powers of logarithms or other slowly
varying amplifications; polynomially divergent face integrals require the
adjacent smaller faces to be subtracted and analyzed first.  No general
theorem yet covers all such nested rates.

\subsection{Mixed-support copulas}
\label{sec:mixed-supported-category}

Mixed support is common once $m\ge3$; representative examples are:
\begin{enumerate}[leftmargin=2.2em]
 \item \emph{One-sided standard multivariate-$t$ scores.}
 Their leading angular vector is proportional to
 \begin{equation}\label{eq:one-sided-t-mixed-direction}
  ((G_1^+)^\nu,\ldots,(G_m^+)^\nu).
 \end{equation}
 For a positive-definite correlation matrix every Gaussian sign orthant
 has positive probability.  Hence different sign patterns generate axes,
 every available proper coordinate face, and the full interior.  This is
 a mixed-support model, whereas the two-sided version is
 interior-supported.

 \item \emph{Extremal-$t$ copulas} \cite{Opitz2013}.
 Their positive-part Gaussian spectral representation yields the same
 sign-pattern mechanism and hence a natural mixture of coordinate strata.

 \item \emph{Marshall--Olkin/max-linear and Tawn asymmetric logistic
 models.}
 Including shocks or logistic components indexed by subsets of different
 sizes produces a mixture of axis atoms, proper-face components, and
 possibly full-interior mass.

 \item \emph{Convex mixtures of copulas in a common lower MDA.}
 If exact-uniform copulas $C_k$ are mixed with fixed probabilities
 $\pi_k$, their exponent measures mix as
 \begin{equation}\label{eq:mixture-exponent-measures}
  \mu=\sum_k\pi_k\mu_k.
 \end{equation}
 Mixing, for example, an axis-supported Gaussian copula with an
 interior-supported Clayton copula gives a direct mixed-support model.
\end{enumerate}

Mixed support need not always be analyzed one face at a time.  For
$\nu>1$, the common-radial representation gives the one-sided standard-$t$
expansions for both $R_{\bw}$ and $T_{\bw}$ by integrating over the whole
Gaussian angular law; their second terms differ when the bounded
one-sided angular contribution becomes visible.  Face recursion
becomes indispensable when charged strata have different quantitative
rates, absent strata appear at hidden scales, or angular integrals fail
near proper subfaces.  A general second-order mixed-support theorem is not
claimed here.

\paragraph*{Why axes and full interior remain the canonical poles}

Axis support and full-interior support are not exhaustive in dimension at
least three, but they are defensible organizing poles for this problem.
Axis-supported models represent asymptotic independence and include
independence, Gaussian copulas, and many smooth lower-tail-independent
families.  Their central second-order question is how dependence first
reappears on hidden faces.  Full-interior models represent asymptotic
dependence and include positive Clayton and two-sided standard-$t$ scores.
Their central question is how already-visible joint mass approaches its
radial-angular limit.

Proper-face and mixed-support models describe clustered or blockwise
extremes and cannot be dismissed as pathological.  They form the natural
intermediate theory.  The four subsections above are therefore an
organizing principle, not a claim that the axial and full-interior poles
exhaust multivariate copulas.

The four support categories classify where the first-order exponent measure
lives.  The mechanisms in \cref{tab:second-order-mechanisms} instead classify
how the next contribution reaches the weighted half-space.  These are crossed
classifications: first-order support alone does not determine the second-order
mechanism.  Singular exclusion, a pure score-map correction, and rate-free
convergence are additional possibilities beyond the four principal
measure-transfer mechanisms in the table.

\begingroup
\small
\setlength{\tabcolsep}{4pt}
\begin{longtable}{@{}p{0.23\textwidth}p{0.40\textwidth}p{0.29\textwidth}@{}}
\caption{Four principal second-order measure-transfer mechanisms.}
\label{tab:second-order-mechanisms}\\
\toprule
Mechanism & Explanation & Examples \\
\midrule
\endfirsthead
\toprule
Mechanism & Explanation & Examples \\
\midrule
\endhead
\bottomrule
\endlastfoot

Ordinary second-order MRV
& On the main cone,
$t\Pp(\boldsymbol X/t\in\cdot)
=\mu(\cdot)+A(t)\nu_2(\cdot)+o\{A(t)\}$.
When the signed perturbation is integrable on the weighted half-space, it
produces an absolute tail correction of order $t^{-1}A(t)$.  This mechanism
perturbs mass already visible at first order.
& Positive Clayton and standard multivariate-$t_\nu$, $\nu>1$, through the
proved common-factor expansions; one-sided reciprocal $t$ aggregation also
has an angular competition. \\[3pt]

Integrable hidden face
& The first-order measure gives no mass to a face $I$, but a hidden measure
appears at an absolute scale $r_I(t)=o(t^{-1})$.  If
$\int |G_I|\,d\nu_I<\infty$, together with the required uniform
integrability, the aggregate inherits the same scale $r_I(t)$; the face
integration creates no additional logarithm.
& One-sided Gaussian copulas with an active positive correlation and
two-sided Gaussian copulas with a nonzero active correlation.  Their
compact hidden-face limits are known; the weighted-half-space coefficient is
conditional on the explicit boundary hypotheses in
\TechnicalLocation. \\[3pt]

Critical hidden face
& The hidden-face integral fails exactly at its integrability boundary.  A
finite-threshold cutoff against adjacent subfaces then produces a slowly
varying amplification.  One simple critical boundary gives $\log t$,
although nested boundaries can yield other slowly varying factors.
& Product copula and the three-edge pair-shock model; their
$t^{-2}\log t$ terms are proved. \\[3pt]

Nonintegrable hidden face
& The truncated face integral diverges polynomially.  The nominal hidden
rate cannot be transferred directly to the sum tail: adjacent
lower-dimensional strata must first be subtracted and analyzed through the
recursive M\"obius increments.
& The formal lower--lower hidden face of a one-sided Gaussian copula with
negative correlation; the diagnosis is established, but its leading
aggregation expansion remains open. \\
\end{longtable}
\endgroup

\cref{tab:second-order-mechanisms} is the conceptual map.  The full
model-by-model ledger in \FullLedgerPlacement{} records first-order support,
absolute second-order rate, mechanism, proof status, and bounded-density
diagnostics.

\section{Gaussian copulas}

\subsection{One-sided \texorpdfstring{$p$}{p}-values}

Let $\boldsymbol Z=(Z_1,\ldots,Z_m)$ be centered Gaussian with unit
variances and correlation matrix $R=(\rho_{ij})$.  Assume
\begin{equation}\label{eq:rho-max}
  \rho_{\max}=\max_{i\ne j}|\rho_{ij}|<1.
\end{equation}
Define the one-sided $p$-values
\[
  U_i=1-\Phi(Z_i).
\]

\begin{theorem}[Gaussian-copula score vector]\label{thm:gaussian}
Under \cref{eq:rho-max}, the reciprocal scores
$Y_i=\{1-\Phi(Z_i)\}^{-1}$ and positive Half-Cauchy scores
$H_i=\cot[\pi\{1-\Phi(Z_i)\}/2]$ are index-one MRV with axis-supported tail
measure.  Hence \cref{thm:hcct-mrv} applies to both $R_{\bw}$ and $T_{\bw}$.
\end{theorem}

This proof avoids applying an equal-threshold bivariate-normal asymptotic to
unequal thresholds.  It also makes clear why strict inequality in
\cref{eq:rho-max} is required: correlation one can produce first-order
simultaneous extremes, although that case may still be MRV with a
non-axis-supported measure.

\subsection{Two-sided \texorpdfstring{$p$}{p}-values}

For
\[
  U_i=2\{1-\Phi(|Z_i|)\},
\]
the pair event is the union of four sign-quadrant events.  For signs
$\varepsilon_i,\varepsilon_j\in\{-1,1\}$, the pair
$(\varepsilon_i Z_i,\varepsilon_j Z_j)$ is Gaussian with correlation
$\varepsilon_i\varepsilon_j\rho_{ij}$, whose absolute value is at most
$\rho_{\max}$.  Applying the pair bound from \cref{thm:gaussian} to all four quadrants and
summing proves \cref{eq:pair-comparable}.  Thus the two-sided Gaussian
reciprocal and positive Half-Cauchy score vectors are also index-one MRV
with axis-supported tail measure.

\begin{remark}[What is regularly varying]
The Gaussian vector $\boldsymbol Z$ itself is not regularly varying; its
tails are too light.  The transformed vector $(1/U_1,\ldots,1/U_m)$, and
hence the Half-Cauchy score vector, is regularly varying.  Here MRV concerns
the transformed $p$-values, not the Gaussian observations.
\end{remark}

\paragraph*{Hidden-pair scale and the independence boundary}
For a one-sided Gaussian pair with $0<\rho<1$, simultaneous reciprocal
scores have scale
\[
 t^{-q_\rho}(\log t)^{-\beta_\rho},\qquad
 q_\rho=\frac{2}{1+\rho},\quad
 \beta_\rho=\frac{\rho}{1+\rho}.
\]
The hidden pair density has spatial exponent
$a_\rho=(1+\rho)^{-1}<1$, so its pair-face integral is finite for fixed
$\rho>0$.  At $\rho=0$, $a_\rho=1$ and the face becomes critical,
producing the independent $(\log t)/t^2$ term.  The nonuniform crossover
occurs at $\rho\log t=O(1)$.  The detailed bivariate calculation and the
negative-correlation diagnosis are in \TechnicalLocation.

\paragraph*{Conditional fixed-dimensional expansion}
Let $\rho_*$ be the largest active positive correlation for one-sided
$p$-values, or the largest active absolute correlation for two-sided
$p$-values.  Under the explicit face-uniform-integrability and
higher-face remainder conditions stated in \TechnicalLocation, the
reciprocal and positive Half-Cauchy tails have the respective forms
\begin{align*}
 \Pp(R_{\bw}>t)
 &=t^{-1}+D_{G,\bw}
 t^{-q_{\rho_*}}(\log t)^{-\beta_{\rho_*}}
 +o\{t^{-q_{\rho_*}}(\log t)^{-\beta_{\rho_*}}\},\\
 \Pp(T_{\bw}>t)
 &=ct^{-1}+c^{q_{\rho_*}}D_{G,\bw}
 t^{-q_{\rho_*}}(\log t)^{-\beta_{\rho_*}}
 +o\{t^{-q_{\rho_*}}(\log t)^{-\beta_{\rho_*}}\},
 \qquad c=2/\pi.
\end{align*}
The coefficient is an explicit sum over the maximally correlated pairs.
The bivariate hidden vague limit is proved, but the required truncation at
adjacent axes and the higher-face remainder are not yet verified in full
dimension.  This is therefore a conditional weighted-half-space result,
not an unconditional Gaussian theorem; it is also nonuniform at both the
independence and perfect-correlation boundaries.  The one-sided problem
with nonpositive correlations remains open.

\section{Common-factor copulas: multivariate-\texorpdfstring{$t$}{t} and
positive Clayton}

This section specializes the common-factor machinery of
\cref{sec:interior-supported-category} first to the standard
multivariate-$t$ copula and then to positive Clayton copulas.

For the standard common-scale multivariate-$t_\nu$ copula, both
one- and two-sided reciprocal and positive Half-Cauchy score vectors are
index-one MRV.  Two-sided scores have interior spectral mass; the
one-sided positive-part angular vector produces mixed support.  The
first-order proof, the common-scale representation, and the usual
tail-dependence formula are collected in \TechnicalLocation.

\paragraph*{Radial source of the second-order rate}
The transformed radial factor has index one and relative second-order
exponent $2/\nu$, yielding candidate aggregate order $t^{-1-2/\nu}$.
For one-sided reciprocal scores it competes with a bounded angular term at
$\nu=2$; see \TechnicalLocation.

\subsection{Second-order theorem for the standard
multivariate-\texorpdfstring{$t$}{t} copula}

The preceding qualification is resolved here for the standard common-scale
model.  For $\boldsymbol G\sim N_m(0,R)$, define
\begin{align}
 a_\nu&=\frac{\nu^{\nu/2}}
 {2^{\nu/2}\Gamma(\nu/2+1)},
 &b_\nu&=\frac{\nu^2}{2(\nu+2)},\label{eq:t-ab}\\
 k_\nu&=\frac{\Gamma((\nu+1)/2)\nu^{\nu/2-1}}
 {\sqrt\pi\,\Gamma(\nu/2)},
 &d_\nu&=\frac{\nu^2(\nu+1)}{2(\nu+2)}.
 \label{eq:t-kd}\\[-2pt]
 A_+&=k_\nu^{-1}\sum_iw_i(G_i^+)^\nu,
 &C_-&=\sum_iw_i\mathbf1\{G_i<0\},
 \label{eq:t-Aplus}\\[-2pt]
 B_+&=d_\nu k_\nu^{-1}\sum_iw_i
       \mathbf1\{G_i>0\}G_i^{\nu-2},\label{eq:t-Bplus}\\
 A_{\lvert\cdot\rvert}&=(2k_\nu)^{-1}\sum_iw_i|G_i|^\nu,
 &B_{\lvert\cdot\rvert}&=d_\nu(2k_\nu)^{-1}
       \sum_iw_i|G_i|^{\nu-2}.
 \label{eq:t-AabsBabs}\\[-2pt]
 D_{\nu,\boldsymbol w}^{s}
 &=a_\nu\E\!\left[(A_s)^{2/\nu}(B_s-b_\nu A_s)\right],
 \qquad s\in\{+,|\cdot|\}.
 \label{eq:t-D}
\end{align}

\begin{theorem}[Second-order standard multivariate-$t$ aggregation]
\label{thm:t-second-order-formal}
Fix $m$, a correlation matrix $R$, simplex weights $\boldsymbol w$, and
$\nu>1$.  Let $\boldsymbol Z$ have the standard multivariate-$t_\nu$
distribution with correlation matrix $R$, and let $T_\nu$ denote the
univariate standard-$t_\nu$ CDF.

For one-sided reciprocal scores $Y_i=\{1-T_\nu(Z_i)\}^{-1}$,
\begin{equation}\label{eq:t-one-sided-pareto-formal}
 \Pp\!\left(\sum_iw_iY_i>t\right)
 =\begin{cases}
 t^{-1}+D_{\nu,\boldsymbol w}^{+}t^{-1-2/\nu}
       +o(t^{-1-2/\nu}),&\nu>2,\\[3pt]
 t^{-1}+\{D_{2,\boldsymbol w}^{+}
       +a_2\E(A_+C_-)\}t^{-2}+o(t^{-2}),&\nu=2,\\[3pt]
 t^{-1}+a_\nu\E(A_+C_-)t^{-2}+o(t^{-2}),&1<\nu<2.
 \end{cases}
\end{equation}
For two-sided reciprocal scores
$Y_i=[2\{1-T_\nu(|Z_i|)\}]^{-1}$,
\begin{equation}\label{eq:t-two-sided-pareto-formal}
 \Pp\!\left(\sum_iw_iY_i>t\right)
 =t^{-1}+D_{\nu,\boldsymbol w}^{|\cdot|}t^{-1-2/\nu}
 +o(t^{-1-2/\nu}).
\end{equation}
For either one- or two-sided Half-Cauchy scores, with the corresponding
choice of $s$,
\begin{equation}\label{eq:t-hc-formal}
 \Pp\!\left(\sum_iw_iH_i>t\right)
 =\frac ct+c^{1+2/\nu}D_{\nu,\boldsymbol w}^{s}
 t^{-1-2/\nu}+o(t^{-1-2/\nu}),
 \qquad c=2/\pi.
\end{equation}
The coefficient is zero for a single active coordinate, as exact marginal
uniformity requires.
\end{theorem}

The displayed standard-$t$ coefficients are dimension-free for
fixed $\nu>1$, but this does not control the fixed-dimensional remainders.
The coefficient proposition and the unresolved $\nu\le1$ angular boundary
are stated in \TechnicalLocation.

\subsection{Positive Clayton copulas: common-frailty aggregation}
\label{sec:clayton-formal}

For $\theta>0$, the $m$-variate Clayton copula is
\begin{equation}\label{eq:clayton-m}
 C_{\theta,m}(u_1,\ldots,u_m)
 =\left(\sum_{i=1}^mu_i^{-\theta}-m+1\right)^{-1/\theta}.
\end{equation}
It is lower-tail dependent, but it is not a $t$ copula.  Its Archimedean
frailty representation supplies a particularly transparent common-factor
calculation.  Let $V\sim\operatorname{Gamma}(1/\theta,1)$ and, independently,
$E_i\stackrel{\mathrm{iid}}\sim\operatorname{Exp}(1)$; set
$U_i=(1+E_i/V)^{-1/\theta}$.
Then $\boldsymbol U$ has copula \cref{eq:clayton-m} and uniform margins.
Write $p=1/\theta$ and $R=V^{-1/\theta}$.
The exact reciprocal representation and radial expansion are
\begin{align}
 Y_i&=R(E_i+R^{-\theta})^p,
 \label{eq:clayton-Y-representation}\\
 \Pp(R>r)&=a_\theta r^{-1}
 \{1-b_\theta r^{-\theta}+O(r^{-2\theta})\},
 \label{eq:clayton-R-tail}\\
 a_\theta&=\Gamma(1+1/\theta)^{-1},
 \qquad b_\theta=(1+\theta)^{-1}.
\end{align}
For simplex weights put $A_{\boldsymbol w}=\sum_iw_iE_i^p$.
Write $B_{\boldsymbol w}=p\sum_iw_iE_i^{p-1}$, and let
$D_{\theta,\boldsymbol w}=a_\theta\E\bigl[A_{\boldsymbol w}^{\theta}
(B_{\boldsymbol w}-b_\theta A_{\boldsymbol w})\bigr]$.

\begin{theorem}[Reciprocal/harmonic-mean aggregation under positive Clayton]
\label{thm:clayton-second-order}
For every fixed $m$, $\theta>0$, and simplex weight vector,
\begin{equation}\label{eq:clayton-pareto-second}
 \Pp\!\left(\sum_iw_iY_i>t\right)
 =\frac1t+D_{\theta,\boldsymbol w}t^{-1-\theta}
 +o(t^{-1-\theta}).
\end{equation}
Moreover,
\begin{equation}\label{eq:clayton-D-alternative}
 D_{\theta,\boldsymbol w}
 =\frac{a_\theta}{\theta+1}
 \E\!\left[(m-1)A_{\boldsymbol w}^{\theta+1}
       -\sum_{i=1}^m(A_{\boldsymbol w}-w_iE_i^p)^{\theta+1}\right].
\end{equation}
It is strictly positive when at least two weights are positive and is zero
for a single active coordinate.
\end{theorem}

The Half-Cauchy map has a second intrinsic scale:
\begin{equation}\label{eq:hc-third-expansion}
 \cot\!\left(\frac{\pi}{2Y}\right)
 =cY-dY^{-1}+O(Y^{-3}),
 \qquad c=\frac2\pi,\quad d=\frac\pi6.
\end{equation}
This competes with the Clayton radial correction.

\begin{theorem}[Positive Half-Cauchy phase transition under Clayton]
\label{thm:clayton-hc-phase}
Under the conditions of \cref{thm:clayton-second-order}, put
$Q_{\boldsymbol w}=\sum_iw_iE_i^{-1/\theta}$.  Then
\begin{equation}\label{eq:clayton-hc-phase}
 \Pp\!\left(\sum_iw_iH_i>t\right)=\frac ct+
 \begin{cases}
 c^{1+\theta}D_{\theta,\boldsymbol w}t^{-1-\theta}
   +o(t^{-1-\theta}),&0<\theta<2,\\
 \{c^3D_{2,\boldsymbol w}-a_2dc^2
   \E(A_{\boldsymbol w}^2Q_{\boldsymbol w})\}t^{-3}+o(t^{-3}),
   &\theta=2,\\
 -a_\theta dc^2\E(A_{\boldsymbol w}^2Q_{\boldsymbol w})t^{-3}
   +o(t^{-3}),&\theta>2.
 \end{cases}
\end{equation}
The inverse moment in the last two cases is finite.  For one active coordinate,
the formula reduces to the exact univariate expansion
$\Pp(H>t)=c/t-c/(3t^3)+O(t^{-5})$.
\end{theorem}

\subsection{Parametric benchmarks}

Positive Clayton and standard multivariate-$t$ are both common-factor
models after reciprocal-score transformation, but they have different
angular laws and second-order exponents: $\theta$ for Clayton and $2/\nu$
for standard $t$.  The generic phase diagram, the coefficient bounds, and
the explanation of the angularly critical $t$ boundary are given in
\TechnicalLocation.

The principal parametric benchmarks can now be compared directly.  In
\cref{tab:principal-model-rates}, the absolute second-order rate is the
order of the first nonzero aggregate-tail term after $1/t$ for
$R_{\bw}$, or after $c/t$, $c=2/\pi$, for $T_{\bw}$.  Every rate marked
proved is fixed-dimensional.  For the Gaussian row, $\rho_\star$ denotes
the largest active positive correlation for one-sided $p$-values and the
largest active absolute correlation for two-sided $p$-values; put
$q_\star=\textstyle\frac2{1+\rho_\star}$ and
$\beta_\star=\textstyle\frac{\rho_\star}{1+\rho_\star}$.

\begingroup
\footnotesize
\setlength{\tabcolsep}{3.5pt}
\begin{longtable}{@{}p{0.13\textwidth}p{0.15\textwidth}
                    p{0.38\textwidth}p{0.24\textwidth}@{}}
\caption{Principal support, rate, mechanism, and proof-status landscape.}
\label{tab:principal-model-rates}\\
\toprule
Model & First-order support & Second-order mechanism and current status
& Absolute second-order rate \\
\midrule
\endfirsthead
\toprule
Model & First-order support & Second-order mechanism and current status
& Absolute second-order rate \\
\midrule
\endhead
\bottomrule
\endlastfoot
Gaussian, $\rho_\star>0$
& Axes
& Integrable hidden pair faces; the candidate coefficient is derived, and
the expansion is conditional on the explicit face-transfer conditions in
\TechnicalLocation.
& $t^{-q_\star}(\log t)^{-\beta_\star}$ \\

Product copula
& Axes
& Critical hidden-pair accumulation; the logarithmic coefficients are
proved for $R_{\bw}$ and $T_{\bw}$.
& $t^{-2}\log t$ \\

Standard multivariate-$t_\nu$, $\nu>1$
& Two-sided: interior; one-sided: mixed
& Ordinary common-radial perturbation.  For one-sided $R_{\bw}$, a bounded
angular term dominates when $1<\nu<2$ and ties at $\nu=2$.  Proved in
fixed dimension.
& $T_{\bw}$ and two-sided $R_{\bw}$: $t^{-1-2/\nu}$;
one-sided $R_{\bw}$: $t^{-1-2/\nu}$ for $\nu>2$ and $t^{-2}$ for
$1<\nu<2$, with a tie at $\nu=2$ \\

Positive Clayton, $0<\theta<2$
& Interior
& Ordinary common-frailty perturbation.  Both fixed-$m$ expansions are
proved.
& $t^{-1-\theta}$ for $R_{\bw}$ and $T_{\bw}$ \\

Positive Clayton, $\theta\ge2$
& Interior
& Common frailty for $R_{\bw}$; for $T_{\bw}$, a tie at $\theta=2$ and
score-map dominance above.  Fixed-$m$ result proved.
& $R_{\bw}$: $t^{-1-\theta}$; $T_{\bw}$: $t^{-3}$ \\
\end{longtable}
\endgroup

No rate is assigned here to one-sided Gaussian models having no positive
active correlation and at least one strictly negative correlation, or to
standard multivariate-$t$ models with $0<\nu\le1$: their leading
second-order aggregation problems remain open.  The all-zero Gaussian case
is independence and is already settled.
The broader family-by-family classification, including conditional
mechanism diagnoses and rate-free cases, is given in
\FullLedgerPlacement.

\section{Discussion}
\label{sec:main-novelty-boundary}

The practical motivation is the absence of exact null calibration for
heavily right aggregation under dependence.  Harmonic-mean and positive
Half-Cauchy scores have universal index-one first-order tails, but
$c/t+o(t^{-1})$ does not quantify error at testing levels.  Relative to the
exact independence law, the next term determines size distortion and the
dependence-specific critical-value shift.  Its rate, coefficient, and sign
therefore govern how quickly first-order validity becomes statistically
useful.

The central geometric conclusion is that the weighted half-space reaches
every coordinate face.  Its correction depends both on where the
first-order exponent measure lives and on how finite-threshold mass
approaches the face lattice.  Ordinary signed perturbations change mass
already visible at first order; integrable hidden faces transfer their own
scale; critical faces create slowly varying amplification; and
nonintegrable faces require recursive subtraction.  Support and mechanism
are therefore complementary classifications.

The machinery makes this geometry operational.  Exact Pareto
standardization removes marginal artifacts, quantitative exponent measures
identify the scale, M\"obius increments decompose the noncompact
half-space, and the hidden-face and common-factor theorems turn face or
radial--angular expansions into tail coefficients.  Calibration inversion
then converts those coefficients into rejection-probability and
critical-value errors.  This is a reusable route from dependence geometry
to statistical calibration, rather than a collection of unrelated copula
calculations.

The examples show why both layers are needed.  Product and Gaussian
copulas are axial but lie in critical and integrable hidden-face regimes;
positive Clayton and two-sided standard-$t$ perturb interior mass;
one-sided standard-$t$ has mixed support; and the pair-shock model gives a
sharp proper-face logarithm.  \cref{tab:principal-model-rates} compares the
principal benchmarks, while the full ledger in \FullLedgerLabel{} records
support, rate, mechanism, and proof status without turning conditional
diagnoses into theorems.

Harmonic-mean and positive Half-Cauchy aggregation share first-order
geometry but separate at genuine transformation boundaries.  The latter
has an intrinsic $t^{-3}$ correction, while one-sided reciprocal $t$
aggregation can have a bounded angular term absent from its Half-Cauchy
analogue.  Treating the rules in parallel reveals which corrections belong
to dependence and which belong to the score map.

The second-order conclusions are fixed-dimensional.  The final block of
\TechnicalLocation{} explains why fixed-dimensional MRV is insufficient
when $m$ grows, gives uniform-MRV projection and radial-kernel criteria,
and proves dimension-uniform first-order validity for fixed-parameter
positive Clayton and standard multivariate-$t$ common-factor arrays.  It
also shows why dimension-free coefficients do not control second-order
remainders.  The main open problems are a recursive theorem for nested
critical and nonintegrable faces, unconditional Gaussian boundary control,
the angularly critical standard-$t$ cases, and uniform control of
active-face multiplicity when $m$ grows.

More broadly, the classification applies to dependence structures rather
than only named copula families.  Organizing them by exponent-measure
support and by how finite-threshold mass crosses the aggregation boundary
turns universal first-order calibration into a quantitative theory of rate,
coefficient, and sign, linking extreme-value geometry to finite-level
heavily right combination tests.

\section*{Acknowledgements}
\addcontentsline{toc}{section}{Acknowledgements}
OpenAI's ChatGPT and Codex assisted with mathematical exploration, language editing, and \LaTeX{} preparation.

\clearpage
\appendix
\AppendixCrefSetup
\section{Technical extensions and supporting results}
\label{app:technical-extensions}

This appendix collects results and extended derivations that support
the main geometric and calibration arguments.  All proof environments
remain in the separate proofs appendix.
Throughout, $U_i$ denotes an exact-uniform $p$-value,
$Y_i=U_i^{-1}$, $H_i=\cot(\pi U_i/2)$,
$R_{\bw}=\sum_iw_iY_i$, $T_{\bw}=\sum_iw_iH_i$, and $c=2/\pi$.

\subsection{Exponent measures, Poisson limits, and max-stability}
\label{sec:exponent-measure-intuition}

The term \emph{exponent measure} is most concrete through an extreme point
process.  If $\boldsymbol X^{(1)},\ldots,\boldsymbol X^{(n)}$ are
independent copies of an index-one MRV vector, then, under the usual point
process formulation,
\begin{equation}\label{eq:extreme-point-process}
 N_n=\sum_{k=1}^n\delta_{\boldsymbol X^{(k)}/n}
 \quad\Longrightarrow\quad N,
\end{equation}
where $N$ is a Poisson point process with intensity $\mu$.  Consequently,
$\mu(A)$ is the limiting expected number, among $n$ observations, whose
normalized extreme vector falls in $A$.  In particular,
\begin{equation}\label{eq:poisson-void}
 \Pp\{N(A)=0\}=\exp\{-\mu(A)\}.
\end{equation}

Let $\boldsymbol M_n$ denote the componentwise maximum of the $n$ copies.
For $\boldsymbol x>\boldsymbol0$, put
\begin{equation}\label{eq:exponent-function-from-measure}
 V(\boldsymbol x)
 =\mu\!\left([0,\infty]^m\setminus[\boldsymbol0,\boldsymbol x]\right)
 =\mu\{\boldsymbol y:y_i>x_i\text{ for some }i\}.
\end{equation}
Then
\begin{align}
 \Pp(\boldsymbol M_n/n\le\boldsymbol x)
 &=\left[1-\Pp\{X_i>nx_i\text{ for some }i\}\right]^n\notag\\
 &\longrightarrow\exp\{-V(\boldsymbol x)\}.
 \label{eq:max-stable-from-exponent}
\end{align}
The exponential therefore comes from the zero-count probability of a
Poisson limit, equivalently from $(1-V/n)^n\to e^{-V}$.  This is the
origin of the word ``exponent''; it is unrelated to a Lie-group
exponential map.

Homogeneity gives $V(a\boldsymbol x)=a^{-1}V(\boldsymbol x)$, and hence
\[
 G(\boldsymbol x)=e^{-V(\boldsymbol x)}
 \quad\Longrightarrow\quad
 G(a\boldsymbol x)^a=G(\boldsymbol x)
\]
for integer $a$, with the usual continuous extension.  Thus the maximum
of $a$ independent vectors having distribution $G$, divided by $a$, has
distribution $G$ again.  This is the meaning of \emph{max-stable} at
index one.

There is a limited Haar-measure analogy.  Positive scalars act on
$\mathbb E_m$ by dilation.  Haar measure on the multiplicative group
$(0,\infty)$ is $dr/r$, whereas the index-$\alpha$ radial exponent measure
is
\begin{equation}\label{eq:haar-character}
 \alpha r^{-\alpha-1}\,dr
 =\alpha r^{-\alpha}\frac{dr}{r}.
\end{equation}
It is therefore Haar measure multiplied by the character $r^{-\alpha}$:
the exponent measure is relatively invariant, not invariant, under
dilations.  In logarithmic radius $u=\log r$, the radial measure is the
exponential tilt $\alpha e^{-\alpha u}du$ of translation-invariant Haar
measure.  This observation explains homogeneity, but not the terminology
``exponent measure.''

The lower-tail replicate interpretation is the same construction after
reciprocation.  For independent copies
$\boldsymbol U^{(1)},\ldots,\boldsymbol U^{(n)}$, set
$M_{n,i}=\min_{1\leq k\leq n}U_i^{(k)}$.  If the lower stable-tail
function $\ell_L$ in \cref{def:lstd} exists, then
\begin{equation}\label{eq:min-limit}
 \Pp(nM_{n,i}>x_i\text{ for every }i)
 =\left[1-\Pp\!\left\{\bigcup_i(U_i\leq x_i/n)\right\}\right]^n
 \longrightarrow \exp\{-\ell_L(\boldsymbol x)\}.
\end{equation}
Thus $\ell_L$ is the exponent of the limiting min-stable law.  Since
\begin{equation}\label{eq:min-max-identity}
 \min_{1\leq k\leq n}U_i^{(k)}
 =\left\{\max_{1\leq k\leq n}Y_i^{(k)}\right\}^{-1},
\end{equation}
the same statement is equivalently a max-domain limit for
$\boldsymbol Y=1/\boldsymbol U$ and index-one MRV of $\boldsymbol Y$.
This also explains the phrase ``lower max-domain of attraction.''  In two
dimensions, the union in \cref{eq:min-limit} describes at least one small
$p$-value, while the intersection probability
$C(sx,sy)$ describes both being small; the limit
$C(sx,sy)/s\to\Lambda_L(x,y)$ is the corresponding first-order overlap.

\subsection{A copula outside the lower max-domain of attraction}
\label{sec:oscillating-copula-counterexample}

The existence requirement in \cref{def:lstd} is substantive.
Uniform margins do not prevent the copula from alternating between
different lower-tail clustering patterns on successively smaller scales.
The following bivariate construction gives an explicit example.

Let $s_0=e^{-2}$ and, for $0<s\leq s_0$, define
\begin{equation}\label{eq:oscillating-diagonal-small}
  \delta(s)
  =s\left[
    \frac12+\frac18\sin\{\log\log(1/s)\}
  \right],
  \qquad \delta(0)=0.
\end{equation}
Extend $\delta$ linearly from $(s_0,\delta(s_0))$ to $(1,1)$:
\begin{equation}\label{eq:oscillating-diagonal-extension}
  \delta(s)
  =\delta(s_0)
   +\frac{1-\delta(s_0)}{1-s_0}(s-s_0),
  \qquad s_0\leq s\leq1.
\end{equation}
On $(0,s_0]$,
\[
  \delta'(s)
  =\frac12+\frac18\sin\{\log\log(1/s)\}
   -\frac{\cos\{\log\log(1/s)\}}{8\log(1/s)},
\]
so $0<\delta'(s)<2$.  The linear extension also has slope in $(0,2)$,
and the construction satisfies
\[
  \delta(1)=1,
  \qquad
  \delta(s)\leq s,
  \qquad
  0\leq\delta(v)-\delta(u)\leq2(v-u)
  \quad(0\leq u\leq v\leq1).
\]
These are the standard conditions for a bivariate copula diagonal.  Hence
there exists a copula $C_\delta$ with
$C_\delta(s,s)=\delta(s)$; diagonal-section characterizations and explicit
constructions are given by Durante and Jaworski~\cite{DuranteJaworski2008}.

For exact uniform margins,
\begin{equation}\label{eq:oscillating-conditional}
  \frac{C_\delta(s,s)}{s}
  =\Pp(U_2\leq s\mid U_1\leq s)
  =\frac12+\frac18\sin\{\log\log(1/s)\},
  \qquad 0<s\leq s_0.
\end{equation}
Along the sequences
\begin{align*}
  s_n^{(+)}
  &=\exp\!\left[-\exp\!\left\{\frac{\pi}{2}+2\pi n\right\}\right],\\
  s_n^{(-)}
  &=\exp\!\left[-\exp\!\left\{\frac{3\pi}{2}+2\pi n\right\}\right],
\end{align*}
the ratio in \cref{eq:oscillating-conditional} equals $5/8$ and $3/8$,
respectively.  Therefore the lower-tail coefficient
$\Lambda_L(1,1)$ does not exist.  Equivalently,
\begin{equation}\label{eq:oscillating-mrv-failure}
  t\Pp(Y_1>t,Y_2>t)
  =tC_\delta(1/t,1/t)
  =\frac{C_\delta(1/t,1/t)}{1/t}
\end{equation}
has different subsequential limits, so $(Y_1,Y_2)$ has no unique MRV tail
measure.

The normalized lower-minimum probability also oscillates:
\begin{equation}\label{eq:oscillating-lower-minimum}
  \frac1s\Pp\{\min(U_1,U_2)\leq s\}
  =2-\frac{C_\delta(s,s)}s,
\end{equation}
with subsequential values $11/8$ and $13/8$.  The copula therefore has no
stable first-order lower-extreme geometry, even though both marginal
$p$-values are exactly uniform.

This example proves failure of lower MDA/MRV, not automatically failure of
either aggregate-tail ratio.  Convergence of an entire multivariate tail
measure is stronger than convergence of one particular linear-functional
tail; a reciprocal or positive Half-Cauchy weighted sum can conceivably
converge even when the full measure does not.  A counterexample to
first-order aggregation validity itself would have to control the oscillating
copula mass relative to the corresponding weighted-sum half-space, not merely
its diagonal.  Standard parametric copulas are usually much more regular;
the example is intentionally scale-oscillatory.

\subsection{Quantitative rates and a hierarchy of assumptions}
\label{sec:quantitative-tail-control}

It is important not to conflate three logically different conditions.

\begin{description}[leftmargin=2.7em,style=nextline]
  \item[Level I: lower max-domain attraction, or ordinary MRV.]
  This is the existence of the first-order limit
  \[
    t\Pp(\boldsymbol X/t\in A)\longrightarrow\mu_X(A),
    \qquad \boldsymbol X\in\{\boldsymbol Y,\boldsymbol H\},
  \]
  on fixed continuity sets.  For the corresponding weighted half-space it
  gives only
  \[
    \Pp\!\left(\sum_iw_iX_i>t\right)
    =\frac{c_X}{t}+o(t^{-1}),
    \qquad c_Y=1,\quad c_H=\frac2\pi.
  \]
  It supplies no numerical rate for the little-$o$ remainder.

  \item[Level II: quantitative second-order MRV.]
  This adds a rate to the convergence of the exponent measures on a class
  containing the weighted-sum half-space.  It is the general condition
  needed for a second-order aggregate-tail conclusion and allows both asymptotic
  independence and asymptotic dependence.

  \item[Level III: product-type pair control.]
  A shell-probability bound is an elementary sufficient condition for the
  desired Level-II rate in the asymptotically independent case.  A direct
  bound on each pairwise score density is stronger and implies the shell
  bound.  A globally bounded bivariate copula density is an even simpler but
  still stricter way to verify the score-density bound.  None of these
  Level-III conditions is necessary for Level II.
\end{description}

These levels form a one-way hierarchy.  The examples and quantitative
projection below clarify its strict steps; the full chain ends this section.

\begin{example}[Axial lower MDA with an arbitrarily slow, nonregular
diagonal remainder]
\label{ex:slow-copula-mda}
For $s$ sufficiently small, put
\begin{equation}\label{eq:slow-copula-diagonal}
 r(s)=\log(1/s),\qquad
 \eta(s)=\frac{2+\sin r(s)}{\log r(s)},\qquad
 \delta(s)=s\eta(s).
\end{equation}
Choose $s_0>0$ small enough and extend $\delta$ linearly from
$(s_0,\delta(s_0))$ to $(1,1)$, with $\delta(0)=0$.  On $(0,s_0]$,
\[
 \delta'(s)
 =\frac{2+\sin r-\cos r}{\log r}
  +\frac{2+\sin r}{r(\log r)^2}.
\]
For sufficiently small $s_0$, this derivative lies in $(0,2)$,
$\delta(s)\leq s$, and the linear extension has the same properties.
Thus $\delta$ is an admissible bivariate copula diagonal, so there is a
copula $C_\delta$ satisfying
$C_\delta(s,s)=\delta(s)$; see Durante and
Jaworski~\cite{DuranteJaworski2008}.

This copula is nevertheless in the axial lower max-domain of attraction.
Indeed, for fixed $x,y>0$ and $M=\max(x,y)$, monotonicity gives
\begin{equation}\label{eq:slow-copula-axial-bound}
 0\leq \frac{C_\delta(sx,sy)}s
 \leq \frac{C_\delta(sM,sM)}s
 =M\eta(sM)\longrightarrow0.
\end{equation}
Hence every pairwise comparable lower-tail event is $o(s)$ and
\cref{prop:pairwise} gives axis-supported index-one MRV.

No regularly varying normalization makes the diagonal remainder converge
to a finite nonzero constant.  Indeed, the quantity
$C_\delta(s,s)/s=\eta(s)$ tends to zero more slowly than $s^\kappa$ for
every $\kappa>0$, while
\[
 \frac{\eta(e^{-\pi}s)}{\eta(s)}
\]
has subsequential limits $1/3$ and $3$.  Thus ordinary lower MDA supplies
neither a positive-power rate nor regular variation of its rectangle
errors.  This diagonal obstruction does not exclude a global
second-order signed measure vanishing on this rectangle or a special
cancellation on one weighted half-space.
\end{example}

\subsection{Concrete conditions near independence}
\label{sec:near-independence-conditions}

The following probability and density assumptions provide such control
for copulas sufficiently close to product behavior in the relevant lower
tail.  The most general elementary version used here is a probability
bound on joint shells; it does not require the pair to possess a joint
density.

\begin{proposition}[Shell-probability alternative]
\label{prop:shell-probability}
Let $\boldsymbol w$ be a fixed simplex weight vector and, after discarding
zero weights, let $m$ be the number of active coordinates.  Let
$\boldsymbol X$ be either $\boldsymbol Y$ or $\boldsymbol H$, put
$c_Y=1$, $c_H=2/\pi$, and set $Z_i=w_iX_i$.
Suppose that, for every ordered pair $i\ne j$, there is $K_{ij}<\infty$ such
that, for all $0\leq a<b<\infty$ and $y\geq0$,
\begin{equation}\label{eq:shell-probability-bound}
  \Pp(a<Z_i\leq b,\ Z_j>y)
  \leq
  \frac{K_{ij}}{1+y}
  \int_a^b\frac{dx}{(1+x)^2}.
\end{equation}
Then
\begin{equation}\label{eq:shell-probability-conclusion}
  \Pp\!\left(\sum_iw_iX_i>t\right)
  =\frac{c_X}{t}
   +O\!\left(\frac{\log t}{t^2}\right).
\end{equation}
\end{proposition}

The next density condition is stronger but often easier to verify.

\begin{theorem}[Product-type pair densities]
\label{thm:pair-density-second-order}
Let $\boldsymbol w$ be a fixed simplex weight vector and, after discarding
zero weights, let $m$ be the number of active coordinates.  Let
$\boldsymbol X$ be either $\boldsymbol Y$ or $\boldsymbol H$, put
$c_Y=1$, $c_H=2/\pi$, and set $Z_i=w_iX_i$.
Suppose every pair $(Z_i,Z_j)$ is absolutely continuous and, for constants
$K_{ij}<\infty$,
\begin{equation}\label{eq:score-density-bound}
  g_{ij}(x,y)
  \leq\frac{K_{ij}}{(1+x)^2(1+y)^2},
  \qquad x,y\geq0.
\end{equation}
Then \cref{eq:shell-probability-conclusion} holds.
\end{theorem}

Condition \cref{eq:score-density-bound} implies the shell-probability
bound \cref{eq:shell-probability-bound}, because
\[
  \int_y^\infty\frac{dz}{(1+z)^2}=\frac{1}{1+y}.
\]
Thus the shell-probability hypothesis in
\cref{prop:shell-probability} is weaker than the product-density hypothesis
in \cref{thm:pair-density-second-order}; the proposition is therefore the
logically stronger statement and subsumes the theorem.  The density
formulation is retained because its proof makes the source of the
logarithmic integral transparent.

\begin{corollary}[Bounded bivariate copula densities]
\label{cor:bounded-copula-density}
The conclusion of \cref{thm:pair-density-second-order} holds for both score
vectors if each
bivariate $p$-value copula is absolutely continuous and
\begin{equation}\label{eq:global-copula-bound}
  \sup_{0<u,v<1}c_{ij}(u,v)<\infty.
\end{equation}
\end{corollary}

The global bound \cref{eq:global-copula-bound} is stronger than the proof
needs.  Since one selected score in the proof is at least $t/m$, it is
enough that, for some $u_0<1$, the copula densities are bounded on both
lower-tail strips
\begin{equation}\label{eq:strip-copula-bound}
  \sup_{0<u\leq u_0,\ 0<v<1}c_{ij}(u,v)<\infty,
  \qquad
  \sup_{0<v\leq u_0,\ 0<u<1}c_{ij}(u,v)<\infty.
\end{equation}
Thus a density singularity confined to the upper--upper corner need not
affect either aggregate's upper tail.  Conversely, bounded density implies
\[
  C_{ij}(s,s)\leq Ms^2,
\]
so it excludes every copula with positive lower-tail-dependence
coefficient.  This explains why Level III cannot cover all Level-II MRV
models.

\subsection{General quantitative half-space projection}
\label{sec:general-quantitative-mrv}

The preceding examples and sufficient conditions concern sharply different
boundary behaviors.  Their common abstract formulation is quantitative
convergence of the exponent measure on the particular weighted half-space.

\begin{proposition}[Projection of a quantitative tail-measure bound]
\label{thm:second-order-mrv}
Under the assumptions of \cref{thm:hcct-mrv}, let
$\boldsymbol X$ be either $\boldsymbol Y$ or $\boldsymbol H$, put
$c_Y=1$, $c_H=2/\pi$, and let
\[
  A_{\boldsymbol w}
  =\{\boldsymbol x:\boldsymbol w^{\mathsf T}\boldsymbol x>1\}.
\]
Suppose that, for some deterministic $a(t)\downarrow0$,
\begin{equation}\label{eq:quant-mrv}
  t\Pp(\boldsymbol X/t\in A_{\boldsymbol w})
  =\mu_X(A_{\boldsymbol w})+O\{a(t)\}.
\end{equation}
Then
\begin{equation}\label{eq:quant-mrv-conclusion}
  \Pp\!\left(\sum_iw_iX_i>t\right)
  =\frac{c_X}{t}+O\!\left(\frac{a(t)}{t}\right).
\end{equation}
In particular, if
\begin{equation}\label{eq:desired-second-rate}
  a(t)=O\!\left(\frac{\log t}{t}\right),
\end{equation}
then
\begin{equation}\label{eq:log-second-order}
  \boxed{
  \Pp\!\left(\sum_iw_iX_i>t\right)
  =\frac{c_X}{t}
   +O\!\left(\frac{\log t}{t^2}\right).}
\end{equation}
\end{proposition}

Condition \cref{eq:quant-mrv} concerns the exponent measure on the
relevant half-space, not merely the marginal tails.  A conventional
second-order MRV formulation assumes, on a suitable class of continuity
sets,
\[
  t\Pp(\boldsymbol X/t\in\cdot)
  =\mu_X(\cdot)+A(t)\nu(\cdot)+o\{A(t)\}.
\]
\cref{thm:second-order-mrv} follows when
$A(t)=O\{(\log t)/t\}$, the weighted half-space belongs to the class on
which second-order convergence is assumed, and
$|\nu|(\partial A_{\boldsymbol w})=0$.  Finiteness of
$\nu(A_{\boldsymbol w})$ alone is not a signed-measure portmanteau
condition.

In lower-copula notation, a necessary diagnostic for such a uniform
measure-level rate is
\begin{equation}\label{eq:second-order-ell-diagnostic}
  \frac{1}{s}\Pp\!\left(
    \bigcup_i\{p_i\leq sx_i\}
  \right)
  =\ell_L(\boldsymbol x)
   +O\{s\log(1/s)\},
  \qquad s\downarrow0,
\end{equation}
uniformly on the required compact sets.  Condition
\cref{eq:second-order-ell-diagnostic} on rectangle complements alone is not
always sufficient for the weighted-sum half-space; a
convergence-determining class and boundary control are also needed.  It is,
however, useful for ruling out the desired rate: if
\cref{eq:second-order-ell-diagnostic} already fails there, the full standard
second-order exponent-measure condition fails.

The conclusions of this section are more accurately summarized by two
one-way chains:
\begin{equation}\label{eq:condition-hierarchy}
\begin{aligned}
  &\text{globally bounded bivariate copula densities}\\
  &\quad\Longrightarrow
  \text{product-type pairwise score densities}\\
  &\quad\Longrightarrow
  \text{joint shell-probability bounds}\\
  &\quad\Longrightarrow
  \begin{matrix}
  \text{quantitative aggregate-tail control on }A_{\boldsymbol w},\\[-1pt]
  \text{and axis-supported ordinary lower MDA/MRV};
  \end{matrix}\\[4pt]
  &\text{second-order MRV on a convergence-determining class}\\
  &\quad\Longrightarrow
  \begin{matrix}
  \text{ordinary MRV, and quantitative control on }A_{\boldsymbol w}\\[-1pt]
  \text{when that half-space is admissible.}
  \end{matrix}
\end{aligned}
\end{equation}
For the first chain, the shell bound gives pairwise comparable-exceedance
probabilities of order $O(s^2)$ by taking one shell endpoint to infinity;
\cref{prop:pairwise} then gives axial MRV.  One-half-space control does not
imply full MRV, and no reverse implication holds in general.

\subsection{Exact marginal normalization on the axes}

\begin{proposition}[Exact Pareto margins rule out a pure-axis second-order
measure]
\label{prop:exact-margin-axis-second}
Let $\boldsymbol Y$ have exact unit-Pareto margins.  Suppose, for a
nonzero $A(t)\to0$ and a signed Radon measure $\nu_2$, that
\begin{equation}\label{eq:exact-margin-second-mrv}
 t\Pp(\boldsymbol Y/t\in\cdot)
 =\mu(\cdot)+A(t)\nu_2(\cdot)+o\{A(t)\}
\end{equation}
holds on all marginal upper half-spaces
$B_i(x)=\{\boldsymbol y:y_i>x\}$, $x>0$.  Then
\begin{equation}\label{eq:second-measure-zero-margins}
 \nu_2\{\boldsymbol y:y_i>x\}=0
 \qquad(i=1,\ldots,m,\ x>0).
\end{equation}
If $\nu_2$ is supported on the union of the coordinate axes, then
$\nu_2=0$.

For exact standard Half-Cauchy margins, the marginal expansion instead is
\begin{equation}\label{eq:hc-axis-second-measure}
 t\Pp(H_i>tx)
 =\frac{c}{x}-\frac{c}{3x^3}t^{-2}+O(t^{-4}),
 \qquad c=\frac2\pi,
\end{equation}
so a nonzero pure-axis score-transformation correction exists at relative
order $t^{-2}$, equivalently absolute tail order $t^{-3}$.
\end{proposition}

\subsection{The product-power face trichotomy}

\begin{proposition}[Symmetric product-power face trichotomy]
\label{prop:product-power-trichotomy}
Let $w_1,w_2>0$ and let an open pair-face measure have density
\begin{equation}\label{eq:symmetric-product-power-density}
 \nu_a(dx_1,dx_2)
 =K a^2x_1^{-a-1}x_2^{-a-1}\,dx_1dx_2,
 \qquad a>0.
\end{equation}
Then the behavior of the absolute face integral near either axis is
\begin{equation}\label{eq:product-power-trichotomy-rate}
 \int_{\min(x_1,x_2)>\varepsilon}|G_{12}|\,d\nu_a
 =
 \begin{cases}
  O(1),&0<a<1,\\
  \Theta\{\log(1/\varepsilon)\},&a=1,\\
  \Theta(\varepsilon^{1-a}),&a>1,
 \end{cases}
 \qquad \varepsilon\downarrow0,
\end{equation}
up to finite contributions away from the axes.  Thus $a<1$, $a=1$, and
$a>1$ yield integrability, criticality, and polynomial nonintegrability,
respectively.
\end{proposition}

\subsection{Support-agnostic common-factor inversion}
\newcommand{\CommonFactorTheory}{%
On this index-one scale, the following support-agnostic theorem gives the
reusable inversion.  Its principal applications have interior mass, but it
also covers the mixed one-sided multivariate-$t$ model.

\begin{theorem}[Common-heavy-factor aggregation]
\label{thm:root-inversion}
Let $Q$ and $\Theta$ be independent.  Suppose, for $a>0$, $\rho>0$,
and $b\in\R$,
\begin{equation}\label{eq:common-radial-tail}
 \Pp(Q>q)=a q^{-1}\{1-bq^{-\rho}+r_Q(q)\},
 \qquad r_Q(q)=o(q^{-\rho}).
\end{equation}
Let $S(q,\Theta)$ be a nonnegative statistic.  For some finite collection
$\lambda_j>0$, suppose its large-$q$ expansion is
\begin{equation}\label{eq:common-stat-expansion}
 S(q,\Theta)
 =qA(\Theta)\left\{1+
       \sum_{j=1}^Jc_j(\Theta)q^{-\lambda_j}
       +r_S(q,\Theta)\right\}.
\end{equation}
Put
\[
 \delta=\min\{\rho,\lambda_1,\ldots,\lambda_J\}.
\]
For large $t$, assume that the upper tail event can be written
\begin{equation}\label{eq:common-threshold-event}
 \{S(Q,\Theta)>t\}=\{Q>q_t(\Theta)\},
\end{equation}
up to an event of probability $o(t^{-1-\delta})$, where
$q_t\to\infty$ almost surely on $\{A>0\}$; set $q_t=\infty$ on
$\{A=0\}$.  Assume the following explicit dominated-inversion conditions:
\begin{align}
 &\E\left|q_t^{-1}-\frac At
   -\sum_{j:\lambda_j=\delta}
       c_jA^{1+\delta}t^{-1-\delta}\right|
       =o(t^{-1-\delta}),
 \label{eq:common-inverse-L1}\\
 &\E(q_t^{-1-\rho})
   =t^{-1-\rho}\E(A^{1+\rho})+o(t^{-1-\delta}),
 \label{eq:common-radial-L1}\\
 &\E\{q_t^{-1}|r_Q(q_t)|\}=o(t^{-1-\delta}).
 \label{eq:common-remainder-L1}
\end{align}
Here and below a product containing $A$ is defined as zero on $\{A=0\}$.
Then
\begin{equation}\label{eq:common-factor-result}
 \Pp\{S(Q,\Theta)>t\}
 =\frac{a\E A}{t}+K_\delta t^{-1-\delta}
  +o(t^{-1-\delta}),
\end{equation}
where
\begin{equation}\label{eq:common-factor-coefficient}
 K_\delta
 =a\sum_{j:\lambda_j=\delta}
       \E(c_jA^{1+\delta})
  -a\mathbf1\{\rho=\delta\}b\E(A^{1+\delta}).
\end{equation}
Thus corrections having the smallest score-scale exponent add, and a
radial correction adds precisely when its exponent ties that minimum.
\end{theorem}

\begin{remark}[How to verify the inversion assumptions]
On any angular truncation on which $A$ is bounded away from zero and all
coefficients are bounded, ordinary algebraic inversion of
\cref{eq:common-stat-expansion} gives
\[
 q_t^{-1}=\frac At+
 \sum_{j:\lambda_j=\delta}c_jA^{1+\delta}t^{-1-\delta}
 +o(t^{-1-\delta}).
\]
\cref{eq:common-inverse-L1,eq:common-remainder-L1} are exactly the
uniform-integrability
conditions needed to remove that truncation.  They can be checked using
moments of $A$, $c_jA^{1+\delta}$, and $A^{1+\rho}$ together with a
negligible bound for the tail event outside the truncation.  Exponential
angular tails verify them for positive Clayton; Gaussian moment and local
integrability bounds verify them for standard multivariate $t$ when
$\nu>1$.  This explicit formulation replaces the ambiguous phrase
``assume dominated tail inversion.''
\end{remark}

The theorem immediately separates bounded angular effects in
reciprocal/harmonic-mean scores from the intrinsic third-order expansion
of the positive Half-Cauchy map.

\begin{corollary}[Reciprocal/harmonic-mean common-factor phase diagram]
\label{cor:common-reciprocal-phase}
Suppose \cref{eq:common-radial-tail} holds and, under the dominated
inversion conditions of \cref{thm:root-inversion},
\begin{equation}\label{eq:common-reciprocal-stat}
 S_Y(q,\Theta)=qA+q^{1-\kappa}B+C
  +o(q^{1-\kappa}+1).
\end{equation}
Then, with $\delta=\min(\rho,\kappa,1)$,
\begin{align}
 \Pp\{S_Y(Q,\Theta)>t\}
 &=\frac{a\E A}{t}+K_Yt^{-1-\delta}+o(t^{-1-\delta}),
 \label{eq:common-reciprocal-result}\\
 K_Y
 &=a\mathbf1\{\kappa=\delta\}\E(A^\kappa B)
   +a\mathbf1\{1=\delta\}\E(AC)
   -ab\mathbf1\{\rho=\delta\}\E(A^{1+\rho}).
 \label{eq:common-reciprocal-coefficient}
\end{align}
Terms tie and add.  In particular, when $\rho=\kappa$, the common-factor
correction is $a\E\{A^\kappa(B-bA)\}$ provided
$\rho=\kappa<1$; at $\rho=\kappa=1$ the bounded angular coefficient
$a\E(AC)$ also ties.  A nonzero bounded angular edge dominates the
common-factor terms when $\rho=\kappa>1$ provided its projected coefficient
$\E(AC)$ does not cancel.

If $C\equiv0$ and the sharper expansion
\begin{equation}\label{eq:common-reciprocal-zero-C}
 S_Y(q,\Theta)=qA+q^{1-\kappa}B+o(q^{1-\kappa})
\end{equation}
holds under the dominated-inversion conditions of
\cref{thm:root-inversion} at
$\delta_0=\min(\rho,\kappa)$, then the exponent one must be omitted:
\begin{align}
 \Pp\{S_Y(Q,\Theta)>t\}
 &=\frac{a\E A}{t}+K_0t^{-1-\delta_0}
   +o(t^{-1-\delta_0}),\label{eq:common-reciprocal-zero-C-result}\\
 K_0
 &=a\mathbf1\{\kappa=\delta_0\}\E(A^\kappa B)
   -ab\mathbf1\{\rho=\delta_0\}\E(A^{1+\rho}).
 \label{eq:common-reciprocal-zero-C-coefficient}
\end{align}
In particular, if $\rho=\kappa$ at any positive value, then
$K_0=a\E\{A^\kappa(B-bA)\}$.
\end{corollary}

\begin{corollary}[Positive Half-Cauchy common-factor phase diagram]
\label{cor:common-hc-phase}
Suppose \cref{eq:common-radial-tail} holds with $\rho=\kappa$ and,
under the dominated inversion conditions of
\cref{thm:root-inversion},
\begin{equation}\label{eq:common-hc-stat}
 S_H(q,\Theta)=cqA+cq^{1-\kappa}B-dq^{-1}D
 +o(q^{1-\kappa}+q^{-1}),
 \qquad c=\frac2\pi,\quad d=\frac\pi6.
\end{equation}
This hypothesis deliberately excludes a nonzero $O(1)$ angular term.  If
such a term is present, it enters at relative exponent one and must be
included by applying \cref{thm:root-inversion} directly.  The applications
below of this corollary have no such term.
Then
\begin{equation}\label{eq:common-hc-result}
 \Pp\{S_H(Q,\Theta)>t\}
 =\frac{ac\E A}{t}+K_Ht^{-1-\min(\kappa,2)}
 +o\{t^{-1-\min(\kappa,2)}\},
\end{equation}
where
\begin{equation}\label{eq:common-hc-coefficient}
K_H=
\begin{cases}
 ac^{1+\kappa}\E\{A^\kappa(B-bA)\},&0<\kappa<2,\\[2pt]
 ac^3\E\{A^2(B-bA)\}-adc^2\E(A^2D),&\kappa=2,\\[2pt]
 -adc^2\E(A^2D),&\kappa>2.
\end{cases}
\end{equation}
Consequently, the dependence/radial perturbation dominates for
$\kappa<2$, collides with the Half-Cauchy transformation at $\kappa=2$,
and is hidden behind the universal $t^{-3}$ transformation scale for
$\kappa>2$.  The conclusion at $\kappa=2$ requires
$\E(A^2|D|)<\infty$; failure of this angular moment is a critical-face
problem rather than a failure of the radial expansion.
\end{corollary}

}

\CommonFactorTheory

\subsection{Standard proper-face constructions and an exact max-linear example}

Genuine proper faces first appear when $m\ge3$.  In two dimensions the
simplex has only its two axis vertices and its open interior, so there is
no intermediate coordinate-face dimension.  Several standard
multivariate extreme-value constructions can be parameterized to produce
pure proper-face support:
\begin{enumerate}[leftmargin=2.2em]
 \item \emph{Marshall--Olkin and max-linear common-shock copulas}
 \cite{MarshallOlkin1967}.
 Independent shocks indexed by nonempty subsets $J\subseteq[m]$ generate
 spectral atoms on rays contained in $F_J^\circ$.  Shocks affecting pairs
 or larger proper blocks can therefore concentrate the exponent measure on
 selected proper faces.  For small-$p$ analysis one uses the
 survival/reflected orientation that maps large max-stable variables to
 small uniforms.

 \item \emph{Tawn asymmetric logistic extreme-value copulas}
 \cite{Tawn1990}.
 Their subset components can allocate diffuse spectral mass to selected
 coordinate faces.  Retaining only components indexed by proper subsets
 produces a proper-face-supported model; adding axes or a full-set
 component moves the model into the mixed category below.
\end{enumerate}
These examples concern support in the tail orientation relevant to the
chosen extremes.  Reflecting a copula can move upper-tail face structure
to the lower tail of the $p$-values and must be stated explicitly.

A concrete three-dimensional max-linear example makes the geometry exact.
Let $Z_{12},Z_{13},Z_{23}$ be independent unit-Fr\'echet variables and set
\begin{equation}\label{eq:pair-shock-max-linear}
 X_1=\max(Z_{12}/2,Z_{13}/2),\quad
 X_2=\max(Z_{12}/2,Z_{23}/2),\quad
 X_3=\max(Z_{13}/2,Z_{23}/2).
\end{equation}
Every margin is unit Fr\'echet and
\begin{equation}\label{eq:pair-shock-spectral-measure}
 S=\delta_{(1/2,1/2,0)}
  +\delta_{(1/2,0,1/2)}
  +\delta_{(0,1/2,1/2)}.
\end{equation}
Thus the first-order exponent measure is concentrated exactly on the three
pair-edge rays.  Pairwise extremes have order $t^{-1}$, while a three-way
extreme requires at least two independent shocks and is of smaller order.
The exact-uniform transformation
\begin{equation}\label{eq:max-linear-to-p-copula}
 U_i=1-\exp(-1/X_i)
\end{equation}
makes small $U_i$ correspond to large $X_i$, and $1/U_i\sim X_i$; hence
the reciprocal score vector retains the same edge-supported exponent
geometry.  By \cref{prop:equiv}, the positive Half-Cauchy score vector has
the identical support after the deterministic radial rescaling by $c$.

\subsection{Three-edge critical integral}

\begin{lemma}[Three-edge critical integral]
\label{lem:three-edge-critical-integral}
Let $R_{12},R_{13},R_{23}$ be independent copies of a nonnegative
continuous random variable with distribution function $F_R$ and density
$f_R$.  Suppose, for some $a>0$,
\begin{equation}\label{eq:edge-shock-tail-assumption}
 1-F_R(x)=\frac{a}{x}+O(x^{-2}),
 \qquad
 f_R(x)=\frac{a}{x^2}+O(x^{-3}).
\end{equation}
For fixed positive weights satisfying $\sum_iw_i=1$, put
\[
 G_{\bw}
 =w_1\max(R_{12},R_{13})
  +w_2\max(R_{12},R_{23})
  +w_3\max(R_{13},R_{23}).
\]
Then
\begin{equation}\label{eq:three-edge-critical-expansion}
 \Pp(G_{\bw}>t)
 =\frac{2a}{t}
  +4a^2\sigma_2(\bw)\frac{\log t}{t^2}
  +O(t^{-2}).
\end{equation}
\end{lemma}

\subsection{Higher-order Gaussian tails and the independence boundary}
\label{sec:gaussian-higher-order}

Ordinary MRV retains only the axis-supported first-order tail measure.  It
therefore discards the smaller probability that two Gaussian-copula scores
are simultaneously large.  Hidden regular variation recovers this
probability.  The following bivariate calculation identifies its order; see
Fung and Seneta~\cite{FungSeneta2011} and the multivariate heavy-tailed
Gaussian-copula analysis of Das and Fasen-Hartmann~\cite{DasFasen2024}.

Let $(Z_1,Z_2)$ be standard bivariate normal with correlation
$0<\rho<1$, put
\[
  U_i=1-\Phi(Z_i),
  \qquad Y_i=U_i^{-1},
\]
and define
\[
  q_\rho=\frac{2}{1+\rho},
  \qquad
  a_\rho=\frac{1}{1+\rho}.
\]
Then, for an explicit constant $K_\rho>0$,
\begin{equation}\label{eq:gaussian-equal-hidden-tail}
  \Pp(Y_1>t,Y_2>t)
  \sim
  K_\rho\,
  t^{-q_\rho}(\log t)^{-\rho/(1+\rho)}.
\end{equation}
Moreover, for fixed $x_1,x_2>0$, the corresponding hidden-tail ratio has
the form
\begin{equation}\label{eq:gaussian-hidden-ratio}
  \frac{\Pp(Y_1>tx_1,Y_2>tx_2)}
       {\Pp(Y_1>t,Y_2>t)}
  \longrightarrow
  (x_1x_2)^{-a_\rho}.
\end{equation}
Equivalently, when the deterministic scaling is taken to be
$t^{-q_\rho}(\log t)^{-\rho/(1+\rho)}$, the hidden measure
$\nu_{\mathrm{hid}}$ has survival function
$K_\rho(x_1x_2)^{-a_\rho}$ and, in the interior quadrant, density
proportional to
\begin{equation}\label{eq:gaussian-hidden-density}
  a_\rho^2x_1^{-a_\rho-1}x_2^{-a_\rho-1}.
\end{equation}

\cref{eq:gaussian-hidden-ratio,eq:gaussian-hidden-density} identify the
pair hidden measure.  Its
M\"obius half-space increment is finite for every fixed $\rho>0$, because
$a_\rho<1$.  The conditional fixed-dimensional theorem below transfers
this measure and sums all dominant pairs when explicit boundary and
truncation conditions hold; those conditions are not verified here.
Since $q_\rho<2$, the hidden rectangle scale is larger than
$(\log t)/t^2$ and already rules out the conventional Level-II
measure-rate target.  Under the stated transfer conditions it also becomes
the aggregate correction recorded in \cref{tab:copula-hierarchy}.

It is tempting, but incorrect, to set $\rho=0$ in the fixed-$\rho>0$
formula and expect the independent answer.  If $\rho=0$, the
comparable-extremes part of
\cref{eq:gaussian-equal-hidden-tail} is only of order $t^{-2}$.  The larger
independent remainder instead comes from edge regions in which one score is
of order $t$ while the other ranges over all intermediate scales:
\begin{equation}\label{eq:gaussian-independent-edge-log}
  \frac1{t^2}\int_1^t\frac{dy}{y}
  =\frac{\log t}{t^2}.
\end{equation}
The hidden-density calculation makes the nonuniformity explicit.  Near an
axis, the width of the weighted-sum shell in $x_1$ is proportional to
$x_2$.  Combining this width with \cref{eq:gaussian-hidden-density}
produces an integral of the form
\begin{equation}\label{eq:gaussian-axis-integral}
  \int_0^1y^{-a_\rho}\,dy
  =\frac{1}{1-a_\rho}
  =\frac{1+\rho}{\rho},
  \qquad \rho>0.
\end{equation}
It is finite for every fixed positive $\rho$ but diverges as
$\rho\downarrow0$.  Keeping the finite-$t$ cutoff gives the useful
transition factor
\begin{equation}\label{eq:gaussian-transition-factor}
  \int_{1/t}^1y^{-a_\rho}\,dy
  =\frac{1-t^{-(1-a_\rho)}}{1-a_\rho}.
\end{equation}
Consequently,
\begin{equation}\label{eq:gaussian-transition-regimes}
  \frac{1-t^{-(1-a_\rho)}}{1-a_\rho}
  \sim
  \begin{cases}
    (1-a_\rho)^{-1},
      &(1-a_\rho)\log t\longrightarrow\infty,\\[3pt]
    \log t,
      &(1-a_\rho)\log t\longrightarrow0.
  \end{cases}
\end{equation}
Since $1-a_\rho\sim\rho$, the crossover occurs on the scale
$\rho\log t=O(1)$.  In particular,
\[
  \lim_{\rho\downarrow0}\lim_{t\to\infty}
  \quad\hbox{and}\quad
  \lim_{t\to\infty}\lim_{\rho\downarrow0}
\]
are different asymptotic operations.  The fixed-$\rho>0$ expansion is not
uniform at the independence boundary.

For one-sided tests with $-1<\rho<0$, the comparable lower--lower exponent
$2/(1+\rho)$ exceeds two, so that particular joint-extreme term is smaller
than $t^{-2}$.  It does not by itself identify the leading weighted-sum
remainder, which may instead be governed by an edge contribution and needs
a separate calculation.  For two-sided tests, the four sign quadrants
include a dominant positively correlated quadrant, and the corresponding
pairwise exponent uses $|\rho|$:
\[
  q_{|\rho|}=\frac{2}{1+|\rho|}<2
  \qquad(\rho\ne0).
\]

\begin{remark}[Relation to the Gaussian-copula literature]
The bivariate normal-copula regular-variation input is classical; the
general cone geometry is part of the hidden-regular-variation literature
for Gaussian-copula risks.  The calculation here identifies the explicit
face-increment coefficient for the weighted reciprocal/harmonic-mean
half-space, the sum over all maximally correlated pairs, and the
corresponding positive Half-Cauchy transfer under stated boundary
conditions.  This is more specialized than general Gaussian portfolio-tail
theory and does not discover Gaussian hidden regular variation.
\end{remark}

\begin{openproblem}[One-sided nonpositive Gaussian dependence]
Suppose every active one-sided correlation is nonpositive and at least one
is strictly negative.  Then the positive-pair
power term in \cref{thm:gaussian-second-order-m} is absent.  For both
$R_{\bw}$ and $T_{\bw}$ the leading dependence correction can be an edge
term, with independence as a critical logarithmic boundary; $T_{\bw}$
also has its intrinsic score-map scale.  A uniform theorem covering mixed
zero and negative correlations remains to be written.  The all-zero case
is the product copula and is already settled by
\cref{thm:independent-sharp}.
\end{openproblem}

\subsection{Conditional fixed-dimensional Gaussian expansion}

The preceding bivariate calculation identifies the candidate pair
measures and coefficients, but an arbitrary-dimensional weighted-half-space
theorem also needs uniform control near every proper subface.  The next
result states that missing boundary requirement explicitly.  Let
$A=\{i:w_i>0\}$ and, for one-sided tests, define
\[
 \rho_*=\max_{i<j,\ i,j\in A}\rho_{ij},
 \qquad
 E_*=\{(i,j):i<j,\ i,j\in A,\ \rho_{ij}=\rho_*\}.
\]
For $0<\rho<1$, put
\begin{equation}\label{eq:gaussian-constants}
 a_\rho=\frac1{1+\rho},\quad q_\rho=2a_\rho,\quad
 \beta_\rho=1-a_\rho,
\end{equation}
\begin{equation}\label{eq:gaussian-KC}
 K_\rho=\frac{(1+\rho)^2}{\sqrt{1-\rho^2}}
          (4\pi)^{a_\rho-1},
 \qquad
 C(a)=-\frac{\Gamma(1-a)^2}{\Gamma(1-2a)}.
\end{equation}
Notice that $C(a)>0$ for $1/2<a<1$.

\begin{theorem}[Gaussian aggregation under face-transfer conditions]
\label{thm:gaussian-second-order-m}
Let $m<\infty$, let $R=(\rho_{ij})$ be a positive-definite correlation
matrix, and let $Y_i=\{1-\Phi(Z_i)\}^{-1}$ for
$\boldsymbol Z\sim N_m(0,R)$.  Suppose $\rho_*>0$ and put
\[
 r_*(t)=t^{-q_{\rho_*}}(\log t)^{-\beta_{\rho_*}}.
\]
Assume that every $(i,j)\in E_*$ satisfies the uniform-integrability
condition \cref{eq:hidden-face-ui} at scale $r_*(t)$, and that the remaining
M\"obius faces satisfy
\begin{equation}\label{eq:gaussian-face-remainder-condition}
 \sum_{\substack{i<j\\(i,j)\notin E_*}}
   \left|\E G_{\{i,j\}}(Y_i/t,Y_j/t)\right|
 +\sum_{\substack{I\subseteq A\\|I|\geq3}}
   \left|\E G_I(\boldsymbol Y_I/t)\right|
 =o\{r_*(t)\}.
\end{equation}
Then
\begin{align}
 \Pp\!\left(\sum_{i=1}^mw_iY_i>t\right)
 &=\frac1t+D_{G,\boldsymbol w}
 t^{-q_{\rho_*}}(\log t)^{-\beta_{\rho_*}}
 +o\!\left(t^{-q_{\rho_*}}(\log t)^{-\beta_{\rho_*}}\right),
 \label{eq:gaussian-m-pareto}\\
 D_{G,\boldsymbol w}
 &=K_{\rho_*}C(a_{\rho_*})
   \sum_{(i,j)\in E_*}(w_iw_j)^{a_{\rho_*}}.
 \label{eq:gaussian-m-coefficient}
\end{align}
For the Half-Cauchy scores and $c=2/\pi$,
\begin{equation}\label{eq:gaussian-m-hc}
 \Pp\!\left(\sum_iw_iH_i>t\right)
 =\frac ct+c^{q_{\rho_*}}D_{G,\boldsymbol w}
 t^{-q_{\rho_*}}(\log t)^{-\beta_{\rho_*}}
 +o\!\left(t^{-q_{\rho_*}}(\log t)^{-\beta_{\rho_*}}\right).
\end{equation}

For two-sided $p$-values $U_i=2\{1-\Phi(|Z_i|)\}$, replace
$\rho_*$ by
\[
 \rho_*^{\mathrm{abs}}=\max_{i<j,\ i,j\in A}|\rho_{ij}|,
\]
replace $E_*$ by the pairs attaining that maximum, and replace
$K_{\rho_*}$ by $2^{1-q_{\rho_*}}K_{\rho_*}$.  The conclusions hold when
$\rho_*^{\mathrm{abs}}>0$ and the analogous two-sided conditions hold.
\end{theorem}

\begin{remark}[Status of the Gaussian boundary conditions]
The bivariate normal tail calculation proves the hidden vague limit on
compact subsets of the open pair face.  What is not proved here is the
uniform truncation at the adjacent axes and, in dimension at least three,
the full estimate \cref{eq:gaussian-face-remainder-condition}.  A complete
unconditional fixed-$m$ theorem requires a stratified Gaussian boundary
lemma controlling dominant pair shells, nonpositive pairs, and all higher
M\"obius faces.  The displayed coefficient is therefore a conditional
weighted-half-space conclusion, not an unconditional Gaussian model
theorem.

The expansion is also pointwise in the correlation parameters.  It is not
uniform as $\rho_*\uparrow1$: then $q_{\rho_*}\downarrow1$, and the formal
relative correction approaches $(\log t)^{-1/2}$.  For every fixed
$\rho_*<1$ the relative correction remains polynomial, with exponent
tending to zero as $\rho_*\uparrow1$; only the boundary scale is logarithmic.
\end{remark}

The bivariate hidden regular variation used above is classical.  The
contribution here is the weighted-half-space face coefficient and its
positive Half-Cauchy transfer under explicit boundary hypotheses; the
unresolved boundary conditions are stated explicitly above.

\subsection{First-order multivariate-t copula geometry}

Let $\boldsymbol Z$ have a standard $m$-variate $t$ distribution with
common degrees of freedom $\nu>0$ and correlation matrix $R$.  Use the
stochastic representation
\begin{equation}\label{eq:t-representation}
  \boldsymbol Z=R_0\boldsymbol G,
  \qquad
  R_0=\sqrt{\frac{\nu}{W}},
\end{equation}
where $\boldsymbol G\sim N_m(\boldsymbol0,R)$,
$W\sim\chi^2_\nu$, and $\boldsymbol G$ and $W$ are independent.

\begin{lemma}[Regular variation of the multivariate $t$ vector]
\label{lem:t-mrv}
The random vector $\boldsymbol Z$ in \cref{eq:t-representation} is
multivariate regularly varying with index $\nu$.
\end{lemma}

Let $T_\nu$ denote the univariate $t_\nu$ CDF.  Its upper tail satisfies
\begin{equation}\label{eq:t-tail}
  \overline T_\nu(z)=1-T_\nu(z)
  \sim k_\nu z^{-\nu},
  \qquad z\to\infty,
\end{equation}
for a constant $k_\nu>0$.

\begin{theorem}[Multivariate-$t$-copula score vector]
\label{thm:t-copula}
For the one-sided $p$-values
\[
  U_i=1-T_\nu(Z_i),
\]
the vectors $(1/U_1,\ldots,1/U_m)$ and
\[
  \left(\cot(\pi U_1/2),\ldots,\cot(\pi U_m/2)\right)
\]
are multivariate regularly varying with index one.  The same conclusion
holds for the two-sided $p$-values
\[
  U_i=2\{1-T_\nu(|Z_i|)\}.
\]
Consequently, \cref{thm:hcct-mrv} applies to both $R_{\bw}$ and
$T_{\bw}$ in both cases.
\end{theorem}

Unlike the Gaussian case, a multivariate-$t$ copula has positive tail
dependence.  In the bivariate case its tail-dependence coefficient is
\begin{equation}\label{eq:t-lambda}
  \lambda_t
  =2T_{\nu+1}\!\left(
      -\sqrt{\frac{(\nu+1)(1-\rho)}{1+\rho}}
    \right)>0,
  \qquad -1<\rho<1;
\end{equation}
see Demarta and McNeil~\cite{DemartaMcNeil2005}.  Its limiting spectral
measure consequently has interior mass.  This is why the multivariate-$t$
case can satisfy the MRV theorem while failing a pairwise
$o(s)$ tail-independence condition.

\begin{remark}[Scope of the $t$ proof]
\cref{thm:t-copula} concerns the standard multivariate-$t$ model with
a common random scale, equivalently the usual multivariate-$t$ copula.  It
does not automatically cover a collection of marginal Student statistics
having different, data-dependent denominators.  The joint copula of those
statistics must be analyzed separately.
\end{remark}

\subsection{Radial source of the standard-\texorpdfstring{$t$}{t}
second-order rate}
\label{sec:t-higher-order}

The standard multivariate-$t$ copula is asymptotically dependent, so its
higher-order mechanism is not Gaussian hidden regular variation.  It comes
from second-order regular variation of the common radial multiplier in
\cref{eq:t-representation}, together with the second-order accuracy of the
marginal score transformation.  General second-order elliptical MRV and the
multivariate-$t$ example are treated by Kim~\cite{Kim2024}.

The lower-tail expansion of the $\chi^2_\nu$ distribution gives, directly,
\begin{equation}\label{eq:t-radial-second-order}
\begin{split}
  \Pp(R_0>r)
  &=\frac{\nu^{\nu/2}}
       {2^{\nu/2}\Gamma(\nu/2+1)}r^{-\nu}
    \left\{
      1-\frac{\nu^2}{2(\nu+2)}r^{-2}
      +O(r^{-4})
    \right\}.
\end{split}
\end{equation}
The leading score map is homogeneous of degree $\nu$ by
\cref{thm:t-copula}.  Therefore a score threshold of order $t$
corresponds to a radial threshold of order $t^{1/\nu}$, and the $r^{-2}$
relative correction in \cref{eq:t-radial-second-order} becomes the
score-scale exponent-measure rate
\begin{equation}\label{eq:t-score-second-order-rate}
  A(t)\asymp t^{-2/\nu}.
\end{equation}
The exact $t$-CDF transformation and the signed Gaussian angles also
contribute.  Exact Pareto margins remove pure one-coordinate corrections,
but they do not remove an interior perturbation of the exponent measure.
The theorem in \cref{thm:t-second-order-formal} performs the required inversion and shows how
the radial scale competes with the bounded angular term at $\nu=2$ and
with the Half-Cauchy score-map scale at $\nu=1$.

\subsection{Coefficient control and the angularly critical standard-t boundary}

\begin{proposition}[Dimension-free coefficient control]
\label{prop:t-coefficient-uniform}
For fixed $\nu>1$, the absolute values of
$D_{\nu,\boldsymbol w}^{+}$ and
$D_{\nu,\boldsymbol w}^{|\cdot|}$ are bounded by a constant depending only
on $\nu$, uniformly over $m$, all correlation matrices, and all simplex
weights.  The additional one-sided angular coefficient
$\E(A_+C_-)$ is bounded uniformly as well.
\end{proposition}

\begin{remark}[The unresolved critical and subcritical cases]
At $\nu=1$, $|G_i|^{\nu-2}$ is at the angular integrability boundary and
the radial correction meets the $t^{-3}$ Half-Cauchy term.  For
$0<\nu<1$, proper angular subfaces can dominate.  These cases require the
recursive nested-face theory discussed in
\cref{sec:main-novelty-boundary}; no theorem for them is claimed here.
\end{remark}

\begin{remark}[Fixed-dimensional scope]
The expansions in this subsection are fixed-dimensional and concern
$t\to\infty$, equivalently $\alpha\downarrow0$.
\cref{prop:t-coefficient-uniform} controls their displayed coefficients,
not their remainders, and therefore does not provide the uniformity needed
for growing $m$.
\end{remark}

\subsection{Clayton versus multivariate \texorpdfstring{$t$}{t}: the broader class}

Both models are concrete instances of
\cref{thm:root-inversion} after reciprocal score transformation.
Their angular laws and second-order parameters differ:
\begin{table}[!ht]
\caption{Common-factor comparison of positive Clayton and standard
multivariate-$t$ copulas.}
\label{tab:common-factor-comparison}
\centering
\begin{tabular}{@{}lll@{}}
\toprule
& positive Clayton & standard multivariate $t_\nu$\\
\midrule
score-scale common factor & $V^{-1/\theta}$ & $(\nu/W)^{1/2}$\\
score homogeneity & one after frailty rewrite & degree $\nu$\\
transformed tail index & one & one\\
score-scale second-order $\kappa$ & $\theta$ & $2/\nu$\\
angular variables & positive exponentials & signed Gaussian powers\\
first-order spectral mass & interior & two-sided interior; one-sided mixed\\
angular critical boundary & none for first correction & $\nu=1$\\
\bottomrule
\end{tabular}
\end{table}
Thus \cref{tab:common-factor-comparison} shows that $\theta=2/\nu$ is an
equal-rate surface, not an equality of copula
families.  Across the broader common-factor class, $\kappa<2$ means the
dependence/radial correction dominates the intrinsic $t^{-3}$ Half-Cauchy
term, $\kappa=2$ is the collision point, and $\kappa>2$ makes the
Half-Cauchy transform dominate.  This gives the Clayton transition
$\theta=2$ and the formal $t$ transition $\nu=1$.  The latter is more
singular because the Gaussian angular inverse moments also become critical.
Thus \cref{cor:common-hc-phase} is a genuine shared theorem, while
the verification of its dominated-inversion assumptions remains
model-specific.

The class overlaps with Gaussian hidden regular variation only at the level
of a generic face expansion, not at the first-order geometry.  Gaussian
copulas with fixed correlations strictly below one have axis-supported
first-order MRV; their first dependence correction lives on hidden pair
faces.  Positive Clayton and two-sided $t$ scores are interior-supported.
One-sided $t$ scores are mixed-supported but still have interior mass.  In
both $t$ orientations, simultaneous extremes are already present at first
order, so their next common-factor term perturbs the existing measure
rather than marking the first appearance of an interior face.

\begin{proposition}[Uniform Clayton coefficient bounds]
\label{prop:clayton-coefficient-uniform}
For fixed $\theta>0$,
\[
 0\le D_{\theta,\boldsymbol w}
 \le 1
\]
uniformly over $m$ and all simplex weights.  For fixed $\theta>1$, the same
uniform bound holds for $\E(A_{\boldsymbol w}^2Q_{\boldsymbol w})$.
\end{proposition}

\subsection{Why fixed-dimensional MRV does not solve growing
\texorpdfstring{$m$}{m}}
\label{sec:growing-m}

\cref{def:mrv} is posed on a fixed state space $\mathbb E_m$.
When $m=m(t)\to\infty$, both the state space and the dependence law change
with $t$.  Row-by-row MRV gives a limit only after holding the row dimension
fixed; it gives no uniform control at the diagonal point where the tail
threshold and dimension grow together.

For the existing one-big-jump proof, the relevant remainder has the form
\begin{equation}\label{eq:growing-pair}
  m_t^2R_t,
  \qquad
  R_t=\max_{i\ne j}
  \Pp\!\left(
    p_{i,t}<\frac{2w_{i,t}m_t}{\pi t},
    p_{j,t}<\frac{2w_{j,t}m_t}{\pi\delta_t t}
  \right).
\end{equation}
It must be shown uniformly that
\begin{equation}\label{eq:growing-required}
  m_t^2R_t=o(t^{-1}).
\end{equation}
A pointwise $o(t^{-1})$ statement for each pair cannot be summed over a
growing collection of pairs.  Likewise, for
$\boldsymbol X\in\{\boldsymbol Y,\boldsymbol H\}$, fixed-dimensional
convergence
\[
  t\Pp(\boldsymbol X/t\in A)\to\mu_X(A)
\]
does not provide a convergence rate uniform in $m$ or in the weight vector.

A genuine growing-dimensional MRV theory needs a triangular-array
condition that controls the tail-measure error uniformly over the moving
half-spaces
\[
  \{\boldsymbol x:\boldsymbol w_t^{\mathsf T}\boldsymbol x>1\}.
\]
This is additional information, not a consequence of row-by-row MRV.

Liu, Meng and Pillai~\cite{LiuMengPillai2025} give a rigorous first-order
one-big-jump route for both positive Half-Cauchy and reciprocal aggregation.
Its master assumption is \cref{eq:growing-required}, together with
$\delta_t\downarrow0$ and $t\delta_t\to\infty$.  This proves calibration
against the single Half-Cauchy or Pareto tail.  The further comparison with
the exact finite-$m$ independence reference uses
\[
 \max_iw_{i,t}=O(m_t^{-1}),\qquad m_t=o(\sqrt t).
\]
A polynomial corollary follows if the pair event is
$o\{t^{-(1+\gamma)}\}$ uniformly and
$m_t=O(t^{\gamma/2})$.  For pairwise Gaussian $p$-values with
$|\rho_{ij}|\le\rho_{\max}<1$, put
\[
 \gamma_0=\frac{1-\rho_{\max}}{1+\rho_{\max}}.
\]
The proved general boundary is $m_t=o(t^{\gamma_0/2})$; when
$0<\rho_{\max}<1$, the Gaussian logarithmic factor sharpens this to
$m_t=O(t^{\gamma_0/2})$.  If $\rho_{\max}=0$ and only pairwise Gaussianity
is assumed, the elementary argument retains the strict endpoint
$m_t=o(\sqrt t)$.

Those one-big-jump results are sufficient, not necessary.  They are
first-order and pair-accumulation based; they are naturally adapted to
axis-supported arrays and exclude important interior-supported examples
such as standard multivariate $t$, even though those examples can remain
calibrated for arbitrary dimension by a common-factor argument.

\subsection{A precise uniform-MRV theorem for growing dimension}

Growing-dimensional results require explicit uniformity.  The next theorem
states exactly what is needed.

\begin{theorem}[Triangular-array aggregation under uniform MRV]
\label{thm:uniform-mrv}
Choose $X\in\{Y,H\}$, put $c_Y=1$ and $c_H=2/\pi$, and let
$t_n\to\infty$ and $m_n\geq1$.  For row $n$, let
$\boldsymbol X_n\in[0,\infty)^{m_n}$ be the reciprocal score vector with
exact unit-Pareto margins when $X=Y$, or the positive Half-Cauchy score
vector when $X=H$.  Let $\boldsymbol w_n$ be nonnegative with
$\sum_{i=1}^{m_n}w_{i,n}=1$.  Suppose there is an index-one homogeneous
Radon measure $\mu_n$ on $\mathbb E_{m_n}$ satisfying
\begin{equation}\label{eq:uniform-margins}
  \mu_n\{\boldsymbol x:x_i>1\}=c_X
  \qquad(i=1,\ldots,m_n).
\end{equation}
Define the moving half-space
\[
  A_n=\{\boldsymbol x:\boldsymbol w_n^{\mathsf T}\boldsymbol x>1\}.
\]
If the triangular array satisfies the uniform tail approximation
\begin{equation}\label{eq:uniform-mrv-assumption}
  \left|
    t_n\Pp(\boldsymbol X_n/t_n\in A_n)-\mu_n(A_n)
  \right|\longrightarrow0,
\end{equation}
then
\begin{equation}\label{eq:uniform-mrv-conclusion}
  t_n\Pp(\boldsymbol w_n^{\mathsf T}\boldsymbol X_n>t_n)
  \longrightarrow c_X.
\end{equation}
There is no intrinsic restriction on the numerical growth rate of $m_n$;
the permissible rate is determined by the uniform approximation
\cref{eq:uniform-mrv-assumption}.
\end{theorem}

\subsection{Metrics and radial kernels that verify uniform MRV}

The half-space assumption in \cref{thm:uniform-mrv} can be checked
either by a measure metric or, more sharply for common-factor models, by a
one-dimensional conditional radial tail.

\begin{theorem}[Bounded-Lipschitz control of moving half-spaces]
\label{thm:bl-moving-halfspace}
Let $\nu_{n,t}=t\Pp(\boldsymbol X_n/t\in\cdot)$ and let $\mu_n$ be an
index-one homogeneous measure with marginal constants $c$.  Suppose that,
for functions supported on $\{\|\boldsymbol x\|_\infty\ge1/2\}$,
\begin{equation}\label{eq:bl-discrepancy}
 \left|\int f\,d(\nu_{n,t}-\mu_n)\right|
 \le\delta_n(t)\{\|f\|_\infty+\operatorname{Lip}(f)\}.
\end{equation}
Then, for $0<\eta\le1/2$,
\begin{equation}\label{eq:bl-halfspace-bound}
 \sup_{\boldsymbol w\in\Delta_{m_n-1}}
 \left|t\Pp(\boldsymbol w^{\mathsf T}\boldsymbol X_n>t)-c\right|
 \le\delta_n(t)(1+\eta^{-1})+\frac{c\eta}{1-\eta}.
\end{equation}
In particular, choosing $\eta=\sqrt{\delta_n(t)}$ proves uniform weighted
projection convergence whenever $\delta_n(t)\to0$.
\end{theorem}

\begin{theorem}[Exact radial-angular error identity]
\label{thm:radial-angular-uniform}
For row $n$, suppose
$\boldsymbol X_n=R_n\boldsymbol\Theta_n$, where
$\boldsymbol\Theta_n$ lies on the positive $\ell_1$ simplex and is
independent of $R_n$.  Assume, for $r$ large,
\[
 \Pp(R_n>r)=\frac{C_n}{r}\{1+e_n(r)\},
 \qquad C_n\E\Theta_{i,n}=c\quad\hbox{for every }i.
\]
Writing $L_{n,\boldsymbol w}=\boldsymbol w^{\mathsf T}
\boldsymbol\Theta_n$, one has, for every sufficiently large $t$, the exact
identity
\begin{equation}\label{eq:radial-angular-error}
 t\Pp(\boldsymbol w^{\mathsf T}\boldsymbol X_n>t)-c
 =C_n\E\!\left[L_{n,\boldsymbol w}
 e_n\{t/L_{n,\boldsymbol w}\}\right].
\end{equation}
The integrand is defined to be zero on
$\{L_{n,\boldsymbol w}=0\}$.
Hence $\sup_{r\ge t}|e_n(r)|\to0$ is sufficient, uniformly in dimension
and weights.

More generally, independence may be replaced by the conditional kernel
\[
 \Pp(R_n>r\mid\boldsymbol\Theta_n=\boldsymbol s)
 =\frac{k_n(\boldsymbol s)}r
 \{1+e_n(r,\boldsymbol s)\},
 \quad \E\{k_n(\boldsymbol\Theta_n)\Theta_{i,n}\}=c.
\]
Then the right side of \cref{eq:radial-angular-error} becomes
\[
 \E\!\left[k_n(\boldsymbol\Theta_n)L_{n,\boldsymbol w}
 e_n\{t/L_{n,\boldsymbol w},\boldsymbol\Theta_n\}\right].
\]
\end{theorem}

\subsection{Arbitrary-dimensional first-order theorems for the common-factor
families}

The only uniform product-tail input needed below is recorded explicitly.
This avoids inferring uniformity from the ordinary fixed-law form of
Breiman's lemma.

\begin{lemma}[Uniform index-one product tail]
\label{lem:uniform-breiman-one}
Let $Q\geq0$ satisfy
\[
  \Pp(Q>q)=a q^{-1}\{1+o(1)\},
  \qquad q\to\infty,
\]
and let $S_n\geq0$ be independent of $Q$.  If, for some $\delta>0$,
\[
  \sup_n\E S_n^{1+\delta}<\infty,
\]
then, for every sequence $t_n\to\infty$,
\begin{equation}\label{eq:uniform-breiman-one}
  t_n\Pp(QS_n>t_n)-a\E S_n\longrightarrow0.
\end{equation}
The conclusion remains valid for triangular arrays; the law and dimension
of $S_n$ may vary.
\end{lemma}

\begin{theorem}[Growing-dimensional standard multivariate-$t$ aggregation]
\label{thm:t-growing-arbitrary}
Fix $\nu>0$.  Let $m_n$ be arbitrary, let $R_n$ be arbitrary
$m_n$-dimensional correlation matrices, let
$\boldsymbol Z_n\sim t_\nu(0,R_n)$ be the standard common-scale
multivariate-$t$ vector, and let $\boldsymbol w_n$ be arbitrary simplex
weights.  For either one- or two-sided exact $t_\nu$ $p$-values and any
$t_n\to\infty$,
\begin{align}
 t_n\Pp\!\left(\sum_iw_{i,n}Y_{i,n}>t_n\right)&\longrightarrow1,
 \label{eq:t-growing-Y}\\
 t_n\Pp\!\left(\sum_iw_{i,n}H_{i,n}>t_n\right)&\longrightarrow\frac2\pi.
 \label{eq:t-growing-H}
\end{align}
There is no restriction on the rate at which $m_n$ grows.
\end{theorem}

\begin{theorem}[Growing-dimensional positive Clayton aggregation]
\label{thm:clayton-growing-arbitrary}
Fix $\theta>0$.  For arbitrary $m_n$, arbitrary simplex weights, and
$m_n$-variate Clayton copulas with the common fixed parameter $\theta$,
the conclusions \cref{eq:t-growing-Y,eq:t-growing-H} hold for
every $t_n\to\infty$, again with no restriction on the growth of $m_n$.
\end{theorem}

These two theorems are MRV-first generalizations beyond a pairwise
tail-independence argument: they allow interior spectral mass and thus
cover genuine tail dependence.  They complement, rather than duplicate,
the axis-oriented one-big-jump theory of Liu, Meng and
Pillai~\cite{LiuMengPillai2025}.

\begin{conjecture}[Uniform second-order common-factor expansions]
\label{conj:uniform-common-factor-second}
For fixed $\theta>0$, the Clayton expansions in
\cref{thm:clayton-second-order,thm:clayton-hc-phase} should hold uniformly
over $m$ and simplex
weights after their remainders are expressed through uniform exponential
truncation bounds.  For fixed $\nu>1$, the same should hold for the
standard multivariate-$t$ expansions uniformly over $m$, weights, and a
suitably nondegenerate class of correlation matrices.
The bounds in
\cref{prop:clayton-coefficient-uniform,prop:t-coefficient-uniform} show
that the displayed coefficients do
not grow with dimension; they do not, by themselves, prove uniformity of
the remainders.
\end{conjecture}

\begin{openproblem}[Growing-dimensional Gaussian hidden faces]
\label{open:gaussian-growing}
\cref{thm:gaussian-second-order-m} is fixed-dimensional.  A growing
version must control uniformly: all pair-tail approximations, accumulation
of triple and higher faces, covariance matrices approaching singularity,
and the independence boundary $\rho\log t=O(1)$.  As a diagnostic only, if
all active maximal correlations equal a fixed $\rho>0$ and the weights are
of order $1/m$, the displayed pair coefficient can be of order
$m^{2\rho/(1+\rho)}$.  Ignoring the still-unproved uniform remainder would
suggest
\[
 m=o\!\left(t^{(1-\rho)/(2\rho)}\sqrt{\log t}\right).
\]
This is not a theorem: higher-face multiplicities and uniform Gaussian
tail errors must first be bounded.
\end{openproblem}

\subsection{Scope of the growing-dimensional results}

The full one-big-jump development of Liu, Meng and
Pillai~\cite{LiuMengPillai2025} answers a different question under a
different sufficient condition.  The concise comparison above is included
to explain why fixed-dimensional MRV cannot simply be diagonalized in $m$.

The two common-factor theorems show that interior spectral mass does not
obstruct growing dimension and that a radial representation can be sharper
than pair counting.  The abstract uniform-MRV, bounded-Lipschitz, and
uniform product-tail results make the required uniformity explicit; they
are supporting tools for these first-order extensions rather than
second-order calibration theorems.

No growing-dimensional second-order theorem is currently proved.  The
dimension-free $t$ and Clayton coefficient bounds do not control their
remainders, and the Gaussian polynomial rate in
\cref{open:gaussian-growing} is only a diagnostic.  Accordingly, growing
$m$ should be presented as a first-order extension and a second-order open
problem, not as part of the main second-order claim.

\newcommand{\CopulaSupportMechanismLedger}{%
\cref{tab:copula-hierarchy} classifies the families discussed in this
paper.  The orientation is always that small $p$-values are the relevant
lower tail.  The second column gives the four support categories; a
diagonal ray lies in the full-simplex interior, whereas ``proper face'' is
reserved for a coordinate face.  The third column gives the absolute
aggregate-tail rate after the universal $t^{-1}$ term.  The fourth
separates proved half-space expansions from mechanism predictions based
only on rectangle or hidden-tail diagnostics; each row states explicitly
whether a sharp coefficient, only a rate bound, or merely a mechanism
diagnosis is available.
Standard formulas for these families can be found in
Nelsen~\cite{Nelsen2006}.  Corner-unbounded densities in several standard
families, and the broader fact that smooth copulas need not have bounded
densities, are discussed by Lalancette and
Zimmerman~\cite{LalancetteZimmerman2022}.

Here ``rate'' means the absolute aggregate-tail correction beyond $c_S/t$,
not the independence-relative distortion.  A candidate rectangle or hidden
scale in a row marked diagnostic or open is not asserted to be a
weighted-half-space rate, and a displayed expansion scale can vanish for
special weights or parameters.

\begingroup
\footnotesize
\setlength{\tabcolsep}{3pt}
\begin{longtable}{@{}p{0.14\textwidth}p{0.16\textwidth}p{0.18\textwidth}p{0.30\textwidth}p{0.14\textwidth}@{}}
\caption{Exponent-measure support, aggregate second-order rate, mechanism,
proof status, and bounded-density diagnostics.}\label{tab:copula-hierarchy}\\
\toprule
Copula and parameters & First-order support & Aggregate second-order rate
& Mechanism and current status & Global bounded density \\
\midrule
\endfirsthead
\toprule
Copula and parameters & First-order support & Aggregate second-order rate
& Mechanism and current status & Global bounded density \\
\midrule
\endhead
\midrule
\multicolumn{5}{r}{Continued on next page}\\
\endfoot
\bottomrule
\endlastfoot

Product (independence)
& Axes.
& $t^{-2}\log t$ for both aggregates; sharp and proved.
& Critical hidden pair faces; the sharp logarithmic coefficient is proved.
& Yes, $c(u,v)=1$. \\

FGM, $-1\leq\theta\leq1$
& Axes.
& $t^{-2}\log t$ for both aggregates when $\theta>-1$; $O(t^{-2})$ at
$\theta=-1$; proved under the stated fixed-dimensional conditions.
& \cref{thm:sharp-local-corner} gives the critical coefficient from
$\lambda=1+\theta$; at $\theta=-1$ the logarithmic coefficient vanishes
and the remainder is $O(t^{-2})$.
& Yes; $c(u,v)=1+\theta(1-2u)(1-2v)$. \\

AMH, $-1\leq\theta<1$
& Axes.
& $t^{-2}\log t$ for both aggregates; sharp bivariately and proved in
higher dimension under the local-corner theorem's higher-face condition.
& The sharp local-corner theorem applies with
$\lambda=(1-\theta)^{-1}$ and gives the logarithmic coefficient.
& Yes for each fixed $\theta<1$; the bound deteriorates as $\theta\uparrow1$.
\\

AMH, $\theta=1$
& Full-interior mass, with
$\Lambda_L(x,y)=xy/(x+y)$.
& $t^{-2}$ for both aggregates; sharp and proved for two positive weights.
& This endpoint is exactly bivariate Clayton with parameter one.  The
ordinary common-frailty/interior perturbation in
\cref{thm:clayton-second-order,thm:clayton-hc-phase} gives coefficients
$w(1-w)$ and $c^2w(1-w)$ for reciprocal and Half-Cauchy aggregation.
& No; the density is unbounded at the lower corner. \\

Frank, finite fixed parameter
& Axes.
& $t^{-2}\log t$ for both aggregates; sharp bivariately and proved in
higher dimension under the local-corner theorem's higher-face condition.
& The sharp local-corner theorem applies with
$\lambda=\theta/(1-e^{-\theta})$, continuously interpreted as one at
$\theta=0$.
& Yes; include independence as the zero-parameter limit. \\

Plackett, $0<\theta<\infty$ fixed
& Axes.
& $t^{-2}\log t$ for both aggregates; sharp bivariately and proved in
higher dimension under the local-corner theorem's higher-face condition.
& The sharp local-corner theorem applies with $\lambda=\theta$ and gives
the critical logarithmic coefficient.
& Yes; the bound need not be uniform as the parameter degenerates. \\

Finite-degree Bernstein, finite checkerboard, and finite mixtures of
bounded densities
& Axes.
& $O(t^{-2}\log t)$ for both aggregates; a sharp coefficient follows when
the local-corner hypotheses hold.
& The sharp local-corner theorem gives the coefficient whenever every pair
density has a Dini-continuous corner limit and higher faces satisfy its
bounded-strip condition; otherwise only the rate is asserted.
& Yes, provided the degree, grid, and component bounds are fixed. \\

One-sided Gaussian, $\rho=0$
& Axes; independence.
& $t^{-2}\log t$ for both aggregates; sharp and proved.
& Critical hidden pair faces; the sharp independent logarithm is proved.
& Yes. \\

One-sided Gaussian, $0<\rho<1$
& Axes.
& $t^{-2/(1+\rho)}$
\newline${}\times(\log t)^{-\rho/(1+\rho)}$ for both aggregates;
the displayed coefficient is conditional on the face-transfer hypotheses
of \cref{thm:gaussian-second-order-m}.
& Integrable hidden pair face with power $2/(1+\rho)<2$; the compact-face
limit is proved, while uniform weighted-half-space transfer remains open.
& No; the density is unbounded at equal-tail corners. \\

One-sided Gaussian, $-1<\rho<0$
& Axes.
& Undetermined.  The nominal lower--lower hidden scale is not directly
transferable because its face functional is nonintegrable.
& The lower--lower hidden face is nonintegrable for the half-space
increment, so an edge-renormalized term is expected.  The leading expansion
is open.
& No globally. \\

Two-sided Gaussian, $\rho\ne0$
& Axes when $|\rho|<1$.
& $t^{-2/(1+|\rho|)}$
\newline${}\times(\log t)^{-|\rho|/(1+|\rho|)}$ for both aggregates;
conditional on the face-transfer hypotheses of
\cref{thm:gaussian-second-order-m}.
& A dominant sign quadrant gives an integrable hidden pair face with power
$2/(1+|\rho|)<2$; the coefficient is identified, but the noncompact
boundary transfer remains open.
& No. \\

Standard multivariate-$t$, $\nu>2$
& Two-sided scores have full-interior support; one-sided scores have mixed
support with interior mass.
& $t^{-1-2/\nu}$ for both aggregates and both sidedness choices; proved for
fixed $m$.
& Ordinary common-radial perturbation of the existing first-order measure;
the $t^{-1-2/\nu}$ fixed-$m$ expansion is proved for both aggregates in
the standard model.
& No; positive tail dependence forces an unbounded lower-tail density. \\

Standard multivariate-$t$, $1<\nu\le2$
& Two-sided scores have full-interior support; one-sided scores have mixed
support with interior mass.
& Two-sided reciprocal and both Half-Cauchy: $t^{-1-2/\nu}$.
One-sided reciprocal: $t^{-2}$ for $1<\nu<2$, with a tie at $\nu=2$;
proved for fixed $m$.
& The same common-radial mechanism competes with an angular edge term.  The
two-sided reciprocal and both positive Half-Cauchy expansions retain the
$t^{-1-2/\nu}$ term; the one-sided reciprocal expansion has an order
$t^{-2}$ angular term below $\nu=2$, with a tie at $\nu=2$.
& No. \\

Standard multivariate-$t$, $0<\nu\le1$
& Two-sided scores have full-interior support; one-sided scores have mixed
with interior mass.
& Undetermined; no aggregate second-order rate is proved.
& Common-radial and angular subface terms meet or become nonintegrable; the
second-order half-space problem is open.
& No. \\

Clayton, $0<\theta<1$
& Full-interior mass.
& $t^{-1-\theta}$ for both aggregates; sharp and proved for fixed $m$.
& Ordinary common-frailty perturbation with a proved
$t^{-1-\theta}$ correction for both aggregates; it is larger than
$t^{-2}\log t$.
& No. \\

Clayton, $\theta\geq1$
& Full-interior mass.
& Reciprocal: $t^{-1-\theta}$.  Half-Cauchy: $t^{-1-\theta}$ for
$1\leq\theta<2$ and $t^{-3}$ for $\theta\geq2$; proved for fixed $m$.
& The reciprocal correction remains $t^{-1-\theta}$.  The positive
Half-Cauchy correction agrees below $\theta=2$, ties its score-map term at
$\theta=2$, and has order $t^{-3}$ above that boundary.
& No. \\

Bivariate Clayton, $-1\leq\theta<0$
& Axes; comparable lower extremes are eventually excluded.
& Open for $-1<\theta<0$.  At $\theta=-1$, reciprocal: $t^{-2}$;
Half-Cauchy: $O(t^{-3})$; proved.
& Singular exclusion/edge geometry.  The countermonotone endpoint is
proved, but the interior negative-parameter family remains open.
& No: unbounded for $-1<\theta<0$; the endpoint $\theta=-1$ is singular. \\

Gumbel, $\theta=1$
& Axes; independence.
& $t^{-2}\log t$ for both aggregates; sharp and proved.
& Critical hidden pair faces; the sharp independent logarithm is proved.
& Yes. \\

Gumbel, $\theta>1$
& Axes.
& Candidate $t^{-2^{1/\theta}}$ for both aggregates; the aggregate rate is
open.
& The joint-tail scaling suggests an integrable hidden pair-face correction
of power $2^{1/\theta}\in(1,2)$; only the rectangle obstruction to the
target rate is proved here.
& No; the density is unbounded at corners. \\

Joe, $\theta=1$
& Axes; independence.
& $t^{-2}\log t$ for both aggregates; sharp and proved.
& Critical hidden pair faces; the sharp independent logarithm is proved.
& Yes. \\

Joe, $\theta>1$ fixed
& Axes, with
$C(sx,sy)=O(s^2)$.
& $t^{-2}\log t$ bivariately for both aggregates; sharp and proved.
For $m>2$ the same rate is conditional on the higher-face hypothesis.
& In the bivariate case the sharp local-corner theorem applies with
$\lambda=\theta$; in higher dimension its explicit higher-face condition
still has to be verified.
& No globally; unbounded at the upper corner. \\

Cuadras--Aug\'e, $\theta=0$
& Axes; independence.
& $t^{-2}\log t$ for both aggregates; sharp and proved.
& Critical hidden pair faces; the sharp independent logarithm is proved.
& Yes. \\

Cuadras--Aug\'e, $0<\theta<1$
& Axes.
& Candidate $t^{-(2-\theta)}$ for both aggregates; the aggregate rate is
open.
& A singular hidden pair-face term of power $2-\theta$ is expected; only
the rectangle rate and singular-component diagnosis are proved.
& No; the copula has a singular component. \\

Cuadras--Aug\'e, $\theta=1$
& Full-interior diagonal ray.
& Reciprocal: no correction.  Half-Cauchy: $t^{-3}$; exact and proved.
& No dependence perturbation remains after normalization: the reciprocal
aggregate is exactly unit Pareto, while the positive Half-Cauchy aggregate
has its exact univariate $O(t^{-3})$ score-map correction.
& No; singular. \\

Marshall--Olkin with both nontrivial shock parameters
& Axes in the bivariate lower-tail orientation used here; reflected or
higher-dimensional shock models can charge proper faces or be mixed.
& Aggregate rate undetermined.  The equal-threshold rectangle scale is
$t^{-\{2-\min(\alpha,\beta)\}}$ only.
& A singular hidden pair-face power is expected axially; the weighted
half-space expansion remains open.
& No; it has a singular component. \\

Three-dimensional reflected Marshall--Olkin/max-linear pair-shock model
& Pure proper pair faces.
& Reciprocal: $\sigma_2(\bw)t^{-2}\log t$; Half-Cauchy:
$c^2\sigma_2(\bw)t^{-2}\log t$, both with $O(t^{-2})$ remainder; proved.
& Critical accumulation of two independent edge shocks near the charged
proper faces; \cref{thm:pair-shock-second-order} gives the sharp expansion.
& No; the model is max-linear and singular. \\

Comonotone copula
& Full-interior diagonal ray.
& Reciprocal: no correction.  Half-Cauchy: $t^{-3}$; exact and proved.
& Exact one-dimensional reduction: reciprocal aggregation is exactly unit
Pareto and the positive Half-Cauchy score-map correction is proved.
& No; singular. \\

Countermonotone bivariate copula
& Axes; lower comparable extremes are excluded.
& Reciprocal: $2w(1-w)t^{-2}$.  Half-Cauchy:
$\frac2\pi\{2w(1-w)-1/3\}t^{-3}+O(t^{-5})$; the $t^{-3}$
coefficient vanishes at $w=(1\pm1/\sqrt3)/2$, leaving $O(t^{-5})$.
& Singular exclusion geometry; the reciprocal correction is order
$t^{-2}$, whereas the positive Half-Cauchy correction is generically order
$t^{-3}$ and can cancel at the displayed weights.
& No; singular. \\

\end{longtable}
\endgroup

The ledger shows why support and mechanism must be recorded separately.
Positive lower-tail dependence rules out global bounded density but does
not rule out a quantitative half-space theorem: Clayton, standard
multivariate-$t$, and perfect dependence all illustrate this distinction.
Conversely, lower-tail independence identifies axis support but not the
next rate: independence is critical, positive Gaussian and Gumbel models
have larger hidden-face corrections, and singular Cuadras--Aug\'e or
Marshall--Olkin models require separate boundary analysis.

Rotations must be classified by which corner maps to small $p$-values.
They preserve global boundedness, but can move a tail-dependent or
corner-singular family's singularity into the relevant lower corner and
alter the Level-II conclusion.
}

\newcommand{\LiteratureAuditAppendix}{%
\section{Literature boundary and proof-status ledger}
\label{sec:story}

This appendix consolidates the extended literature review and the
claim-by-claim novelty audit.  It separates classical foundations and the
closest existing results from the statements proved here and from the
questions that remain open.  Here \emph{second-order calibration} means the
leading nonzero difference between the dependent tail and the exact
independence tail after their common $t^{-1}$ term.  This usage is broader
than classical second-order MRV: the correction may instead come from a
hidden or critical face, a common radial or angular expansion, or the
Half-Cauchy score map.

\subsection{What was already known}

The following inputs belong to the literature and should not be claimed as
new.
\begin{enumerate}[leftmargin=2.2em]
 \item Reciprocal aggregation underlies the harmonic-mean $p$-value of
 Wilson~\cite{Wilson2019}.  Positive Half-Cauchy aggregation belongs to the
 heavily right framework of Liu, Meng and Pillai
 \cite{LiuMengPillai2025}, building on the signed Cauchy combination test of
 Liu and Xie~\cite{LiuXie2020} and its subsequent dependence analysis
 \cite{LongEtAl2023}.  The signed and positive transformations are not the
 same problem: the nonnegative statistics considered here are governed by
 the lower tail of the $p$-value copula, whereas signed scores can involve
 cancellation and opposite-corner conditions.

 \item Homogeneous mapping of MRV vectors, spectral representations, and
 weighted-sum first-order tails are classical; see, for example,
 Hult and Lindskog~\cite{HultLindskog2002} and
 Basrak, Davis and Mikosch~\cite{BasrakDavisMikosch2002}.  Therefore the
 fixed-dimensional statement that index-one MRV implies a universal
 normalized projection constant is a foundation, not the main novelty.
 For heavy-tailed combination tests specifically, first-order fixed-$m$
 validity under quasi-asymptotic independence is given by Gui, Jiang and
 Wang~\cite{GuiJiangWang2025}; the MRV-copula benchmark under asymptotic
 dependence is given by Gui, Mao, Wang and Wang
 \cite{GuiMaoWangWang2026}; and the angular-measure formulation and
 universal Pareto/harmonic-mean calibration result are given by Chakraborty, Guo,
 Shedden and Stoev~\cite{ChakrabortyGuoSheddenStoev2026}.
 The one-big-jump argument of Liu, Meng and Pillai
 \cite{LiuMengPillai2025} is a sufficient route through pairwise tail
 control, but it is not necessary for the index-one projection identity:
 interior-supported models can retain the same first-order constant.

 \item Transfer from multivariate second-order regular variation to risk
 aggregation is studied by Das and Kratz~\cite{DasKratz2020}.  A generic
 assertion that second-order MRV gives a second-order sum expansion is not
 new.  Second-order convolution theory and risk-concentration expansions
 also predate this paper
 \cite{GelukEtAl1997,BarbeMcCormick2005,DegenLambriggerSegers2010}.
 Yang and Zhang~\cite{YangZhang2023} explicitly obtain the bivariate
 unit-Pareto $(\log t)/t^2$ term and the corresponding edge boundary.
 Thus the logarithm itself is not new.  The issue specific to the present
 problem is its arbitrary fixed-dimensional, unequal-weight, exactly
 standardized harmonic-mean/positive-Half-Cauchy calibration coefficient
 with explicit higher-face control.

 \item Hidden regular variation, regular variation on nested cones,
 residual tail orders, and higher-order tail densities are established
 frameworks
 \cite{LedfordTawn1996,LedfordTawn1997,Resnick2002,Resnick2008,
 HuaJoeLi2014,LiHua2015}.  M\"obius inversion of face and excess measures
 is likewise classical \cite{BalkemaEmbrechts2007}.  The possible
 contribution here is the particular weighted-half-space increment and
 its second-order calibration use, not the underlying cone or
 inclusion--exclusion machinery.

 \item Gaussian-copula risks with regularly varying margins and their
 hidden tail geometry have substantial prior theory, including
 Kortschak~\cite{Kortschak2012} and Das and
 Fasen-Hartmann~\cite{DasFasen2024}.  The overlap is substantial at the
 level of bivariate joint-tail powers and Gaussian cone optimization.  The
 present Gaussian calculation should be described only as a conditional
 reciprocal weighted-half-space face integral, its transfer to positive
 Half-Cauchy aggregation under the same boundary conditions, its formal
 arbitrary-$m$ sum over dominant pairs, the two-sided
 quadrant constant, and the critical transition to the independent
 logarithm.  A theorem-by-theorem comparison with
 Kortschak remains necessary before a novelty claim is finalized.

 \item Frailty representations and lower-tail properties of multivariate
 Archimedean copulas are standard; see McNeil and
 Ne\v{s}lehov\'a~\cite{McNeilNeslehova2009} and Charpentier and
 Segers~\cite{CharpentierSegers2009}.  Tail aggregation under Archimedean
 dependence also predates this paper; see Barbe, Foug\`eres and
 Genest~\cite{BarbeFougeresGenest2006} and Embrechts, Ne\v{s}lehov\'a and
 W\"uthrich~\cite{EmbrechtsNeslehovaWuthrich2009}, as well as the explicit
 second-order work of Tong, Wu and Xu~\cite{TongWuXu2012} and Coqueret
 \cite{Coqueret2014}.  Finite-level harmonic-mean inequalities for some
 Clayton copulas are given by Chen, Wang, Wang and Zhu
 \cite{ChenWangWangZhu2026}.  The Clayton claim here is therefore
 the exact $R_{\bw}$ and $T_{\bw}$ coefficients and the score-map phase
 diagram, not
 the discovery of its tail dependence, frailty representation, or
 finite-level validity bounds.

 \item MRV and second-order radial behavior of elliptical and random-scale
 models are known \cite{HultLindskog2002,DemartaMcNeil2005,
 NikoloulopoulosJoeLi2009,Kim2024}.  The standard-$t$
 contribution here is the exact uniform-$p$ score expansion, explicit
 weighted angular coefficient, and distinct one- versus two-sided reciprocal
 and Half-Cauchy corrections.
\end{enumerate}

\small
\begin{longtable}{@{}p{0.20\textwidth}p{0.34\textwidth}p{0.36\textwidth}@{}}
\caption{Closest-result comparison delimiting the paper's contribution.
The last column states a scope boundary, not an unconditional
priority claim.}\label{tab:novelty-comparison}\\
\toprule
Reference & Established overlap & Scope of the present contribution \\
\midrule
\endfirsthead
\toprule
Reference & Established overlap & Scope of the present contribution \\
\midrule
\endhead
\bottomrule
\endfoot

Gui, Jiang and Wang~\cite{GuiJiangWang2025}
& Fixed-$m$ first-order validity and Bonferroni equivalence under
quasi-asymptotic independence.
& Dependence-specific second-order size distortion under axis, interior,
proper-face, and mixed MRV geometry. \\

Gui, Mao, Wang and Wang~\cite{GuiMaoWangWang2026}
& MRV-copula formulation; first-order validity and power under asymptotic
dependence.
& Rates and coefficients after the first-order MRV benchmark. \\

Chakraborty, Guo, Shedden and Stoev
\cite{ChakrabortyGuoSheddenStoev2026}
& Angular-measure formula and universal first-order Pareto/harmonic-mean calibration
within a broad homogeneous class.
& Exact-independence-relative positive Half-Cauchy and harmonic-mean error, including its sign and
critical-value effect. \\

Das and Kratz~\cite{DasKratz2020}
& General second-order MRV transfer and risk-concentration theory.
& Exactly standardized $p$-scores and noncompact weighted-half-space
functionals whose face limits can fail to be integrable. \\

Kortschak~\cite{Kortschak2012}
& Second-order sums for tail-independent risks, including Gaussian and
some Archimedean copulas.
& Only the exact positive Half-Cauchy and harmonic-mean normalization, arbitrary-$m$ weighted coefficient,
and model-specific calibration phases are potentially new; direct theorem
comparison remains required. \\

Yang and Zhang~\cite{YangZhang2023}
& Bivariate copula-sum expansions, the unit-index independence logarithm,
and an integrable-versus-truncated-moment edge boundary.
& The local-corner theorem extends the calculation to arbitrary fixed $m$,
unequal weights, pair-specific kernels, and explicit higher-face control. \\

Tong, Wu and Xu~\cite{TongWuXu2012}; Coqueret~\cite{Coqueret2014}
& Second-order Archimedean risk concentration and Bernstein-copula
aggregation.
& The reciprocal-$p$/HC coefficients and Clayton score-map transition are
specific contributions, conditional on formula-level
comparison. \\

Balkema and Embrechts~\cite{BalkemaEmbrechts2007}
& Geometric face/excess measures and M\"obius inversion.
& The induced weighted-half-space increment and its calibration use, not
generic face-lattice calculus. \\

Marshall and Olkin~\cite{MarshallOlkin1967}; Wang and Stoev
\cite{WangStoev2011}
& Shock models and discrete max-linear spectral support on selected faces.
& The explicit three-edge weighted-sum coefficient and the comparison of
positive Half-Cauchy with harmonic-mean calibration, pending a specialized
aggregation search. \\

Chen, Wang, Wang and Zhu~\cite{ChenWangWangZhu2026}
& Finite-level subuniformity and threshold bounds, including some Clayton
copulas.
& Small-level asymptotic coefficients and phase transitions rather than
finite-level validity inequalities. \\
\end{longtable}
\normalsize

This audit reflects sources located through August 23, 2026.  It supports
the scope statements above, but it does not replace a final formula-level
comparison with the specialized Gaussian, Archimedean, and
Marshall--Olkin aggregation literature before submission.

\subsection{What is proved in this paper}

The formal results now establish the following package.
\begin{enumerate}[leftmargin=2.2em]
 \item Fixed-dimensional first-order validity of both
 reciprocal/harmonic-mean and positive Half-Cauchy aggregation under
 general index-one MRV/lower MDA, including exponent measures with interior
 mass.  This is retained as a first-order benchmark, not a novelty claim.
 \item An independence-relative tail-distortion definition and a rigorous
 conversion of a regularly varying tail difference into size and
 critical-value expansions.
 \item A quantitative MRV projection result, product-shell and density
 criteria, an exact face-lattice decomposition, and an integrable
 hidden-face transfer theorem.
 \item A support-first exponent-measure taxonomy, a proof that
 exact Pareto normalization rules out nonzero pure-axis second-order
 measures, and an explicit product-power trichotomy for integrable,
 critical, and nonintegrable pair-face functionals.  The
 interior category also formally shows that every subset joint extreme is
 already first-order and that a score map homogeneous to the raw radial
 index converts the common radial variable to index one.
 \item A rigorous common-heavy-factor aggregation theorem with explicit
 $L^1$ inversion assumptions, together with reciprocal-edge and
 Half-Cauchy phase diagrams that separate radial, angular, and score-map
 corrections.
 \item The sharp independent coefficients multiplying $(\log t)/t^2$:
 $2\sum_{i<j}w_iw_j$ for $R_{\bw}$ and
 $2c^2\sum_{i<j}w_iw_j$ for $T_{\bw}$.
 \item The local-corner theorem in
 \cref{thm:sharp-local-corner}: arbitrary fixed dimension and unequal
 weights, pair-specific finite corner limits, explicit control of every
 higher face, and the resulting $R_{\bw}$ and $T_{\bw}$ calibration
 coefficients and sign.
 \item For Gaussian copulas, the compact hidden pair limits and their
 explicit M\"obius coefficients, together with a conditional fixed-$m$
 half-space theorem that isolates the still-unproved stratified boundary
 estimates.  No unconditional arbitrary-dimensional Gaussian expansion is
 claimed.
 \item For the standard multivariate-$t$ copula with $\nu>1$, proved
 reciprocal/harmonic-mean and positive Half-Cauchy second-order expansions
 with computable Gaussian angular expectations, including their necessary
 one-sided distinction when $1<\nu\le2$.
 \item For every positive Clayton parameter and arbitrary fixed $m$, an
 exact reciprocal-sum correction with a nonnegative coefficient, strictly
 positive when at least two weights are active and zero for one active
 coordinate, together with the positive Half-Cauchy phase transition at
 $\theta=2$.
 \item The proper-face pair-shock theorem in
 \cref{thm:pair-shock-second-order}: a pure edge-supported exponent
 measure, an explicit $(\log t)/t^2$ coefficient, and its
 independence-relative calibration consequence.  This is a canonical
 model theorem, not a general proper-face theorem.
 \item Uniform projection and radial-kernel theorems, plus unrestricted
 growing-dimensional first-order validity for fixed-parameter positive
 Clayton and standard multivariate-$t$ common-factor arrays.
\end{enumerate}

\subsection{What remains open}

The following items are not proved and are deliberately not hidden inside
the theorem statements.
\begin{enumerate}[leftmargin=2.2em]
 \item A full nested critical-face theorem covering simultaneous critical
 and nonintegrable subfaces.  The local-corner theorem controls the finite
 pair-kernel boundary under explicit higher-face bounds, but it does not
 supply the desired recursive theory for singular or nonintegrable nested
 faces.
 \item A general proper-face or mixed-support second-order theorem beyond
 the explicit three-dimensional pair-shock calculation.
 \item One-sided Gaussian second terms when all active correlations are
 nonpositive with at least one strictly negative, the stratified boundary
 estimate needed for the conditional positive-correlation theorem, and a
 uniform Gaussian theorem across $\rho\log t=O(1)$.  The all-zero case is
 independence and is already proved.
 \item The critical and subcritical standard-$t$ cases $\nu\le1$.
 \item Uniform second-order remainders for growing Clayton and $t$ arrays.
 The coefficient bounds are dimension-free, but coefficient control is not
 remainder control.
 \item A growing-dimensional Gaussian theorem with explicit accumulation
 of every active face.  The polynomial rate following
 \cref{open:gaussian-growing} is only a diagnostic.
 \item Formula-level, theorem-by-theorem comparisons with the closest
 aggregation results for Gaussian, Archimedean, and Marshall--Olkin models
 before any unconditional priority claim is made.
\end{enumerate}

}

\newcommand{\CopulaClassificationAppendix}{%
\begin{landscape}
\section{Systematic copula classification and supporting proofs}
\label{app:copula-classification-proofs}

This appendix contains the full copula ledger and verifies its support,
rectangle-rate, and bounded-density assertions.  Proved aggregate rates are
either established in the main model theorems or derived explicitly below;
the table cross-references those proofs.  Mechanism statements explicitly
marked ``candidate,'' ``diagnostic,'' or ``open'' are classifications
suggested by the available tail geometry, not weighted-half-space theorems.
Except where a
multivariate model is explicitly stated, the calculations are bivariate.
Specifying every bivariate margin of an $m$-dimensional copula does not in
general determine its full lower stable-tail-dependence function when
$m>2$.

The proof-status language in the rate and mechanism columns is used
literally.  When a row says that an aggregate rate is proved, a full
weighted-half-space conclusion has been established at the scale displayed
in that row.  In rows displaying the Level-II target, this means, for
$X\in\{Y,H\}$ with $c_Y=1$ and $c_H=2/\pi$,
\[
  \Pp\!\left(\sum_iw_iX_i>t\right)=\frac{c_X}{t}
  +O\!\left(\frac{\log t}{t^2}\right).
\]
Any different displayed power or slowly varying scale is to be read
literally.  A rectangle obstruction or a statement labelled ``expected''
is only diagnostic and is not promoted to a half-space theorem.

\CopulaSupportMechanismLedger
\end{landscape}

\subsection{Three reusable criteria}

Let $(U,V)$ have exact uniform margins and bivariate copula $C$.  For fixed
$x,y>0$, inclusion--exclusion gives
\begin{equation}\label{eq:appendix-rectangle-identity}
  \frac1s\Pp(U\leq sx\ \text{or}\ V\leq sy)
  =x+y-\frac{C(sx,sy)}s.
\end{equation}
Consequently:
\begin{enumerate}[leftmargin=2.2em]
  \item if $C(sx,sy)/s\to\Lambda_L(x,y)$, the lower-MDA limit exists and
  equals $x+y-\Lambda_L(x,y)$;

  \item the rectangle form of the Level-II target is
  \begin{equation}\label{eq:appendix-rectangle-rate}
    \frac{C(sx,sy)}s-\Lambda_L(x,y)
    =O\{s\log(1/s)\};
  \end{equation}

  \item in the axis-supported case $\Lambda_L=0$, a single diagonal
  calculation
  \[
    C(s,s)\not=O\{s^2\log(1/s)\}
  \]
  disproves the Level-II rectangle rate.
\end{enumerate}

If $C$ has a globally bounded density $c_C\leq M$, then
\begin{equation}\label{eq:appendix-density-rectangle}
  C(u,v)\leq Muv.
\end{equation}
\cref{thm:pair-density-second-order,cor:bounded-copula-density} then give
the full Level-II bound for both reciprocal/harmonic-mean and positive
Half-Cauchy aggregation.  Conversely,
\cref{eq:appendix-density-rectangle} shows that a copula satisfying
$C(s,s)\gg s^2$ cannot have a globally bounded density.
The same argument applies to any other corner rectangle: if
\[
  \Pp(U\leq s,V\geq1-s)\gg s^2
  \quad\text{or}\quad
  \Pp(U\geq1-s,V\geq1-s)\gg s^2,
\]
then global boundedness is impossible.  A copula with a nonzero singular
component is, of course, not absolutely continuous and cannot satisfy the
density condition.

\subsection{Product, FGM, and the smooth bounded-density families}

\paragraph*{Product copula}
For $C(u,v)=uv$,
\[
  C(sx,sy)=s^2xy,
  \qquad
  c_C(u,v)=1.
\]
Thus $C(sx,sy)/s\to0$, the lower-MDA measure is axis-supported, and the
global density bound proves Level II.
\cref{eq:independent-pareto-m,eq:independent-hc-m} show that the
logarithmic second term is generally genuinely present.

\paragraph*{FGM copula}
For $-1\leq\theta\leq1$,
\begin{align*}
  C_\theta(u,v)
  &=uv\{1+\theta(1-u)(1-v)\},\\
  c_\theta(u,v)
  &=1+\theta(1-2u)(1-2v).
\end{align*}
Hence
\[
  C_\theta(sx,sy)=s^2xy\{1+\theta+O(s)\},
  \qquad
  0\leq c_\theta(u,v)\leq2.
\]
The lower-MDA limit is axial and the bounded-density theorem proves the
full Level-II rate.

\paragraph*{AMH with $\theta<1$}
For
\[
  C_\theta(u,v)
  =\frac{uv}{1-\theta(1-u)(1-v)},
  \qquad -1\leq\theta<1,
\]
the denominator is bounded away from zero on $[0,1]^2$ for every fixed
$\theta<1$.  The copula and its two mixed derivatives are rational
functions with that same nonvanishing denominator.  Therefore the density
is continuous and bounded on the compact square.  Also,
\begin{equation}\label{eq:appendix-amh-interior}
  C_\theta(sx,sy)
  =\frac{s^2xy}{1-\theta+O(s)}=O(s^2).
\end{equation}
This proves the axial Level-I limit and the full Level-II result.  The
lower bound on the denominator tends to zero as $\theta\uparrow1$, so this
argument is not uniform at the endpoint.

\paragraph*{Frank copula}
For fixed finite $\theta\ne0$,
\begin{equation}\label{eq:appendix-frank-copula}
  C_\theta(u,v)
  =-\frac1\theta\log\!\left[
    1+\frac{(e^{-\theta u}-1)(e^{-\theta v}-1)}{e^{-\theta}-1}
  \right].
\end{equation}
Taylor expansion at the lower corner yields
\begin{equation}\label{eq:appendix-frank-lower}
  C_\theta(sx,sy)
  =\frac{\theta}{1-e^{-\theta}}s^2xy+O(s^3).
\end{equation}
Its density is
\begin{equation}\label{eq:appendix-frank-density}
  c_\theta(u,v)
  =\frac{\theta(1-e^{-\theta})e^{-\theta(u+v)}}
  {\{1-e^{-\theta}-(1-e^{-\theta u})(1-e^{-\theta v})\}^2}.
\end{equation}
For fixed finite $\theta$, the denominator in
\cref{eq:appendix-frank-density} is bounded away from zero in absolute
value on the compact square.  Thus the density is bounded.  At $\theta=0$
the continuous limit is the product copula.  This proves all three Frank
entries.

\paragraph*{Plackett copula}
For $0<\theta<\infty$, $\theta\ne1$, write
\begin{align*}
  D_\theta(u,v)
  &=\{1+(\theta-1)(u+v)\}^2
    -4\theta(\theta-1)uv,\\
  C_\theta(u,v)
  &=\frac{1+(\theta-1)(u+v)-\sqrt{D_\theta(u,v)}}
          {2(\theta-1)}.
\end{align*}
Direct differentiation gives
\begin{equation}\label{eq:appendix-plackett-density}
  c_\theta(u,v)
  =\frac{\theta\{1+(\theta-1)(u+v-2uv)\}}
        {D_\theta(u,v)^{3/2}}.
\end{equation}
For every fixed $\theta>0$, $D_\theta$ is strictly positive on
$[0,1]^2$; continuity and compactness therefore give a finite density
bound.  Expanding the copula at $(0,0)$ gives
\begin{equation}\label{eq:appendix-plackett-lower}
  C_\theta(sx,sy)=\theta s^2xy+O(s^3).
\end{equation}
The case $\theta=1$ is the product copula.  Hence the lower limit is axial
and Level II follows from bounded density.  The bound is parameterwise and
need not remain uniform as $\theta\downarrow0$ or $\theta\uparrow\infty$.

\paragraph*{Finite Bernstein, checkerboard, and mixture constructions}
A fixed-degree Bernstein copula density is a finite sum of bounded beta
polynomial basis functions.  A finite checkerboard density is piecewise
constant on finitely many cells.  A finite mixture
$c=\sum_{k=1}^K\pi_kc_k$ of bounded component densities satisfies
\[
  \|c\|_\infty
  \leq\sum_{k=1}^K\pi_k\|c_k\|_\infty<\infty.
\]
Each construction therefore satisfies Level III, and hence Levels II and
I, provided the degree, grid, number of components, and their bounds are
fixed.

For the bivariate FGM, AMH with $\theta<1$, Frank, and Plackett families,
the displayed analytic formulas also give
\[
 c_\theta(u,v)=\lambda_\theta+O(u+v)
 \qquad\text{as }(u,v)\to(0,0),
\]
with $\lambda_\theta=1+\theta$, $(1-\theta)^{-1}$,
$\theta/(1-e^{-\theta})$, and $\theta$, respectively.  The Dini condition
in \cref{thm:sharp-local-corner} therefore holds, and its higher-face
condition is vacuous in two dimensions.  This proves the sharp
$t^{-2}\log t$ coefficient recorded in \cref{tab:copula-hierarchy} for
both aggregates.  In dimensions above two the same conclusion is
conditional on the theorem's explicit higher-face bound; bounded pairwise
margins alone do not verify it.

\subsection{The AMH endpoint}

At $\theta=1$,
\[
  C_1(u,v)=\frac{uv}{u+v-uv}.
\]
For $x+y>0$,
\begin{align}
  \frac{C_1(sx,sy)}s
  &=\frac{xy}{x+y-sxy}\notag\\
  &=\frac{xy}{x+y}
    +s\frac{x^2y^2}{(x+y)^2}+O(s^2).
  \label{eq:appendix-amh-endpoint-expansion}
\end{align}
Thus
\[
  \Lambda_L(x,y)=\frac{xy}{x+y},
  \qquad
  \ell_L(x,y)=x+y-\frac{xy}{x+y},
\]
and the rectangle error is $O(s)$, which is compatible with
\cref{eq:appendix-rectangle-rate}.  This rectangle calculation alone proves
only compatibility; the bounded-density theorem is unavailable.  Indeed,
direct differentiation gives
\begin{equation}\label{eq:appendix-amh-endpoint-density}
  c_1(u,v)=\frac{2uv}{(u+v-uv)^3},
\end{equation}
so $c_1(s,s)\sim(4s)^{-1}$ and Level III fails.

The endpoint is not a separate unresolved model.  Algebraically,
\begin{equation}\label{eq:amh-clayton-identity}
 C_{\mathrm{AMH},1}(u,v)
 =\frac{uv}{u+v-uv}
 =(u^{-1}+v^{-1}-1)^{-1}
 =C_{\mathrm{Clayton},1}(u,v).
\end{equation}
Consequently, with $v=1-w$ and $0<w<1$,
\cref{thm:clayton-second-order,thm:clayton-hc-phase} give
\begin{align}
 \Pp(wY_1+vY_2>t)
 &=\frac1t+wv\,t^{-2}+o(t^{-2}),
 \label{eq:amh-endpoint-reciprocal-tail}\\
 \Pp(wH_1+vH_2>t)
 &=\frac ct+c^2wv\,t^{-2}+o(t^{-2}).
 \label{eq:amh-endpoint-hc-tail}
\end{align}
Thus both sharp coefficients are already part of the proved positive
Clayton common-frailty theorem.

\subsection{Clayton copulas}

\paragraph*{Positive parameter}
For $\theta>0$,
\[
  C_\theta(u,v)
  =(u^{-\theta}+v^{-\theta}-1)^{-1/\theta}.
\]
Let $A=x^{-\theta}+y^{-\theta}$.  Then
\begin{align}
  \frac{C_\theta(sx,sy)}s
  &=(A-s^\theta)^{-1/\theta}\notag\\
  &=A^{-1/\theta}
    +\frac1\theta A^{-1/\theta-1}s^\theta
    +O(s^{2\theta}).
  \label{eq:appendix-clayton-positive-expansion}
\end{align}
Hence the lower-MDA limit exists with
\[
  \Lambda_L(x,y)=(x^{-\theta}+y^{-\theta})^{-1/\theta}.
\]
When $0<\theta<1$, the error $s^\theta$ is larger than
$s\log(1/s)$, so the Level-II rectangle diagnostic fails.  When
$\theta\geq1$, the error is $O\{s\log(1/s)\}$, so the rectangle diagnostic
is compatible.  The separate half-space argument is supplied by
\cref{thm:clayton-second-order,thm:clayton-hc-phase}.  For
$R_{\bw}$ the correction has order $t^{-1-\theta}$ for every
$\theta>0$.  For $T_{\bw}$ it has the same order below $\theta=2$,
ties the intrinsic score-map term at $\theta=2$, and has order $t^{-3}$
above that boundary.  Thus Level II holds for both aggregates whenever
$\theta\ge1$.

For every $\theta>0$, $C_\theta(s,s)\sim2^{-1/\theta}s$.  A globally
bounded density would instead force $C_\theta(s,s)=O(s^2)$ by
\cref{eq:appendix-density-rectangle}.  Thus Level III fails for every
positive Clayton parameter.

\paragraph*{Negative parameter}
Write $\theta=-\alpha$ with $0<\alpha\leq1$.  The bivariate copula is
\begin{equation}\label{eq:appendix-negative-clayton}
  C_{-\alpha}(u,v)
  =(u^\alpha+v^\alpha-1)_+^{1/\alpha}.
\end{equation}
For fixed $x,y>0$,
$s^\alpha(x^\alpha+y^\alpha)<1$ eventually, so
\[
  C_{-\alpha}(sx,sy)=0
\]
for all sufficiently small $s$.  The lower-MDA measure is therefore axial,
and the rectangle diagnostic is exact near the lower corner.  This does
not by itself prove the weighted half-space rate because the support
boundary can be encountered at unequal scales.

For $0<\alpha<1$, the density on
$u^\alpha+v^\alpha>1$ is
\begin{equation}\label{eq:appendix-negative-clayton-density}
  c_{-\alpha}(u,v)
  =(1-\alpha)(uv)^{\alpha-1}
   (u^\alpha+v^\alpha-1)^{1/\alpha-2}.
\end{equation}
Take
\[
  v=(1-u^\alpha/2)^{1/\alpha}.
\]
Then $v\to1$, the last parenthesis in
\cref{eq:appendix-negative-clayton-density} equals $u^\alpha/2$, and
\[
  c_{-\alpha}(u,v)\asymp u^{-\alpha}\longrightarrow\infty.
\]
Thus Level III fails throughout $-1<\theta<0$.  At $\theta=-1$ the copula
is the singular countermonotone copula.

\subsection{Gumbel and Joe copulas}

\paragraph*{Gumbel}
For $\theta\geq1$,
\[
  C_\theta(u,v)
  =\exp\!\left[-\{(-\log u)^\theta+(-\log v)^\theta\}^{1/\theta}\right].
\]
At $\theta=1$ this is the product copula.  Suppose $\theta>1$ and let
$L=\log(1/s)$ and $q=2^{1/\theta}$.  For fixed $x,y>0$,
\[
  \{(L-\log x)^\theta+(L-\log y)^\theta\}^{1/\theta}
  =qL-\frac q2(\log x+\log y)+O(L^{-1}).
\]
Consequently,
\begin{equation}\label{eq:appendix-gumbel-axes}
  C_\theta(sx,sy)
  =s^q(xy)^{q/2}\{1+O(L^{-1})\}=o(s),
\end{equation}
so the lower-MDA limit is axial.  On the diagonal the formula is exact:
\begin{equation}\label{eq:appendix-gumbel-diagonal}
  C_\theta(s,s)=s^{2^{1/\theta}},
  \qquad 1<2^{1/\theta}<2.
\end{equation}
Hence $C_\theta(s,s)/s$ converges too slowly for
\cref{eq:appendix-rectangle-rate}, proving Level-II failure.
\cref{eq:appendix-gumbel-diagonal} also contradicts the $O(s^2)$ bound
implied by a globally bounded density, so Level III fails.

The sharper expansion in \cref{eq:appendix-gumbel-axes} has the
product-power hidden survival $(xy)^{-a}$ after reciprocal scoring, with
$a=q/2\in(1/2,1)$.  The corresponding M\"obius half-space increment is
therefore integrable by \cref{prop:product-power-face}.  This supports the
mechanism prediction in \cref{tab:copula-hierarchy}; it does not by itself
verify the uniform truncation needed for a full weighted-sum expansion.
Because $q<2$, any validated reciprocal correction at this scale would
transfer through \cref{lem:bounded} to $T_{\bw}$ with coefficient scaled
by $c^q$.  The missing truncation argument is therefore open for both
aggregates, not only for the reciprocal scores.

\paragraph*{Joe}
For $\theta\geq1$,
\begin{equation}\label{eq:appendix-joe-copula}
  C_\theta(u,v)
  =1-\{(1-u)^\theta+(1-v)^\theta
       -(1-u)^\theta(1-v)^\theta\}^{1/\theta}.
\end{equation}
At $\theta=1$ this is the product copula.  For fixed $x,y>0$ and
$\theta>1$, Taylor expansion at the lower corner gives
\begin{equation}\label{eq:appendix-joe-lower}
  C_\theta(sx,sy)=\theta s^2xy+O(s^3).
\end{equation}
Thus the lower-MDA limit is axial.  To verify the lower-strip condition,
put
\[
  a=(1-u)^\theta,
  \qquad b=(1-v)^\theta,
  \qquad D=a+b-ab.
\]
Differentiation of \cref{eq:appendix-joe-copula} gives
\begin{equation}\label{eq:appendix-joe-density}
\begin{split}
  c_\theta(u,v)
  ={}&(1-u)^{\theta-1}(1-v)^{\theta-1}D^{1/\theta-2}\\
     &\times\{\theta D+(\theta-1)(1-a)(1-b)\}.
\end{split}
\end{equation}
If $u\leq u_0<1$, then $a\geq(1-u_0)^\theta>0$ and $D\geq a$.
Every factor in \cref{eq:appendix-joe-density} is therefore uniformly
bounded for $u\leq u_0$ and $0<v<1$; symmetry handles the other lower
strip.  \cref{thm:pair-density-second-order} with the strip
observation \cref{eq:strip-copula-bound} proves the full Level-II rate.
Moreover, expansion of \cref{eq:appendix-joe-density} at the lower corner
gives $c_\theta(u,v)=\theta+O(u+v)$.  Hence the bivariate Dini condition in
\cref{thm:sharp-local-corner} holds with $\lambda=\theta$, proving the
sharp logarithmic coefficient for both aggregates.  For $m>2$, the same
coefficient remains conditional on the theorem's higher-face hypothesis.

Global boundedness nevertheless fails.  From
\cref{eq:appendix-joe-copula},
\begin{align}
  \Pp(U>1-s,V>1-s)
  &=1-2(1-s)+C_\theta(1-s,1-s)\notag\\
  &=s\{2-(2-s^\theta)^{1/\theta}\}\notag\\
  &\sim \{2-2^{1/\theta}\}s \asymp s.
  \label{eq:appendix-joe-upper-tail}
\end{align}
This is incompatible with the $O(s^2)$ upper-corner probability implied by
a globally bounded density.

\subsection{Cuadras--Aug\'e, Marshall--Olkin, and the Fr\'echet endpoints}

\paragraph*{Cuadras--Aug\'e}
For
\[
  C_\theta(u,v)=\{\min(u,v)\}^\theta(uv)^{1-\theta},
  \qquad0\leq\theta\leq1,
\]
assume without loss that $x\leq y$.  Then
\begin{equation}\label{eq:appendix-ca-scaling}
  C_\theta(sx,sy)
  =s^{2-\theta}x y^{1-\theta}.
\end{equation}
At $\theta=0$ this is the product copula.  For $0<\theta<1$,
$C_\theta(sx,sy)=o(s)$, so the lower-MDA measure is axial, but
\[
  \frac{C_\theta(sx,sy)}s
  =s^{1-\theta}xy^{1-\theta}
\]
is too large for the Level-II rectangle rate.  The copula has a nonzero
singular component for every $\theta>0$, so Level III fails.  At
$\theta=1$, $C_1(u,v)=\min(u,v)$ is the comonotone copula treated below.

\paragraph*{Marshall--Olkin}
For $0\leq\alpha,\beta\leq1$,
\begin{equation}\label{eq:appendix-mo-copula}
  C_{\alpha,\beta}(u,v)
  =\min\!\left\{u^{1-\alpha}v,\,uv^{1-\beta}\right\}.
\end{equation}
For fixed $x,y>0$,
\begin{equation}\label{eq:appendix-mo-scaling}
  C_{\alpha,\beta}(sx,sy)
  =\min\{s^{2-\alpha}x^{1-\alpha}y,
          s^{2-\beta}xy^{1-\beta}\}.
\end{equation}
If $\alpha,\beta>0$ and $(\alpha,\beta)\ne(1,1)$, this has diagonal order
\begin{equation}\label{eq:appendix-mo-diagonal}
  C_{\alpha,\beta}(s,s)
  =s^{2-\min(\alpha,\beta)},
  \qquad 1<2-\min(\alpha,\beta)<2.
\end{equation}
Therefore the lower-MDA limit is axial but the Level-II rectangle rate
fails.  The copula has a singular component on
$u^\alpha=v^\beta$, so Level III fails.  The endpoint
$\alpha=\beta=1$ is comonotone.  If either shock parameter is zero, the
diagonal exponent is two and the failure argument in
\cref{eq:appendix-mo-diagonal} no longer applies; this is why the table
states both parameters to be nontrivial.

\paragraph*{Comonotone copula}
If $U_1=\cdots=U_m=U$, then the copula is comonotone, the lower tail measure
is diagonal, and normalized weights give
\[
 R_{\bw}=U^{-1},\qquad \Pp(R_{\bw}>t)=t^{-1},\quad t\ge1,
\]
whereas
\[
  T_{\bw}=\cot(\pi U/2).
\]
Hence
\begin{align}
  \Pp(T_{\bw}>t)
  &=1-\frac2\pi\arctan t\notag\\
  &=\frac{2}{\pi t}-\frac{2}{3\pi t^3}+O(t^{-5}).
  \label{eq:appendix-comonotone-tail}
\end{align}
Thus reciprocal aggregation is exactly Pareto, while positive Half-Cauchy
aggregation has the sharper $O(t^{-3})$ remainder.
The copula is singular, so Level III fails.  The same calculation proves
the Cuadras--Aug\'e endpoint entry.

\paragraph*{Countermonotone copula}
Let $(U_1,U_2)=(U,1-U)$ and let $0\leq w\leq1$.  Put
\[
  H=\cot(\pi U/2).
\]
Then $H$ is standard Half-Cauchy and
\[
  \cot\{\pi(1-U)/2\}=H^{-1},
  \qquad
  T_{\bw}=wH+(1-w)H^{-1}.
\]
The reciprocal calculation is already given in
\cref{eq:axis-counter-reciprocal}; it has an order-$t^{-2}$ correction.
The positive Half-Cauchy correction is different:
For $0<w<1$ and $t>2\sqrt{w(1-w)}$, define
\begin{equation}\label{eq:appendix-counter-roots}
  h_\pm(t)
  =\frac{t\pm\sqrt{t^2-4w(1-w)}}{2w}.
\end{equation}
The event $\{T_{\bw}>t\}$ is the disjoint union
$\{H<h_-(t)\}\cup\{H>h_+(t)\}$.  Since
$F_H(h)=2\arctan(h)/\pi$,
\begin{equation}\label{eq:appendix-counter-exact-tail}
  \Pp(T_{\bw}>t)
  =\frac2\pi\left\{
      \arctan h_-(t)+\arctan\frac1{h_+(t)}
    \right\}.
\end{equation}
Expansion of \cref{eq:appendix-counter-roots} gives
\[
  h_-(t)=\frac{1-w}{t}
          +\frac{w(1-w)^2}{t^3}+O(t^{-5}),
  \qquad
  h_+(t)^{-1}=\frac{w}{t}
          +\frac{w^2(1-w)}{t^3}+O(t^{-5}).
\]
Using $\arctan x=x-x^3/3+O(x^5)$ therefore gives
\begin{equation}\label{eq:appendix-counter-tail}
  \Pp(T_{\bw}>t)=\frac{2}{\pi t}
  +\frac2\pi\left\{2w(1-w)-\frac13\right\}t^{-3}
  +O(t^{-5}).
\end{equation}
The cases $w=0$ and $w=1$ reduce directly to a standard Half-Cauchy
variable.  Also,
\[
  C_W(sx,sy)=\max\{s(x+y)-1,0\}=0
\]
for all sufficiently small $s$, proving the exact axial lower-MDA limit.
The countermonotone copula is singular, so Level III fails.

\subsection{Gaussian and standard multivariate-\texorpdfstring{$t$}{t}
copulas}

\paragraph*{One-sided Gaussian copula}
\cref{thm:gaussian} already proves Level I for every fixed
$|\rho|<1$.  In the bivariate case, the sharper equal-threshold asymptotic
is
\begin{equation}\label{eq:appendix-gaussian-equal-tail}
  C_\rho(s,s)
  \asymp
  s^{2/(1+\rho)}
  \{\log(1/s)\}^{-\rho/(1+\rho)}.
\end{equation}
For $0<\rho<1$, the power $2/(1+\rho)$ lies strictly between one and two,
so the Level-II rectangle diagnostic fails.  The same display contradicts
the $O(s^2)$ implication of global bounded density.

For $-1<\rho<0$, the power in
\cref{eq:appendix-gaussian-equal-tail} exceeds two.  Thus the
lower--lower rectangle error is $o\{s\log(1/s)\}$ and is compatible with
Level II.  It does not complete the half-space proof, because the Gaussian
copula density is singular at opposite corners.  Indeed, writing
$z_u=\Phi^{-1}(u)$, the density is
\begin{equation}\label{eq:appendix-gaussian-density}
  c_\rho(u,v)
  =\frac1{\sqrt{1-\rho^2}}
   \exp\!\left[
     \frac{2\rho z_uz_v-\rho^2(z_u^2+z_v^2)}
          {2(1-\rho^2)}
   \right].
\end{equation}
For $\rho>0$, take $u=v\downarrow0$; for $\rho<0$, take
$u\downarrow0$ and $v\uparrow1$ with $z_v=-z_u$.  The exponent tends to
$+\infty$ in each case, proving that the global density is unbounded.  At
$\rho=0$, \cref{eq:appendix-gaussian-density} equals one and the copula is
the product copula.

For $0<\rho<1$, the aggregate rate in \cref{tab:copula-hierarchy} is not
deduced from the equal-threshold rectangle alone.  The required unequal-tail
approximation and integrable face functional identify the candidate
coefficient.  Truncation control and the higher-face remainder are explicit
assumptions of \cref{thm:gaussian-second-order-m}, not proved consequences
of the rectangle limit.

\paragraph*{Two-sided Gaussian copula}
For $p_i=2\{1-\Phi(|Z_i|)\}$, the joint small-$p$ event is the union of four
sign quadrants.  A quadrant whose signed Gaussian correlation is
$|\rho|$ gives a lower bound of order
\[
  s^{2/(1+|\rho|)}
  \{\log(1/s)\}^{-|\rho|/(1+|\rho|)},
\]
while the four-quadrant union bound gives the same power as an upper bound.
Thus Level I remains axial for $|\rho|<1$, but the Level-II rectangle rate
fails whenever $\rho\ne0$.  The induced two-sided $p$-value copula cannot
have a globally bounded density, since boundedness would force every
lower-corner square probability to be $O(s^2)$.
The rate and coefficient recorded in \cref{tab:copula-hierarchy} are
conditional on the two-sided face-transfer assumptions of
\cref{thm:gaussian-second-order-m}; they do not follow from this rectangle
diagnostic alone.

\paragraph*{Standard multivariate-$t$ copula}
\cref{thm:t-copula} proves index-one MRV of the reciprocal and
Half-Cauchy score vectors.  Its spectral measure generally has interior
mass.  In the bivariate case,
\[
  \frac{C_{t,\rho,\nu}(s,s)}s
  \longrightarrow
  \lambda_t
  =2T_{\nu+1}\!\left(
    -\sqrt{\frac{(\nu+1)(1-\rho)}{1+\rho}}
  \right)>0.
\]
If the copula density were globally bounded, the numerator would be
$O(s^2)$, a contradiction.  Hence Level III fails.  Ordinary MRV supplies
no numerical second-order rate.  The model-specific calculation in
\cref{thm:t-second-order-formal} supplies the missing mapping and
half-space argument for $\nu>1$.  For $\nu>2$, both $R_{\bw}$ and
$T_{\bw}$ have a radial correction of order $t^{-1-2/\nu}$.  For
one-sided $R_{\bw}$, an angular order-$t^{-2}$ term ties the radial term at
$\nu=2$ and dominates it for $1<\nu<2$; two-sided $R_{\bw}$ and either
orientation of $T_{\bw}$ retain the radial order.  The cases
$0<\nu\le1$ remain open because of angular criticality.

\subsection{Rotations and completion of the table audit}

If an absolutely continuous copula is reflected in one or both
coordinates, its rotated density is obtained from $c_C$ by replacing
$u$ with $1-u$ and/or $v$ with $1-v$.  Thus a global density bound is
preserved exactly under rotations.  Lower-tail behavior is not preserved:
the reflection permutes the lower--lower, lower--upper, upper--lower, and
upper--upper corners.  Consequently, the lower-MDA and Level-II
classification of a rotated tail-dependent or corner-singular family must
be recomputed at the corner mapped to small $p$-values.  This proves the
rotation statement following \cref{tab:copula-hierarchy} and completes
the classification audit.
}

\CopulaClassificationAppendix

\LiteratureAuditAppendix
\section{Proofs of the main results}
\label{app:proofs}

The proofs are collected in the order in which their statements appear in
the main text.

\subsection{Heavily right scores and exponent-measure geometry}

\phantomsection\label{proof:prop:equiv}
\begin{proof}[Proof of \cref{prop:equiv}]
The expansion at zero is
\[
  \cot\!\left(\frac{\pi u}{2}\right)
  =\frac{2}{\pi u}-\frac{\pi u}{6}+O(u^3).
\]
Thus $d(u)\to0$ as $u\downarrow0$.  If $x=\pi u/2$, then
\[
 d'(u)=\frac{\pi}{2}\left\{\csc^2x-x^{-2}\right\}>0,
\]
because $\sin x<x$ for $0<x\leq\pi/2$.  Since the continuous extension
satisfies $d(0)=0$ and $d(1)=c$, one has $0\leq d\leq c$.
Applying the scalar bound coordinatewise proves
\cref{eq:bounded-difference}.  Moreover,
\[
  \frac{\|\boldsymbol H-c\boldsymbol Y\|_\infty}{t}
  \leq\frac{c}{t}\longrightarrow0
\]
deterministically.  Bounded perturbations do not alter vague tail
convergence on sets bounded away from the origin.  This proves the
equivalence and the scaling relation for $\mu_H$.  Its marginal masses
follow directly from the displayed Half-Cauchy survival expansion in
\cref{sec:score-equivalence}.
\end{proof}

\phantomsection\label{proof:thm:hcct-mrv}
\begin{proof}[Proof of \cref{thm:hcct-mrv}]
For $\boldsymbol Y$, exact Pareto margins give $c_i=1$; for
$\boldsymbol H$, \cref{prop:equiv} gives $c_i=2/\pi$.
Substitution into \cref{eq:linear-tail}, together with
$\sum_iw_i=1$, gives the two constants.  The radial measure has no atoms,
so $\mu(\partial A_{\boldsymbol w})=0$ and \cref{def:mrv} applies to the
half-space.  Under independence, both score vectors are MRV with axial
spectral measures, and the same calculation gives the two reference tails.
\end{proof}

\subsection{Lower max-domains for p-value copulas}

\phantomsection\label{proof:thm:copula-mrv}
\begin{proof}[Proof of \cref{thm:copula-mrv}]
For $s>0$ small, write
\[
 L_s(\boldsymbol x)
 =s^{-1}\Pp\!\left(\bigcup_i\{U_i\leq sx_i\}\right).
\]
Exact uniformity and a symmetric-difference bound give
\[
 |L_s(\boldsymbol x)-L_s(\boldsymbol y)|
 \leq\sum_i|x_i-y_i|.
\]
Hence the pointwise limit $\ell_L$ is Lipschitz continuous.  It is also
homogeneous of order one.  By \cref{eq:min-limit},
$\exp\{-\ell_L(\boldsymbol x)\}$ is the continuous pointwise limit of the
joint survival functions of normalized componentwise minima.  After taking
reciprocals,
\[
 G(\boldsymbol y)=\exp\{-\ell_L(1/y_1,\ldots,1/y_m)\}
\]
is a proper, nondegenerate unit-Fr\'echet distribution; the exact marginal
bounds give properness, and homogeneity of $\ell_L$ gives max-stability.
The standard exponent-measure representation of a max-stable law therefore
supplies a unique Radon measure $\mu_Y$ on $\mathbb E_m$; see, for example,
Resnick~\cite{Resnick2008}.  For $\boldsymbol x\geq0$, define
\[
  B_{\boldsymbol x}
  =\left\{\boldsymbol y\in\mathbb E_m:
       y_i>x_i^{-1}\ \text{for at least one }i\right\},
\]
where a coordinate with $x_i=0$ contributes no event.  The representation
satisfies $\mu_Y(B_{\boldsymbol x})=\ell_L(\boldsymbol x)$, and
\begin{align*}
  t\,\Pp(\boldsymbol Y/t\in B_{\boldsymbol x})
  &=t\,\Pp\!\left(
      U_i<\frac{x_i}{t}\ \text{for at least one }i
    \right)\\
  &\longrightarrow\ell_L(\boldsymbol x).
\end{align*}
The sets $B_{\boldsymbol x}$ form a convergence-determining class for
boundedly finite measures on $\mathbb E_m$ and are $\mu_Y$-continuity sets;
the latter also follows directly from the diffuse radial component of an
exponent measure.  Their convergence therefore gives the vague convergence
in \cref{def:mrv}.
\cref{prop:equiv} transfers the convergence to $\boldsymbol H$.
\end{proof}

\phantomsection\label{proof:prop:pairwise}
\begin{proof}[Proof of \cref{prop:pairwise}]
The union bound supplies the upper bound in \cref{eq:ell-sum}.
Bonferroni's lower bound differs from the sum of marginal probabilities by
at most the finite sum of pairwise intersections.  Exact uniformity and
\cref{eq:pair-comparable} give \cref{eq:ell-sum}.  Invoking
\cref{thm:copula-mrv} now proves MRV.  Finally,
\cref{eq:spectral-moment-constraints} and the representation of
$\ell_L$ uniquely identify the spectral measure: the identity
$\ell_L(\boldsymbol x)=\sum_i x_i$ is precisely the unit mass constraint
carried separately by the coordinate axes.  \cref{prop:equiv} transfers
the same support geometry to $\boldsymbol H$.
\end{proof}

\phantomsection\label{proof:cor:lower-mda-validity}
\begin{proof}[Proof of \cref{cor:lower-mda-validity}]
Apply \cref{thm:copula-mrv,thm:hcct-mrv}.
\end{proof}

\subsection{Second-order calibration and quantitative tail control}

\phantomsection\label{proof:thm:calibration-inversion}
\begin{proof}[Proof of \cref{thm:calibration-inversion}]
For $q_0=c_0/\alpha$, set
\[
 x_j=q_0\left[1+\frac{d_{j,\bw}}{c_0}
                    q_0^{-\kappa}L(q_0)\right].
\]
Since $x_j/q_0\to1$, the uniform convergence theorem for slowly varying
functions gives $L(x_j)/L(q_0)\to1$.  Expanding $c_0/x_j$ and the second
term in \cref{eq:common-second-order-tail-expansion} shows that
$\overline F_{j,\bw}(x_j)=\alpha+
  o\{\alpha q_0^{-\kappa}L(q_0)\}$.  For a fixed $\varepsilon>0$, instead
set
\[
 x_{j,\pm}=q_0\left[1+\left\{
   \frac{d_{j,\bw}}{c_0}\pm\varepsilon\right\}
   q_0^{-\kappa}L(q_0)\right].
\]
The same expansion gives
\[
 \overline F_{j,\bw}(x_{j,+})<\alpha
 <\overline F_{j,\bw}(x_{j,-})
\]
for all sufficiently small $\alpha$.  Monotonicity brackets the quantile
between $x_{j,-}$ and $x_{j,+}$; letting $\varepsilon\downarrow0$ proves
\cref{eq:second-order-quantile-expansion}, including when
$d_{j,\bw}=0$.

Subtracting the two instances of
\cref{eq:common-second-order-tail-expansion} and evaluating at
$q_{\Pi,\bw}(\alpha)\sim q_0(\alpha)$ gives
\cref{eq:second-order-size-distortion}.  Division by $\alpha$ gives
\cref{eq:second-order-relative-size-distortion}, and subtraction of the
two quantile expansions gives
\cref{eq:critical-value-calibration-shift}.
\end{proof}

\phantomsection\label{proof:thm:independent-sharp}
\begin{proof}[Proof of \cref{thm:independent-sharp}]
For two coordinates, conditioning on $Y_1$ reduces the cooperative part
to
\[
 \frac{w_1}{t-w_2}
 +w_1w_2\int_{w_1}^{t-w_2}\frac{dz}{z^2(t-z)}.
\]
Partial fractions give \cref{eq:independent-pareto-exact}.  The logarithm
is not a marginal remainder: both Pareto margins are exact.  After scaling
by $t$, the hidden product density is proportional to
$t^{-2}x^{-2}y^{-2}$.  Near, say, $y=0$, the cooperative band cut out by
the weighted-sum boundary has $x$-width proportional to $y$.  Its
truncated face integral is therefore
\[
 t^{-2}\int_{1/t}^1y\,y^{-2}\,dy
 =t^{-2}\log t.
\]
The other axis supplies the second copy.  Equivalently, expanding
\cref{eq:independent-pareto-exact} yields
$2w_1w_2(\log t)/t^2$.  For general fixed $m$, use the exact M\"obius
decomposition in \cref{lem:face-decomposition}.  The singleton increments
sum to $t^{-1}$.  Applying the bivariate identity to each pair increment
gives the displayed logarithmic coefficient, while
\cref{lem:product-corner-face-envelope} bounds every fixed face of
dimension at least three by $O(t^{-2})$.  Summing the finitely many face
increments proves \cref{eq:independent-pareto-m}.

Finally, \cref{lem:bounded} gives a fixed $K$ such that
\[
 \Pp(cR_{\boldsymbol w}>t+K)
 \leq \Pp(T_{\boldsymbol w}>t)
 \leq \Pp(cR_{\boldsymbol w}>t-K).
\]
A fixed threshold shift changes the leading $c/t$ term by $O(t^{-2})$
and leaves the $(\log t)/t^2$ coefficient unchanged, proving
\cref{eq:independent-hc-m}.
\end{proof}

\phantomsection\label{proof:prop:shell-probability}
\begin{proof}[Proof of \cref{prop:shell-probability}]
If $m=1$, the conclusion is the exact one-dimensional marginal expansion,
so suppose $m\geq2$.
Taking $(a,b)=(t,\infty)$ by monotone approximation in
\cref{eq:shell-probability-bound} gives
\[
  \Pp(Z_i>t,Z_j>t)=O(t^{-2}).
\]
Thus fixed-dimensional Bonferroni bounds make the union of individual
big-jump events $c_X/t+O(t^{-2})$.

On the event
\[
  \left\{\sum_iZ_i>t,\ \max_iZ_i\leq t\right\},
\]
choose a largest coordinate $Z_i=x>t/m$.  Some $j\ne i$ then satisfies
$Z_j>(t-x)/(m-1)$.  Partition $[t/m,t]$ into intervals, apply
\cref{eq:shell-probability-bound} on each interval with the smallest
corresponding lower bound for $Z_j$, and let the mesh tend to zero.  This
gives
\begin{align*}
  \Pp\!\left(\sum_iZ_i>t,\ \max_iZ_i\leq t\right)
  &\leq C_m\int_{t/m}^{t}
  \frac{dx}{(1+x)^2(1+t-x)}\\
  &=O\!\left(\frac{\log t}{t^2}\right).
\end{align*}
Combining the two bounds proves the proposition.
\end{proof}

\phantomsection\label{proof:thm:pair-density-second-order}
\begin{proof}[Proof of \cref{thm:pair-density-second-order}]
If $m=1$, the conclusion is the exact one-dimensional marginal expansion,
so suppose $m\geq2$.
Let $A_i=\{Z_i>t\}$.  Exact Pareto and Half-Cauchy margins give,
respectively,
\[
  \Pp(A_i)
  =\frac{c_Xw_i}{t}+O(t^{-2}).
\]
The density bound gives
\[
  \Pp(A_i\cap A_j)=O(t^{-2}),
\]
so fixed-dimensional Bonferroni bounds yield
\begin{equation}\label{eq:big-union-second}
  \Pp\!\left(\bigcup_iA_i\right)
  =\frac{c_X}{t}+O(t^{-2}).
\end{equation}

It remains to bound the event
\[
  R_t=\left\{\sum_iZ_i>t,\ \max_iZ_i\leq t\right\}.
\]
On $R_t$, choose a largest coordinate $Z_i=x$.  Then $x>t/m$, and some
$j\ne i$ satisfies $Z_j>(t-x)/(m-1)$.  Hence
\begin{align*}
  \Pp(R_t)
  &\leq\sum_{i\ne j}
  \int_{t/m}^{t}
  \int_{(t-x)/(m-1)}^\infty g_{ij}(x,y)\,dy\,dx\\
  &\leq C_m
  \int_{t/m}^{t}
    \frac{dx}{(1+x)^2(1+t-x)}
  =O\!\left(\frac{\log t}{t^2}\right).
\end{align*}
Combining this bound with \cref{eq:big-union-second} proves the result.
\end{proof}

\phantomsection\label{proof:cor:bounded-copula-density}
\begin{proof}[Proof of \cref{cor:bounded-copula-density}]
Under monotone marginal transformation, the joint score density is the
copula density multiplied by the two marginal score densities.  Both a
weighted unit-Pareto density and a weighted Half-Cauchy density are bounded
by $C_i(1+x)^{-2}$.  Therefore
\cref{eq:global-copula-bound} implies
\cref{eq:score-density-bound}.
\end{proof}

\phantomsection\label{proof:thm:second-order-mrv}
\begin{proof}[Proof of \cref{thm:second-order-mrv}]
The spectral calculation in \cref{eq:linear-tail} gives
$\mu_X(A_{\boldsymbol w})=c_X$.  Dividing
\cref{eq:quant-mrv} by $t$ proves \cref{eq:quant-mrv-conclusion}.
\end{proof}

\subsection{Face calculus for higher-order aggregation}

\phantomsection\label{proof:lem:face-decomposition}
\begin{proof}[Proof of \cref{lem:face-decomposition}]
\cref{eq:face-decomposition} is the M\"obius inversion identity
on the subset lattice, applied to $g_{[m]}$.
\end{proof}

\phantomsection\label{proof:thm:hidden-face-transfer}
\begin{proof}[Proof of \cref{thm:hidden-face-transfer}]
Because $\nu_I$ is Radon, one may choose
$\varepsilon_n\downarrow0$ and $M_n\uparrow\infty$ so that
$\nu_I(\partial K_{\varepsilon_n,M_n})=0$: on each fixed compact annulus,
the restricted coordinate projections are finite measures and therefore
have at most countably many atoms.  On each such rectangle,
$G_I\mathbf1_{K_{\varepsilon_n,M_n}}$ is bounded, compactly supported, and
$\nu_I$-almost-everywhere continuous, so vague convergence and the
portmanteau theorem give convergence of its integral.  Condition
\cref{eq:hidden-face-ui} removes the truncation uniformly.  Summing the
finite identity in \cref{lem:face-decomposition}, with
\cref{eq:hidden-face-omitted-bound} for omitted faces, proves the last
claim.
\end{proof}

\phantomsection\label{proof:prop:product-power-face}
\begin{proof}[Proof of \cref{prop:product-power-face}]
Put $z_i=w_ix_i$ and $z_j=w_jx_j$.  The pair increment equals $+1$ on
\[
 0<z_i,z_j<1,\qquad z_i+z_j>1,
\]
equals $-1$ on $z_i,z_j>1$, and vanishes elsewhere.  The negative region
has finite mass.  Near $z_i=0$, the positive region confines $z_j$ to an
interval of length $O(z_i)$, leaving an $O(z_i^{-a_i})$ integrand; this is
integrable because $a_i<1$, and the symmetric argument handles the other
axis.  Thus the integral is absolutely convergent.  Successive integration
using Euler's beta identity, continued at $a_i+a_j=1$, gives
\cref{eq:product-power-functional}.
\end{proof}

\phantomsection\label{proof:lem:product-corner-face-envelope}
\begin{proof}[Proof of \cref{lem:product-corner-face-envelope}]
If $\Gamma_k(\boldsymbol z)\ne0$, then $\sum_jz_j>1$, so at least one
coordinate is at least $1/k$.  Partition the integration region according
to the coordinates exceeding $1/k$.

On a region having at least two such coordinates, integrate those
coordinates over $[1/k,\infty)$ and use
$|\Gamma_k|\le2^k$.  Their inverse-square integrals are finite, while each
of the remaining $k-2$ coordinates contributes at most
$\int_\varepsilon^\infty z^{-2}\,dz=\varepsilon^{-1}$.  This part is
therefore $O\{\varepsilon^{-(k-2)}\}$.

It remains to consider a region on which exactly one coordinate, say
$z_r$, is at least $1/k$.  All subset sums not containing $r$ are then
strictly below one.  With $h(s)=\mathbf1\{s>1\}$,
\[
 \Gamma_k(\boldsymbol z)=\prod_{j\ne r}\Delta_{z_j}h(z_r),
 \qquad
 \Delta_af(s)=f(s+a)-f(s).
\]
Commutativity of the difference operators, together with
\[
 \|\Delta_af\|_{L^1(\R)}\le2\|f\|_{L^1(\R)},
 \qquad
 \|\Delta_ah\|_{L^1(\R)}=a,
\]
gives
\[
 \int_\R|\Gamma_k(\boldsymbol z)|\,dz_r
 \le2^{k-2}\min_{j\ne r}z_j.
\]
Since $z_r^{-2}$ is bounded on $[1/k,\infty)$, the remaining integral is
bounded by a constant multiple of
\[
 \int_{[\varepsilon,1/k]^{k-1}}
 \min_jz_j\prod_jz_j^{-2}\,d\boldsymbol z.
\]
Writing $n=k-1$ and partitioning according to the smallest coordinate,
this is bounded by
\[
 n\int_\varepsilon^{1/k}r^{-1}
 \left(\int_r^{1/k}z^{-2}\,dz\right)^{n-1}dr.
\]
It is $O\{\log(1/\varepsilon)\}$ when $n=1$ and
$O\{\varepsilon^{-(n-1)}\}=O\{\varepsilon^{-(k-2)}\}$ when $n\ge2$.
Combining the regions proves \cref{eq:product-corner-face-envelope}.
\end{proof}

\phantomsection\label{proof:thm:sharp-local-corner}
\begin{proof}[Proof of \cref{thm:sharp-local-corner}]
Apply the exact face decomposition in
\cref{lem:face-decomposition}.  The singleton terms are exact:
\[
 \sum_i\Pp(w_iY_i>t)=\frac{\sum_iw_i}{t}=\frac1t.
\]

For a pair $i<j$, write
\[
 \gamma(z_1,z_2)
 =\mathbf1\{z_1+z_2>1\}
  -\mathbf1\{z_1>1\}
  -\mathbf1\{z_2>1\}.
\]
The change of variables
$z_1=w_iY_i/t$ and $z_2=w_jY_j/t$ gives
\begin{align}
 \E G_{ij}(Y_i/t,Y_j/t)
 ={}&\frac{w_iw_j}{t^2}
 \int_{w_i/t}^\infty\int_{w_j/t}^\infty
 \gamma(z_1,z_2)\notag\\
 &\quad\times
 c_{ij}\!\left(\frac{w_i}{tz_1},\frac{w_j}{tz_2}\right)
 z_1^{-2}z_2^{-2}\,dz_1dz_2.
 \label{eq:pair-corner-integral}
\end{align}
If the density in \cref{eq:pair-corner-integral} is replaced by one, the
exact independent calculation gives
\[
 \frac{w_iw_j}{t^2}
 \log\!\frac{(t-w_i)(t-w_j)}{w_iw_j}
 =2w_iw_j\frac{\log t}{t^2}+O(t^{-2}).
\]
Replacing one by $\lambda_{ij}$ therefore produces the proposed
logarithmic coefficient.

To control the replacement error, split the positive cooperative region
of $\gamma$ into a central part and its two edge strips.  The central part
and the negative double-exceedance region give $O(t^{-2})$.  On an edge,
say $z_2\downarrow0$ and $z_1\in(1-z_2,1]$, integration in $z_1$ supplies
a factor proportional to $z_2$.  The substitution $s=w_j/(tz_2)$ then
gives
\begin{align}
 &\left|
 \E G_{ij}(Y_i/t,Y_j/t)
 -\lambda_{ij}\E_\Pi G_{ij}(Y_i/t,Y_j/t)
 \right|\notag\\
 &\qquad\le\frac{K_{ij}}{t^2}
 \left\{1+\int_{K_{ij}/t}^{u_0}
       \frac{\omega_{ij}(s)}s\,ds\right\}.
 \label{eq:pair-corner-error}
\end{align}
The lower-strip bound controls the portion on which the second copula
argument exceeds $u_0$.  Since $\omega_{ij}(s)\to0$, the integral in
\cref{eq:pair-corner-error} is $o(\log t)$ by logarithmic Ces\`aro
averaging.  Under \cref{eq:corner-dini} it is bounded.  This proves the
required pair expansion.

Now let $I$ have cardinality $k\ge3$.  On the support of $G_I$, at least
one $w_iY_i/t$ is at least $1/k$.  Hence at least one corresponding copula
coordinate is $O(t^{-1})$, so
\cref{eq:lower-strip-all-face-density} applies for all sufficiently large
$t$.  Changing variables $z_i=w_iY_i/t$ and applying
\cref{lem:product-corner-face-envelope} gives
\[
 \E|G_I(\bY_I/t)|
 \le K_I t^{-k}
 \left(\frac{t}{\min_{i\in I}w_i}\right)^{k-2}
 =O(t^{-2}).
\]
There are only finitely many faces.  Summing the singleton, pair, and
higher-face contributions proves
\cref{eq:pareto-local-corner-expansion,eq:pareto-local-corner-sharp}.

Finally, \cref{lem:bounded} gives a deterministic constant $K$ such that
\[
 \left|\sum_iw_iH_i-c\sum_iw_iY_i\right|\le K.
\]
Sandwich the Half-Cauchy tail between the Pareto-sum tails evaluated at
$(t-K)/c$ and $(t+K)/c$.  Substitution in
\cref{eq:pareto-local-corner-expansion} gives
\cref{eq:hc-local-corner-expansion}; under \cref{eq:corner-dini}, a
bounded threshold displacement changes the tail by only $O(t^{-2})$.
\end{proof}

\phantomsection\label{proof:cor:critical-corner-calibration}
\begin{proof}[Proof of \cref{cor:critical-corner-calibration}]
Under independence, $\lambda_{ij}=1$.  Subtract
\cref{eq:independent-pareto-m,eq:independent-hc-m} from the corresponding
expansions in
\cref{eq:pareto-local-corner-expansion,eq:hc-local-corner-expansion}, and apply
\cref{thm:calibration-inversion} with $\kappa=1$ and $L(t)=\log t$.
\end{proof}

\phantomsection\label{proof:thm:root-inversion}
\begin{proof}[Proof of \cref{thm:root-inversion}]
By \cref{eq:common-threshold-event} and independence,
\[
 \Pp\{S(Q,\Theta)>t\}=\E\{\overline F_Q(q_t)\}
   +o(t^{-1-\delta}).
\]
Insert \cref{eq:common-radial-tail}.  Condition
\cref{eq:common-inverse-L1} supplies the leading term and every
score-map term at exponent $\delta$.  Conditions
\cref{eq:common-radial-L1,eq:common-remainder-L1} show that the
radial term contributes
$-ab\E(A^{1+\delta})t^{-1-\delta}$ when $\rho=\delta$ and is
$o(t^{-1-\delta})$ when $\rho>\delta$.  This is
\cref{eq:common-factor-result,eq:common-factor-coefficient}.
\end{proof}

\phantomsection\label{proof:cor:common-reciprocal-phase}
\begin{proof}[Proof of \cref{cor:common-reciprocal-phase}]
In \cref{eq:common-stat-expansion}, take the relative coefficients
$B/A$ at exponent $\kappa$ and $C/A$ at exponent one, and apply
\cref{thm:root-inversion}.  Under
\cref{eq:common-reciprocal-zero-C}, omit the identically zero exponent-one
coefficient and apply the same theorem with the sole score exponent
$\kappa$.
\end{proof}

\phantomsection\label{proof:cor:common-hc-phase}
\begin{proof}[Proof of \cref{cor:common-hc-phase}]
The relative score coefficient at exponent $\kappa$ is $B/A$ and that at
exponent two is $-dD/(cA)$.  Apply
\cref{thm:root-inversion} with leading angular term $cA$.
\end{proof}

\subsection{The second-order landscape inside index-one MRV}

\phantomsection\label{proof:prop:exact-margin-axis-second}
\begin{proof}[Proof of \cref{prop:exact-margin-axis-second}]
For $tx\ge1$, exact Pareto margins give
\[
 t\Pp(Y_i>tx)=x^{-1}
 =\mu\{\boldsymbol y:y_i>x\}.
\]
Substitution in \cref{eq:exact-margin-second-mrv} proves
\cref{eq:second-measure-zero-margins}.  If $\nu_2$ is axis-supported,
the upper tails in \cref{eq:second-measure-zero-margins} determine its
restriction to each axis and make every such restriction zero.  The
Half-Cauchy assertion follows by expanding
$c\arctan\{(tx)^{-1}\}$.
\end{proof}

\phantomsection\label{proof:prop:product-power-trichotomy}
\begin{proof}[Proof of \cref{prop:product-power-trichotomy}]
Fix $x_2=y$ small.  In the positive cooperative part of $G_{12}$,
$x_1$ lies between $(1-w_2y)/w_1$ and $1/w_1$, an interval of width
proportional to $y$ on which $x_1^{-a-1}$ is bounded above and below.
After integrating in $x_1$, the boundary contribution is therefore
comparable to
\[
 \int_\varepsilon^1y\,y^{-a-1}\,dy
 =\int_\varepsilon^1y^{-a}\,dy.
\]
This is bounded, logarithmic, or polynomial according as $a<1$, $a=1$,
or $a>1$.  Symmetry handles the other axis, and all remaining regions are
integrable for $a>0$.
\end{proof}

\phantomsection\label{proof:prop:interior-no-hidden-subset}
\begin{proof}[Proof of \cref{prop:interior-no-hidden-subset}]
Along direction $\boldsymbol s$, all inequalities $rs_i>1$, $i\in I$,
hold exactly when $r>1/\min_{i\in I}s_i$.  Integrating $r^{-2}dr$ gives
\cref{eq:interior-subset-mass}.  The integrand is strictly positive
$S$-almost everywhere, so its integral is positive.  MRV then proves
\cref{eq:interior-subset-tail}.
\end{proof}

\phantomsection\label{proof:lem:common-radial-index-one}
\begin{proof}[Proof of \cref{lem:common-radial-index-one}]
Substitute $r=q^{1/\alpha}$ in
\cref{eq:raw-common-radial-alpha}.
\end{proof}

\phantomsection\label{proof:lem:three-edge-critical-integral}
\begin{proof}[Proof of \cref{lem:three-edge-critical-integral}]
For an edge $e=\{i,j\}$, let $k$ be the omitted vertex and write
\[
 b_e=w_i+w_j=1-w_k,
 \qquad
 d_e=w_k.
\]
Ties among the edge shocks have probability zero.  Conditional on
$R_e=x$ being the largest edge shock, the two coordinates incident to
$e$ equal $x$, while the remaining coordinate equals
$N_e=\max(R_f:f\ne e)$.  Consequently,
$G_{\bw}=b_ex+d_eN_e$.  Since $G_{\bw}\le x$ on this event, an exceedance
requires $x>t$.  Write $h_e(x)=(t-b_ex)/d_e$ and extend $F_R$ by zero
below its support.  The contribution from the event that $R_e$ is largest
is exactly
\begin{align}
 P_e(t)
 &=\int_t^{t/b_e}
     f_R(x)\{F_R(x)^2-F_R(h_e(x))^2\}\,dx\notag\\
 &\quad+
   \int_{t/b_e}^{\infty}f_R(x)F_R(x)^2\,dx.
 \label{eq:pair-shock-largest-edge-integral}
\end{align}
The second integral is $ab_e/t+O(t^{-2})$.  In the first integral,
change variables from $x$ to $h=(t-b_ex)/d_e$.  The assumptions give
\[
 1-F_R(h)^2=\frac{2a}{h}+O(h^{-2})
\]
for large $h$.  Splitting the $h$-integral at a fixed constant, and using
the density expansion uniformly when both arguments are large, yields
\[
 \frac{d_e}{b_e}\int_0^t
 f_R\!\left(\frac{t-d_eh}{b_e}\right)
 \{1-F_R(h)^2\}\,dh
 =2a^2b_ed_e\frac{\log t}{t^2}+O(t^{-2}).
\]
The term involving
$1-F_R\{(t-d_eh)/b_e\}^2$ is $O(t^{-2})$.  Therefore
\[
 P_e(t)=\frac{ab_e}{t}
       +2a^2b_ed_e\frac{\log t}{t^2}+O(t^{-2}).
\]
Finally,
\[
 \sum_e b_e=2,
 \qquad
 \sum_e b_ed_e
 =\sum_{k=1}^3w_k(1-w_k)
 =2\sigma_2(\bw),
\]
which proves \cref{eq:three-edge-critical-expansion}.
\end{proof}

\phantomsection\label{proof:thm:pair-shock-second-order}
\begin{proof}[Proof of \cref{thm:pair-shock-second-order}]
For $e\in\{12,13,23\}$ define
\[
 W_e=1-\exp(-2E_e),
 \qquad
 Q_e=W_e^{-1}.
\]
Then
\[
 U_1=\min(W_{12},W_{13}),\quad
 U_2=\min(W_{12},W_{23}),\quad
 U_3=\min(W_{13},W_{23}),
\]
and hence
\[
 Y_1=\max(Q_{12},Q_{13}),\quad
 Y_2=\max(Q_{12},Q_{23}),\quad
 Y_3=\max(Q_{13},Q_{23}).
\]
The edge-shock tail and density satisfy
\begin{equation}\label{eq:pair-shock-q-tail}
 \Pp(Q_e>x)
 =1-\sqrt{1-x^{-1}}
 =\frac{1}{2x}+\frac{1}{8x^2}+O(x^{-3}),
 \qquad
 f_Q(x)=\frac{1}{2x^2}+O(x^{-3}).
\end{equation}
At first order only one edge shock can be large.  Such a shock produces
the direction $(\boldsymbol e_i+\boldsymbol e_j)/2$, and
$t\Pp(2Q_e>tr)\to r^{-1}$.  Events containing two large edge shocks have
order $t^{-2}$.  This proves \cref{eq:pair-shock-spectral-sharp}.
Applying \cref{lem:three-edge-critical-integral} with $a=1/2$ proves
\cref{eq:pair-shock-pareto-second-order}.

Similarly, put $K_e=\cot(\pi W_e/2)$.  Monotonicity of the cotangent
gives
\[
 H_1=\max(K_{12},K_{13}),\quad
 H_2=\max(K_{12},K_{23}),\quad
 H_3=\max(K_{13},K_{23}).
\]
Since
\begin{equation}\label{eq:pair-shock-k-tail}
 \Pp(K_e>x)
 =1-\sqrt{1-\frac2\pi\arctan(x^{-1})}
 =\frac{c}{2x}+\frac{c^2}{8x^2}+O(x^{-3}),
 \qquad
 f_K(x)=\frac{c}{2x^2}+O(x^{-3}),
\end{equation}
another application of \cref{lem:three-edge-critical-integral}, now with
$a=c/2$, proves \cref{eq:pair-shock-hc-second-order}.
\end{proof}

\phantomsection\label{proof:cor:pair-shock-calibration}
\begin{proof}[Proof of \cref{cor:pair-shock-calibration}]
Subtract \cref{eq:independent-pareto-m,eq:independent-hc-m} from the
corresponding expansions in
\cref{eq:pair-shock-pareto-second-order,eq:pair-shock-hc-second-order}.
The independent quantile satisfies
$q_{\Pi,S,\boldsymbol w}(\alpha)\sim c_S/\alpha$, so substitution into
\cref{eq:pair-shock-independent-difference} gives
\cref{eq:pair-shock-size-distortion}.  Direct substitution in
the corresponding pair-shock expansion gives
\cref{eq:pair-shock-first-order-threshold}.
\end{proof}

\subsection{Gaussian copulas}

\phantomsection\label{proof:thm:gaussian}
\begin{proof}[Proof of \cref{thm:gaussian}]
Fix $i\ne j$ and $x,y>0$.  Let
\[
  a_s=\Phi^{-1}(1-sx),
  \qquad
  b_s=\Phi^{-1}(1-sy).
\]
Both thresholds satisfy
\[
  a_s\sim b_s\sim\sqrt{2\log(1/s)}.
\]
The simultaneous small-$p$ event is
\[
  \{U_i\leq sx,U_j\leq sy\}
  =\{Z_i\geq a_s,Z_j\geq b_s\}.
\]
On this event, $Z_i+Z_j\geq a_s+b_s$.  Since
\[
  \operatorname{var}(Z_i+Z_j)=2(1+\rho_{ij})
  \leq2(1+|\rho_{ij}|),
\]
the univariate Gaussian tail bound gives
\begin{align}
  \Pp(U_i\leq sx,U_j\leq sy)
  &\leq
  \overline\Phi\!\left(
    \frac{a_s+b_s}{\sqrt{2(1+|\rho_{ij}|)}}
  \right)\notag\\
  &=s^{\,2/(1+|\rho_{ij}|)+o(1)}=o(s),\label{eq:gauss-pair}
\end{align}
because $2/(1+|\rho_{ij}|)>1$.  Fixed $m$ makes the conclusion uniform
over the finitely many pairs.
\cref{prop:pairwise,prop:equiv} proves the result.
\end{proof}

\phantomsection\label{proof:thm:gaussian-second-order-m}
\begin{proof}[Proof of \cref{thm:gaussian-second-order-m}]
For every $(i,j)\in E_*$, the bivariate normal tail expansion, uniformly
when the two thresholds differ by $O(1)$ on the log-probability scale, gives
\[
 \Pp(Y_i>tx_i,Y_j>tx_j)
 \sim K_{\rho_*}t^{-q_{\rho_*}}
       (\log t)^{-\beta_{\rho_*}}(x_ix_j)^{-a_{\rho_*}}.
\]
Thus the pair hidden measure has the product-power density in
\cref{eq:product-power-measure}.  \cref{prop:product-power-face} gives its
weighted shell integral as
$K_{\rho_*}C(a_{\rho_*})(w_iw_j)^{a_{\rho_*}}$.  All other pair
increments are covered directly by
\cref{eq:gaussian-face-remainder-condition}; hence only pairs in $E_*$
remain at scale $r_*$.  Given that a
maximal pair is simultaneously at its Gaussian extreme level, the
conditional regression coefficient of a third coordinate is
\[
 \frac{\rho_{ki}+\rho_{kj}}{1+\rho_*}
 \le \frac{2\rho_*}{1+\rho_*}<1.
\]
Gaussian concentration, with strictness supplied by positive
definiteness, makes comparable triple extremes smaller than the dominant
pair scale.  It does not, however, control scale-degenerate regions near
proper subfaces.  Those regions are exactly what the stated
uniform-integrability and face-remainder conditions cover.
\cref{thm:hidden-face-transfer,lem:face-decomposition} now prove
\cref{eq:gaussian-m-pareto}; no passage from a comparable-threshold bound
to a noncompact half-space integral is being assumed.

Since $q_{\rho_*}<2$, the bounded Half-Cauchy/Pareto comparison in
\cref{lem:bounded}, applied at thresholds differing by a fixed constant,
transfers the regularly varying second term and multiplies its coefficient
by $c^{q_{\rho_*}}$.  In the two-sided case the two
dominant sign quadrants contribute equally, while the marginal threshold
is halved; this changes $K_\rho$ to $2^{1-q_\rho}K_\rho$.
\end{proof}

\subsection{Common-factor copulas}

\phantomsection\label{proof:lem:t-mrv}
\begin{proof}[Proof of \cref{lem:t-mrv}]
As $r\to\infty$,
\begin{align*}
  \Pp(R_0>r)
  &=\Pp(W<\nu/r^2)\\
  &\sim
  \frac{\nu^{\nu/2}}{2^{\nu/2}\Gamma(\nu/2+1)}r^{-\nu}.
\end{align*}
Thus the common radial multiplier $R_0$ is regularly varying with index
$\nu$.  The Gaussian direction $\boldsymbol G$ has moments of every order.
The multivariate version of Breiman's lemma therefore implies that
$R_0\boldsymbol G$ is MRV with index $\nu$; see
Hult and Lindskog~\cite{HultLindskog2002} or
Basrak, Davis and Mikosch~\cite{BasrakDavisMikosch2002}.
\end{proof}

\phantomsection\label{proof:thm:t-copula}
\begin{proof}[Proof of \cref{thm:t-copula}]
For the one-sided case, \cref{eq:t-tail} gives
\begin{equation}\label{eq:y-t-asymp}
  \frac{1}{U_i}
  \sim\frac{(Z_i^+)^\nu}{k_\nu}
  \qquad\text{along positive extremes},
\end{equation}
where $z^+=\max(z,0)$.  Define
\[
  g_i(\boldsymbol z)=\frac{(z_i^+)^\nu}{k_\nu}.
\]
The map $\boldsymbol g$ is continuous and homogeneous of degree $\nu$:
\[
  \boldsymbol g(a\boldsymbol z)=a^\nu\boldsymbol g(\boldsymbol z),
  \qquad a>0.
\]
Applying the mapping theorem to the index-$\nu$ MRV convergence from
\cref{lem:t-mrv}, with the radial scaling raised to the power $\nu$,
shows that $\boldsymbol g(\boldsymbol Z)$ is MRV with index one.  The
qualification ``along positive extremes'' is essential: if $z\leq0$, then
$1\leq\{1-T_\nu(z)\}^{-1}\leq2$.  Thus negative directions contribute
only a bounded score and disappear under score normalization.  More
precisely,
\[
 r^{-\nu}\{1-T_\nu(rz)\}^{-1}
 \longrightarrow \frac{(z^+)^\nu}{k_\nu},
\]
locally uniformly for $z$ in compact subsets of $\mathbb R$.  The
asymptotically homogeneous
mapping theorem therefore transfers the same tail measure to
$(1/U_1,\ldots,1/U_m)$.  \cref{prop:equiv} then
transfers MRV to the Half-Cauchy scores.

For two-sided $p$-values,
\[
  \frac{1}{U_i}\sim\frac{|Z_i|^\nu}{2k_\nu}.
\]
The continuous degree-$\nu$ map
$\boldsymbol z\mapsto(|z_1|^\nu,\ldots,|z_m|^\nu)/(2k_\nu)$ gives the same
result.
\end{proof}

\phantomsection\label{proof:thm:t-second-order-formal}
\begin{proof}[Proof of \cref{thm:t-second-order-formal}]
Besides \cref{eq:t-radial-second-order}, direct expansion of the univariate
$t$ tail gives
\[
 \overline T_\nu(z)=k_\nu z^{-\nu}
 \{1-d_\nu z^{-2}+O(z^{-4})\}.
\]
Put $Q=R_0^\nu$.  \cref{eq:t-radial-second-order} becomes
\[
 \Pp(Q>q)=a_\nu q^{-1}
 \{1-b_\nu q^{-2/\nu}+O(q^{-4/\nu})\}.
\]
On a fixed Gaussian direction, substitution of $Z_i=R_0G_i$ gives the
reciprocal expansion
\[
 S_Y(q,\boldsymbol G)
 =qA_s+q^{1-2/\nu}B_s+C_s+o(q^{1-2/\nu}+1),
\]
where $C_s=C_-$ in the one-sided case and $C_s=0$ in the two-sided case.
Gaussian truncation verifies
\cref{eq:common-inverse-L1,eq:common-remainder-L1}; near a
coordinate hyperplane the worst integrand is $|G_i|^{\nu-2}$, which is
locally integrable exactly when $\nu>1$.

For the one-sided statistic, apply
\cref{cor:common-reciprocal-phase} with the bounded term $C_-$.  For the
two-sided statistic, no exponent-one angular term is present.  On every
Gaussian angular truncation its sharper expansion is
\[
 S_Y(q,\boldsymbol G)
 =qA_{|\cdot|}+q^{1-2/\nu}B_{|\cdot|}
   +O(q^{1-4/\nu}),
\]
and the same Gaussian bounds remove the truncation at
$\delta_0=2/\nu$.  The zero-$C$ refinement
\cref{eq:common-reciprocal-zero-C-result} therefore applies even when
$2/\nu>1$.  These two applications, followed when necessary by one
further analytic term, yield
\cref{eq:t-one-sided-pareto-formal,eq:t-two-sided-pareto-formal}.  In
particular, the radial correction
dominates the bounded angular term for $\nu>2$, the terms meet at $\nu=2$,
and the angular term dominates for $1<\nu<2$.

For the positive Half-Cauchy aggregate $T_{\bw}$, a positive diverging
score satisfies
$\cot(\pi U/2)=cU^{-1}+O(U)$, whereas a negative one-sided direction tends
to zero.  Hence the common-scale expansion has leading angular term
$cA_s$, radial correction $cB_s$, and no $C_-$ term.  The first omitted
Half-Cauchy term is of tail order $t^{-3}$; because
$1+2/\nu<3$ for $\nu>1$, \cref{cor:common-hc-phase} proves
\cref{eq:t-hc-formal}.
\end{proof}

\phantomsection\label{proof:prop:t-coefficient-uniform}
\begin{proof}[Proof of \cref{prop:t-coefficient-uniform}]
Ignore the fixed constants in the definitions of $A_s$ and $B_s$, and put
$x_i=G_i^+$ or $x_i=|G_i|$ as appropriate.  At $\nu=2$, the one-sided
zero-power factor in $B_+$ is $\mathbf1\{G_i>0\}$.  Thus, with
$S_2=\sum_iw_ix_i^2$ and
$S_0=\sum_iw_i\mathbf1\{G_i>0\}\leq1$, one has
$S_2S_0\leq S_2$; in the absolute-value case the corresponding factor is
one almost surely.  Also $S_2^2\leq\sum_iw_ix_i^4$ by Jensen's inequality.
For $\nu>2$, weighted power means give the two pointwise inequalities
\begin{align*}
 \left(\sum_iw_ix_i^\nu\right)^{2/\nu}
 \left(\sum_iw_ix_i^{\nu-2}\right)
 &\leq\sum_iw_ix_i^\nu,\\
 \left(\sum_iw_ix_i^\nu\right)^{1+2/\nu}
 &\leq\sum_iw_ix_i^{\nu+2}.
\end{align*}
Their expectations are bounded by univariate Gaussian moments, uniformly
in the dimension, correlation matrix, and simplex weights.

For $1<\nu<2$,
\[
 \left(\sum_iw_ix_i^\nu\right)^{2/\nu}
 \leq\sum_iw_ix_i^2.
\]
After expanding the resulting double sum, Gaussian regression gives,
uniformly in $|\rho|\le1$,
\[
 \E\{|G_i|^2|G_j|^{\nu-2}\}
 =(1-\rho_{ij}^2)\E|N|^{\nu-2}
   +\rho_{ij}^2\E|N|^\nu.
\]
This proves the asserted bounds for $D_{\nu,\boldsymbol w}^{s}$.  Finally,
$0\leq C_-\leq1$ gives
$\E(A_+C_-)\leq\E A_+=k_\nu^{-1}\E(N^+)^\nu$ and completes the proof.
\end{proof}

\phantomsection\label{proof:thm:clayton-second-order}
\begin{proof}[Proof of \cref{thm:clayton-second-order}]
For fixed positive exponential directions,
\[
 \sum_iw_iY_i
 =RA_{\boldsymbol w}+R^{1-\theta}B_{\boldsymbol w}
   +o(R^{1-\theta}).
\]
Truncation at $E_i\in[\varepsilon,M]$, followed by exponential moment
bounds near zero and infinity, verifies
\cref{eq:common-inverse-L1,eq:common-remainder-L1} for every
$\theta>0$.  The zero-$C$ refinement
\cref{eq:common-reciprocal-zero-C-result}, with
$Q=R$ and $\rho=\kappa=\theta$, gives
\cref{eq:clayton-pareto-second}; exact marginal uniformity gives
$a_\theta\E A_{\boldsymbol w}=1$.

For an exponential variable, integration by parts gives
$\E f'(E)=\E f(E)-f(0)$.  Apply this to
$f=A_{\boldsymbol w}^{\theta+1}$ in every coordinate to obtain
\[
 \E(A_{\boldsymbol w}^{\theta}B_{\boldsymbol w})
 =\frac1{\theta+1}\left\{
 m\E A_{\boldsymbol w}^{\theta+1}
 -\sum_i\E(A_{\boldsymbol w}-w_iE_i^p)^{\theta+1}\right\}.
\]
Together with $b_\theta=1/(1+\theta)$ this proves
\cref{eq:clayton-D-alternative}.  Finally, if
$x_i=w_iE_i^p$ and $A=\sum_i x_i$, then
$(1-x_i/A)^{\theta+1}\le1-x_i/A$.  Summing proves nonnegativity, with
strict inequality whenever two $x_i$ are positive.
\end{proof}

\phantomsection\label{proof:thm:clayton-hc-phase}
\begin{proof}[Proof of \cref{thm:clayton-hc-phase}]
Insert \cref{eq:clayton-Y-representation} into
\cref{eq:hc-third-expansion}.  The common-factor statistic has a Clayton
correction at score-scale exponent $\kappa=\theta$ and a transform
correction $-dR^{-1}Q_{\boldsymbol w}$ at exponent $\kappa=2$.
For $1<\theta<2$, and for the cases $\theta\geq2$,
\cref{cor:common-hc-phase} gives the asserted result and adds the tied
terms at $\theta=2$.  The required inverse exponential moment exists in
these regimes.  Indeed, weighted Jensen and expansion of the two weighted
sums bound
\[
 \E(A_{\boldsymbol w}^2Q_{\boldsymbol w})
 \le \max\{\Gamma(1+p),
       \Gamma(1+2p)\Gamma(1-p)\}<\infty,
 \qquad p=1/\theta<1.
\]

It remains to justify (i) without this inverse moment.  If
$0<\theta<1$, \cref{lem:bounded} and the simplex weights give
\[
 cR_{\boldsymbol w}-c\leq T_{\boldsymbol w}
 \leq cR_{\boldsymbol w}.
\]
The fixed threshold displacement changes the leading $c/t$ tail by
$O(t^{-2})=o(t^{-1-\theta})$.  Applying
\cref{eq:clayton-pareto-second} at $t/c$ therefore transfers its
$c^{1+\theta}D_{\theta,\boldsymbol w}$ coefficient to
$T_{\boldsymbol w}$.

At $\theta=1$, use the sharper scalar bound
\begin{equation}\label{eq:hc-reciprocal-inverse-bound}
 0\leq cy-\cot\!\left(\frac{\pi}{2y}\right)\leq \frac{C}{y},
 \qquad y\geq1,
\end{equation}
which follows from \cref{eq:hc-third-expansion} and continuity on compact
intervals.  Here $Y_i=1+RE_i$, $R=V^{-1}$ with $V$ standard exponential,
and $A=\sum_iw_iE_i$.  Put
$D_E(r)=\sum_iw_i/(1+rE_i)$.  Conditional on the exponentials, the
reciprocal threshold is $r_t=(t/c-1)/A$.  On the discrepancy event
$\{cR_{\boldsymbol w}>t\geq T_{\boldsymbol w}\}$,
\cref{eq:hc-reciprocal-inverse-bound} and monotonicity of $D_E$ imply
\[
 0<R-r_t\leq \frac{C D_E(r_t)}{cA}.
\]
Since the density of $R$ is
$e^{-1/r}r^{-2}\leq r^{-2}$, the probability lost when replacing
$cR_{\boldsymbol w}$ by $T_{\boldsymbol w}$ is at most a constant times
\[
 \E\left\{
   \frac{1}{Ar_t^2}
   \sum_i\frac{w_i}{1+r_tE_i}
 \right\}
 =O\!\left(
   t^{-3}\sum_iw_i
   \E\frac{A^2}{E_i+A/t}
 \right)
 =O(t^{-3}\log t).
\]
To justify the final bound, fix an active $i$ and write
$B_i=\sum_{j\ne i}w_jE_j$, so $A=w_iE_i+B_i$.  Conditional on $B_i=b$,
for all sufficiently large $t$,
\[
 \frac{A^2}{E_i+A/t}
 \leq C_{\boldsymbol w}\left\{E_i+
       \frac{b^2}{E_i+b/t}\right\}.
\]
For $0<b\leq t$,
\[
 \int_0^\infty\frac{e^{-x}}{x+b/t}\,dx
 \leq C\{1+\log(t/b)\},
\]
whereas for $b>t$ the integral is at most $t/b$.  Exponential tails and
$\E[B_i^2\{1+|\log B_i|\}]<\infty$ therefore give the required
$O(\log t)$ expectation.  If only one weight is active, $B_i=0$ and the
expectation is $O(1)$.  Thus the probability loss is $o(t^{-2})$, so the
$\theta=1$ reciprocal coefficient also transfers, proving (i) throughout
$0<\theta<2$.

The one-coordinate identity follows from gamma moments or directly from
the Half-Cauchy tail.
\end{proof}

\phantomsection\label{proof:prop:clayton-coefficient-uniform}
\begin{proof}[Proof of \cref{prop:clayton-coefficient-uniform}]
The lower bound was proved above.  The integration-by-parts identity and
$A^r-(A-x)^r\le rA^{r-1}x$ give
$\E(A^\theta B)\le\E A^{\theta+1}$.  Weighted Jensen then gives
\[
 \E A^{\theta+1}
 \le\E E_1^{(\theta+1)/\theta}=\Gamma(2+1/\theta).
\]
Retaining the negative radial term gives
\[
 D_{\theta,\boldsymbol w}
 \leq a_\theta(1-b_\theta)\E A^{\theta+1}
 \leq a_\theta\frac{\theta}{1+\theta}
       \Gamma(2+1/\theta)=1.
\]
The final claim is the estimate in the proof of
\cref{thm:clayton-hc-phase}.
\end{proof}

\subsection{Growing-dimensional results}

\phantomsection\label{proof:thm:uniform-mrv}
\begin{proof}[Proof of \cref{thm:uniform-mrv}]
For every $n$, the spectral representation and
\cref{eq:uniform-margins} give
\[
  \mu_n(A_n)
  =\sum_{i=1}^{m_n}w_{i,n}
    \mu_n\{\boldsymbol x:x_i>1\}
  =c_X.
\]
Substitution into \cref{eq:uniform-mrv-assumption} proves
\cref{eq:uniform-mrv-conclusion}.
\end{proof}

\phantomsection\label{proof:thm:bl-moving-halfspace}
\begin{proof}[Proof of \cref{thm:bl-moving-halfspace}]
Sandwich the half-space indicator between linear ramps of width $\eta$ in
$\boldsymbol w^{\mathsf T}\boldsymbol x$.  Because simplex weights have
Euclidean norm at most one, the ramps have Lipschitz constant at most
$\eta^{-1}$ and are supported away from the origin.  Homogeneity and the
marginal normalization give
$\mu_n\{\boldsymbol w^{\mathsf T}\boldsymbol x>s\}=c/s$.
The lower ramp loses at most
\[
 c-\frac{c}{1+\eta}=\frac{c\eta}{1+\eta},
\]
whereas the upper ramp gains at most
\[
 \frac{c}{1-\eta}-c=\frac{c\eta}{1-\eta}.
\]
Both one-sided errors are bounded by $c\eta/(1-\eta)$, which proves
\cref{eq:bl-halfspace-bound}.  The larger mass of the full two-sided shell
is not needed.
\end{proof}

\phantomsection\label{proof:thm:radial-angular-uniform}
\begin{proof}[Proof of \cref{thm:radial-angular-uniform}]
Condition on the angle and substitute the radial tail formula at
$t/L_{n,\boldsymbol w}$.  Since $0\le L_{n,\boldsymbol w}\le1$ and
$t/L_{n,\boldsymbol w}\geq t$ on $\{L_{n,\boldsymbol w}>0\}$, the assumed
radial expansion is valid simultaneously for all angles once $t$ is
sufficiently large.  Since
$C_n\E L_{n,\boldsymbol w}=c$, the leading term is exactly $c$ and the
remainder is \cref{eq:radial-angular-error}.  The conditional version is
identical.
\end{proof}

\phantomsection\label{proof:lem:uniform-breiman-one}
\begin{proof}[Proof of \cref{lem:uniform-breiman-one}]
Put $h(q)=q\Pp(Q>q)$, so $h(q)\to a$.  Fix $q_0$ large and split at
$S_n=t_n/q_0$.  On $S_n\le t_n/q_0$,
\[
 t_n\Pp(Q>t_n/S_n\mid S_n)
 =S_n h(t_n/S_n),
\]
with the expression defined as zero when $S_n=0$.  The difference from
$aS_n$ is at most
$\sup_{q\ge q_0}|h(q)-a|S_n$.  On the complementary event, Markov's
inequality and the uniform $(1+\delta)$ moment bound give
\[
 t_n\Pp(S_n>t_n/q_0)=O(t_n^{-\delta}),
 \qquad
 \E\{S_n\mathbf1(S_n>t_n/q_0)\}=O(t_n^{-\delta}).
\]
First let $n\to\infty$ and then $q_0\to\infty$.
\end{proof}

\phantomsection\label{proof:thm:t-growing-arbitrary}
\begin{proof}[Proof of \cref{thm:t-growing-arbitrary}]
Write $\boldsymbol Z_n=R_0\boldsymbol G_n$ as in
\cref{eq:t-representation}.  Scalar regular variation of the exact
$t_\nu$ score gives, for every $\varepsilon>0$, upper and lower envelopes
of the form
\[
 (1-\varepsilon)R_0^\nu S_n-C_\varepsilon
 \le\sum_iw_{i,n}Y_{i,n}
 \le(1+\varepsilon)R_0^\nu S_n+C_\varepsilon,
\]
where $S_n$ is $A_+$ or $A_{|\cdot|}$ from
\cref{eq:t-Aplus,eq:t-AabsBabs}.  The constants do not depend on
$m_n$, because the weights sum to one.  Also
$a_\nu\E S_n=1$ and, for some fixed $\delta>0$,
\[
 \sup_n\E S_n^{1+\delta}<\infty
\]
by weighted Jensen and the univariate Gaussian moments.  Applying
\cref{lem:uniform-breiman-one} to the fixed regularly varying factor
$R_0^\nu$, and then letting the envelope error tend to zero, proves
\cref{eq:t-growing-Y}.  The deterministic uniform bound
between the weighted Half-Cauchy and scaled reciprocal sums proves
\cref{eq:t-growing-H}.
\end{proof}

\phantomsection\label{proof:thm:clayton-growing-arbitrary}
\begin{proof}[Proof of \cref{thm:clayton-growing-arbitrary}]
In the frailty representation \cref{eq:clayton-Y-representation}, put
$S_n=\sum_iw_{i,n}E_i^{1/\theta}$.  For $1/\theta\le1$, subadditivity
gives
\[
 R S_n\le\sum_iw_{i,n}Y_i\le RS_n+1.
\]
For $p=1/\theta>1$, use the elementary convexity bound
\[
 (a+b)^p
 \leq(1+\eta)^{p-1}a^p
 +(1+\eta^{-1})^{p-1}b^p,
 \qquad a,b\geq0.
\]
Choosing $\eta$ so that $(1+\eta)^{p-1}=1+\varepsilon$ gives the same
sandwich with $(1+\varepsilon)RS_n+C_{\varepsilon,\theta}$ on the upper
side.  Furthermore,
$a_\theta\E S_n=1$ and weighted Jensen yields
\[
 \sup_n\E S_n^{1+\delta}
 \le\E E_1^{(1+\delta)/\theta}<\infty
\]
for every fixed $\delta>0$.  Uniform Breiman convergence for the fixed
index-one factor $R$, now justified by
\cref{lem:uniform-breiman-one}, proves the reciprocal result.  The weighted
Half-Cauchy/Pareto difference is bounded independently of $m_n$, proving
the positive Half-Cauchy result.
\end{proof}

\end{document}